\documentclass[11pt,twoside]{amsart}

\usepackage{amssymb}
\usepackage{amsthm}
\usepackage{amsmath}
\usepackage{hyperref}
\usepackage{array}
\usepackage{ifthen}
\usepackage{mathtools}

\usepackage{etoolbox}
\let\bbordermatrix\bordermatrix
\patchcmd{\bbordermatrix}{8.75}{4.75}{}{}
\patchcmd{\bbordermatrix}{\left(}{\left[}{}{}
\patchcmd{\bbordermatrix}{\right)}{\right]}{}{}

\usepackage{graphicx}
\usepackage{tikz}
\usetikzlibrary{arrows}
\usetikzlibrary{shapes,snakes}

\usepackage{ tikz }
\usetikzlibrary{ calc }

\theoremstyle{plain}
\newtheorem{theorem}{Theorem}
\newtheorem{lemma}[theorem]{Lemma}
\newtheorem{corollary}[theorem]{Corollary}
\newtheorem{proposition}[theorem]{Proposition}

\theoremstyle{definition}

\newtheorem{example}[theorem]{Example}
\newtheorem{conjecture}[theorem]{Conjecture}

\newtheorem{problem}[theorem]{Problem}

\theoremstyle{remark}

\newcommand{\set}[1]{\left\{#1\right\}}

\def\tn{\textnormal}
\def\ld{\lambda}

\def\mN{\mathbb{N}}
\def\mR{\mathbb{R}}
\def\mS{\mathbb{S}}
\def\mZ{\mathbb{Z}}

\def\K{\mathcal{K}}
\def\M{\mathcal{M}}

\def\X{\mathcal{X}}

\def\a{\alpha}
\def\b{\beta}

\def\ce{\coloneqq}

\newcommand{\ignore}[1]{}

\DeclareMathOperator{\LS}{LS}

\DeclareMathOperator{\rank}{rank}

\DeclareMathOperator{\FRAC}{FRAC}

\DeclareMathOperator{\STAB}{STAB}

\DeclareMathOperator{\diag}{diag}
\DeclareMathOperator{\Diag}{Diag}
\DeclareMathOperator{\conv}{conv}
\DeclareMathOperator{\cone}{cone}
\DeclareMathOperator{\supp}{supp}

\title[A Computational Search for Minimal Obstruction Graphs for $\LS_+$]{\bf A Computational Search for Minimal Obstruction Graphs for the Lov\'{a}sz--Schrijver SDP Hierarchy}

\author{Yu Hin (Gary) Au}
\thanks{Yu Hin (Gary) Au: Corresponding author. Department of Mathematics and Statistics, University of Saskatchewan, Saskatoon, Saskatchewan, S7N 5E6 Canada. E-mail: gary.au@usask.ca}

\author{Levent Tun\c{c}el}
\thanks{Levent Tun\c{c}el: Research of this author was supported in part by an NSERC Discovery Grant. Department of Combinatorics and Optimization, Faculty of Mathematics, University of Waterloo, Waterloo, Ontario, N2L 3G1 Canada. E-mail: levent.tuncel@uwaterloo.ca}

\date{\today}


\begin{document}

\maketitle 

\begin{abstract}
We study the lift-and-project relaxations of the stable set polytope of graphs generated by $\LS_+$, the SDP lift-and-project operator devised by Lov\'{a}sz and Schrijver. Our focus is on $\ell$-minimal graphs: graphs on $3\ell$ vertices with $\LS_+$-rank $\ell$, i.e., the smallest graphs realizing rank $\ell$. This manuscript makes two complementary contributions. First, we introduce \emph{$\LS_+$ certificate packages}, a modular framework for certifying membership in $\LS_+$-relaxations using only integer arithmetic and simple, concise calculations, thereby making numerical lower-bound proofs more transparent, reliable, and easier to verify. Second, we apply this framework to a computational search for extremal graphs. We prove that there are at least 49 non-isomorphic 3-minimal graphs and at least 4,107 non-isomorphic 4-minimal graphs, improving the previously known counts of 14 and 588, respectively. Beyond the increase in counts, the new examples sharpen the emerging structural picture: stretched cliques remain central but are not exhaustive, clique number is informative but not decisive, and some extremal graphs exhibit previously unseen graph minor and edge density behaviour. We also determine the smallest vertex-transitive graphs of $\LS_+$-rank $\ell$ for every $\ell \leq 4$.
\end{abstract}

\section{Introduction}\label{sec01}

The $\LS_+$ operator introduced by Lov\'asz and Schrijver~\cite{LovaszS91} is one of the most widely studied lift-and-project procedures for strengthening tractable relaxations of 0,1 optimization problems. In the setting of the stable set problem, it yields, for each graph $G$, a hierarchy of progressively tighter relaxations between the fractional stable set polytope $\FRAC(G)$ and the stable set polytope $\STAB(G)$. This gives rise to the associated notion of $\LS_+$-rank, denoted $r_+(G)$, which measures how many iterations $\LS_+$ requires to recover $\STAB(G)$. Recent work of the authors~\cite{AuT25} shows that, for every positive integer $\ell$, the smallest graph of $\LS_+$-rank $\ell$ has exactly $3\ell$ vertices. The next step is therefore structural and computational: determine which graphs attain this extremal bound, understand what they look like, and develop reliable methods for proving rank bounds in concrete instances.

This manuscript addresses those questions from two complementary angles. On the methodological side, we develop a reproducible framework for presenting and verifying $\LS_+$ certificates. On the structural side, we use that framework in a computational search for extremal obstruction graphs. The outcome is a substantial expansion of the known landscape for $\ell$-minimal graphs, especially for $\ell=3$ and $\ell=4$, together with several structural phenomena that were not visible in the previously known examples.

\subsection{Main contributions and changes in the landscape}

The contributions of this manuscript fall into two complementary parts.

\begin{itemize}
\item \textbf{A reproducible certification framework.} We introduce \emph{$\LS_+$ certificate packages}, a modular way to certify membership in $\LS_+$-relaxations using integer data only. The same philosophy is then extended to $\LS_+^2$ and $\LS_+^3$. This separates the conceptual structure of the proof from floating-point SDP output and turns certificate verification into a finite collection of transparent checks over whole-number arithmetic.

\item \textbf{A computational search with structural consequences.} Using these certificate packages together with analytical rank bounds and edge-subgraph arguments, we enlarge the known families of 3-minimal and 4-minimal graphs by a wide margin and isolate several new structural patterns. In particular, the rank-4 search is genuinely large-scale: even after substantial structural reductions, it still required solving roughly 16,000 SDP instances and more than 2,000 CPU hours, starting from a candidate pool of over two million graph--facet pairs.
\end{itemize}

More specifically, the manuscript establishes the following.

\begin{itemize}
\item We prove that there are at least 49 non-isomorphic 3-minimal graphs (Theorem~\ref{thm3minimal2}), improving the previously known count of 14. We also provide strong computational evidence that this list is exhaustive.

\item We prove that there are at least 4,107 non-isomorphic 4-minimal graphs (Theorem~\ref{thm4min}), improving the previously known count of 588.

\item We identify structural features of $\ell$-minimal graphs that go beyond the raw counts. First, the stretched-clique paradigm remains a major source of extremal graphs, but it is not exhaustive. Second, low clique number explains broad classes of examples, yet it is not a necessary condition in full generality. Third, the new data sharpen several conjectural directions by showing where existing heuristics continue to work and where they fail.
\item We determine the smallest vertex-transitive graphs of $\LS_+$-rank 2, 3, and 4 (Propositions~\ref{propVT2}, \ref{propVT3}, and \ref{propVT4}), complementing the extremal picture of the irregular $\ell$-minimal graphs with a symmetry-constrained one.
\end{itemize}

Table~\ref{tab:intro-summary} summarizes the main ways in which the present manuscript changes the current picture.

\begin{table}[htbp]
\centering
\small
\begin{tabular}{|p{2.8cm}|p{6.1cm}|p{6.1cm}|}
\hline
\textbf{Topic} & \textbf{Before this work} & \textbf{This manuscript} \\
\hline
Known 3-minimal graphs & 14 non-isomorphic examples & At least 49 non-isomorphic examples, including new families both inside and outside the stretched-clique framework \\
\hline
Known 4-minimal graphs & 588 non-isomorphic examples & At least 4,107 non-isomorphic examples; many new structures and a much broader sample for formulating conjectures \\
\hline
Certification methodology & Lower-bound proofs relied on ad hoc numerical certificates & Introduces modular $\LS_+$, $\LS_+^2$, and $\LS_+^3$ certificate packages verifiable by integer arithmetic \\
\hline
Structural picture & Stretched cliques were the dominant source of examples & Reveals extremal graphs beyond stretched cliques, including examples with new clique/minor behaviour \\
\hline
Vertex-transitive extremal graphs & Smallest such graphs with $\LS_+$-rank $\ell$ not previously determined explicitly beyond $\ell = 1$ & Determines the smallest vertex-transitive graphs of ranks $2$, $3$, and $4$ \\
\hline
\end{tabular}
\caption{How the present manuscript changes the current picture of $\ell$-minimal graphs.}
\label{tab:intro-summary}
\end{table}

To place these contributions in context, we next recall the relevant definitions and prior results on $\LS_+$, stable set polytopes, and the graph operations that motivate our search.

\subsection{The $\LS_+$ operator}

Given a set $P \subseteq [0,1]^n$, let
\[
\cone(P) \ce \set{ \begin{bmatrix} \lambda \\ \lambda x \end{bmatrix} : \ld \geq 0, x \in P}
\]
denote the \emph{homogenized cone} of $P$. Then $\cone(P) \subseteq \mR^{n+1}$, and we will index the new coordinate by $0$. Next, let $[n]$ denote the set $\set{1,2,\ldots, n}$ for every $n \in \mN$, and $\mS_+^n$ denote the set of real, symmetric $n \times n$ positive semidefinite (PSD) matrices. Alternatively, we will also use the notation $M \succeq 0$ to denote that $M$ is a PSD matrix. We also let $e_i$ denote the $i$-th unit vector (which has entry $1$ at position $i$ and $0$ otherwise), and let $\diag(M)$ denote the vector formed by the diagonal entries of a square matrix $M$. Then we define
\[
\widehat{\LS}_+(P) \ce \set{ Y \in \mS_+^{n+1} : Ye_0 = \diag(Y), Ye_i, Y(e_0-e_i) \in \cone(P)~\forall i \in [n] },
\]
and
\[
\LS_+(P) \ce \set{ x \in \mR^n : \exists Y \in \widehat{\LS}_+(P), Ye_0 = \begin{bmatrix} 1 \\ x \end{bmatrix}}.
\]
Given a set $P \subseteq [0,1]^n$, let $P_I \ce \conv\set{P \cap \set{0,1}^n}$ denote the \emph{integer hull} of $P$ (i.e., $P_I$ is the convex hull of all integral points in $P$). A fundamental property of $\LS_+$ is that $P_I \subseteq \LS_+(P) \subseteq P$ for every $P \subseteq [0,1]^n$ (see, for instance, \cite[Lemma 3]{AuT24} for a simple proof). Moreover, if $P$ is tractable (i.e., polynomial-time separable up to arbitrary precision), then so is $\LS_+(P)$. Next, define $\LS_+^0(P) \ce P$, and then recursively define $\LS_+^{\ell}(P) \ce \LS_+ \left( \LS_+^{\ell-1}(P) \right)$ for every $\ell \geq 1$. Then $\LS_+^n(P) = P_I$ for every $P \subseteq [0,1]^n$ --- i.e., it always takes $\LS_+$ no more than $n$ iterations to tighten a set contained in the $n$-dimensional hypercube to its integer hull. (The reader may refer to~\cite{LovaszS91} for the proofs of these properties and further discussion about $\LS_+$.)

\subsection{The stable set problem and $\LS_+$-rank of graphs}

Given a simple, undirected graph $G \ce (V(G), E(G))$, a set of vertices $S \subseteq V(G)$ is a \emph{stable set} in $G$ if no two vertices in $S$ are joined by an edge in $G$. We then define the \emph{fractional stable set polytope} of $G$ to be
\[
\FRAC(G) \ce \set{ x \in [0,1]^{V(G)} : x_i + x_j \leq 1~\forall \set{i,j} \in E(G)},
\]
and the \emph{stable set polytope} of $G$ to be $\STAB(G) \ce \FRAC(G)_I$. Observe that a $0,1$-vector belongs to $\STAB(G)$ if and only if it is the incidence vector of a stable set in $G$. Also, for convenience, we will write $\LS_+^{\ell}(G)$ instead of $\LS_+^{\ell}(\FRAC(G))$. As discussed above, $\LS_+^{\ell}(G)$ provides successively tighter convex relaxations for $\STAB(G)$ as $\ell$ increases. This naturally leads to the notion of the \emph{$\LS_+$-rank} of $G$, which is defined to be the smallest integer $\ell$ where $\LS_+^{\ell}(G) = \STAB(G)$. For convenience, we also use the notation $r_+(G)$ to represent the $\LS_+$-rank of $G$. The notion of rank also applies to specific inequalities: Given a valid inequality $a^{\top}x \leq \b$ of $\STAB(G)$, its $\LS_+$-rank is the smallest integer $\ell$ for which it is valid for $\LS_+^{\ell}(G)$. 

It is well-known that $\FRAC(G) = \STAB(G)$ if and only if $G$ is bipartite, so these are the only graphs which have $\LS_+$-rank $0$. Many families of valid inequalities of $\STAB(G)$ are already valid after one application of $\LS_+$, including clique inequalities ($\sum_{i \in K} x_i \leq 1$ for a clique $K \subseteq V(G)$). As a result, perfect graphs have $\LS_+$-rank at most $1$~\cite{LovaszS91}. The graphs $G$ where $\LS_+(G) = \STAB(G)$ are known as \emph{$\LS_+$-perfect} graphs in the literature. For progress towards a combinatorial characterization of these graphs, see~\cite{BianchiENT13, BianchiENT17, Wagler22, BianchiENW23}.

Of course, in addition to characterizing graphs whose stable set problem is ``easy'' to solve for $\LS_+$, it is also insightful to study the graphs which serve as the worst-case instances for $\LS_+$. First, let $K_n$ denote the \emph{complete graph} on $n$ vertices (we will often use $[n]$ as the vertex labels for $K_n$). It follows from the results in~\cite{StephenT99} that the line graph of $K_{2\ell+1}$  has $\LS_+$-rank $\ell$, giving the first known family of graphs with unbounded $\LS_+$-rank. Subsequently, Lipt\'{a}k and the second author~\cite{LiptakT03} proved the following fact.

\begin{theorem}\label{thmLiptakT031}
\cite[Theorem 39]{LiptakT03}
For every graph $G$, $r_+(G) \leq \left\lfloor \frac{|V(G)|}{3} \right\rfloor$.
\end{theorem}

In other words, if we let $n_+(\ell)$ denote the minimum number of vertices among graphs with $\LS_+$-rank $\ell$, then Theorem~\ref{thmLiptakT031} implies that $n_+(\ell) \geq 3\ell$ for every $\ell \in \mN$. Thus, we say that a graph is \emph{$\ell$-minimal} if $|V(G)| = 3\ell$ and $r_+(G) = \ell$. As the lone non-bipartite graph on three vertices, $K_3$ is the unique 1-minimal graph.

Figure~\ref{figKnownEG} shows several noteworthy known $\ell$-minimal graphs. In particular, Lipt\'{a}k and the second author identified the first 2-minimal graph ($G_{\ref{figKnownEG},1}$) in 2003, while conjecturing that $\ell$-minimal graphs exist for every positive integer $\ell$. Subsequently, Escalante, Montelar, and Nasini~\cite{EscalanteMN06} showed that $G_{\ref{figKnownEG},2}$ is the only other 2-minimal graph, as well as discovered the first known 3-minimal graph $G_{\ref{figKnownEG},3}$. More recently, the authors~\cite{AuT24b} produced the first known 4-minimal graph $G_{\ref{figKnownEG},4}$.

\def\y{0.70}
\def\sc{1.7}
\def\x{180}
\def\z{360/4}

\begin{figure}[htbp]
\centering
\begin{tabular}{cccc}

\begin{tikzpicture}[scale=\sc, thick,main node/.style={circle,  minimum size=1.5mm, fill=black, inner sep=0.1mm,draw,font=\tiny\sffamily}]

\node[main node] at ({cos(\x+(0)*\z)},{sin(\x+(0)*\z)}) (1) {};
\node[main node] at ({cos(\x+(1)*\z)},{sin(\x+(1)*\z)}) (2) {};
\node[main node] at ({cos(\x+(2)*\z)},{sin(\x+(2)*\z)}) (3) {};

\node[main node] at ({ \y* cos(\x+(3)*\z) + (1-\y)*cos(\x+(2)*\z)},{ \y* sin(\x+(3)*\z) + (1-\y)*sin(\x+(2)*\z)}) (4) {};
\node[main node] at ({cos(\x+(3)*\z)},{sin(\x+(3)*\z)}) (5) {};
\node[main node] at ({ \y* cos(\x+(3)*\z) + (1-\y)*cos(\x+(4)*\z)},{ \y* sin(\x+(3)*\z) + (1-\y)*sin(\x+(4)*\z)}) (6) {};

 \path[every node/.style={font=\sffamily}]
(1) edge (2)
(2) edge (3)
(3) edge (1)
(4) edge (5)
(5) edge (6)
(4) edge (3)
(6) edge (1)
(6) edge (2);
\end{tikzpicture}
&
\begin{tikzpicture}[scale=\sc, thick,main node/.style={circle,  minimum size=1.5mm, fill=black, inner sep=0.1mm,draw,font=\tiny\sffamily}]

\node[main node] at ({cos(\x+(0)*\z)},{sin(\x+(0)*\z)}) (1) {};
\node[main node] at ({cos(\x+(1)*\z)},{sin(\x+(1)*\z)}) (2) {};
\node[main node] at ({cos(\x+(2)*\z)},{sin(\x+(2)*\z)}) (3) {};

\node[main node] at ({ \y* cos(\x+(3)*\z) + (1-\y)*cos(\x+(2)*\z)},{ \y* sin(\x+(3)*\z) + (1-\y)*sin(\x+(2)*\z)}) (4) {};
\node[main node] at ({cos(\x+(3)*\z)},{sin(\x+(3)*\z)}) (5) {};
\node[main node] at ({ \y* cos(\x+(3)*\z) + (1-\y)*cos(\x+(4)*\z)},{ \y* sin(\x+(3)*\z) + (1-\y)*sin(\x+(4)*\z)}) (6) {};

 \path[every node/.style={font=\sffamily}]
(1) edge (2)
(2) edge (3)
(3) edge (1)
(4) edge (5)
(5) edge (6)
(4) edge (2)
(4) edge (3)
(6) edge (1)
(6) edge (2);
\end{tikzpicture}

&

\def\x{270 - 360/5}
\def\z{360/5}
\begin{tikzpicture}[scale=\sc, thick,main node/.style={circle,  minimum size=1.5mm, fill=black, inner sep=0.1mm,draw,font=\tiny\sffamily}]
\node[main node] at ({cos(\x+(0)*\z)},{sin(\x+(0)*\z)}) (1) {};
\node[main node] at ({cos(\x+(1)*\z)},{sin(\x+(1)*\z)}) (2) {};
\node[main node] at ({cos(\x+(2)*\z)},{sin(\x+(2)*\z)}) (3) {};

\node[main node] at ({ \y* cos(\x+(3)*\z) + (1-\y)*cos(\x+(2)*\z)},{ \y* sin(\x+(3)*\z) + (1-\y)*sin(\x+(2)*\z)}) (4) {};
\node[main node] at ({cos(\x+(3)*\z)},{sin(\x+(3)*\z)}) (5) {};
\node[main node] at ({ \y* cos(\x+(3)*\z) + (1-\y)*cos(\x+(4)*\z)},{ \y* sin(\x+(3)*\z) + (1-\y)*sin(\x+(4)*\z)}) (6) {};

\node[main node] at ({ \y* cos(\x+(4)*\z) + (1-\y)*cos(\x+(3)*\z)},{ \y* sin(\x+(4)*\z) + (1-\y)*sin(\x+(3)*\z)}) (7) {};
\node[main node] at ({cos(\x+(4)*\z)},{sin(\x+(4)*\z)}) (8) {};
\node[main node] at ({ \y* cos(\x+(4)*\z) + (1-\y)*cos(\x+(5)*\z)},{ \y* sin(\x+(4)*\z) + (1-\y)*sin(\x+(5)*\z)}) (9) {};

 \path[every node/.style={font=\sffamily}]
(1) edge (3)
(1) edge (2)
(2) edge (3)
(4) edge (5)
(5) edge (6)
(7) edge (8)
(8) edge (9)
(4) edge (2)
(6) edge (1)
(6) edge (3)
(7) edge (1)
(7) edge (2)
(9) edge (3)
(6) edge (7);
\end{tikzpicture}

&

\def\x{270 - 360/6}
\def\z{360/6}

\begin{tikzpicture}[scale=\sc, thick,main node/.style={circle,  minimum size=1.5mm, fill=black, inner sep=0.1mm,draw,font=\tiny\sffamily}]

\node[main node] at ({cos(\x+(0)*\z)},{sin(\x+(0)*\z)}) (1) {};
\node[main node] at ({cos(\x+(1)*\z)},{sin(\x+(1)*\z)}) (2) {};
\node[main node] at ({cos(\x+(2)*\z)},{sin(\x+(2)*\z)}) (3) {};

\node[main node] at ({ \y* cos(\x+(3)*\z) + (1-\y)*cos(\x+(2)*\z)},{ \y* sin(\x+(3)*\z) + (1-\y)*sin(\x+(2)*\z)}) (4) {};
\node[main node] at ({cos(\x+(3)*\z)},{sin(\x+(3)*\z)}) (5) {};
\node[main node] at ({ \y* cos(\x+(3)*\z) + (1-\y)*cos(\x+(4)*\z)},{ \y* sin(\x+(3)*\z) + (1-\y)*sin(\x+(4)*\z)}) (6) {};

\node[main node] at ({ \y* cos(\x+(4)*\z) + (1-\y)*cos(\x+(3)*\z)},{ \y* sin(\x+(4)*\z) + (1-\y)*sin(\x+(3)*\z)}) (7) {};
\node[main node] at ({cos(\x+(4)*\z)},{sin(\x+(4)*\z)}) (8) {};
\node[main node] at ({ \y* cos(\x+(4)*\z) + (1-\y)*cos(\x+(5)*\z)},{ \y* sin(\x+(4)*\z) + (1-\y)*sin(\x+(5)*\z)}) (9) {};

\node[main node] at ({ \y* cos(\x+(5)*\z) + (1-\y)*cos(\x+(4)*\z)},{ \y* sin(\x+(5)*\z) + (1-\y)*sin(\x+(4)*\z)}) (10) {};
\node[main node] at ({cos(\x+(5)*\z)},{sin(\x+(5)*\z)}) (11) {};
\node[main node] at ({ \y* cos(\x+(5)*\z) + (1-\y)*cos(\x+(6)*\z)},{ \y* sin(\x+(5)*\z) + (1-\y)*sin(\x+(6)*\z)}) (12) {};

 \path[every node/.style={font=\sffamily}]
(1) edge (2)
(1) edge (3)
(2) edge (3)
(4) edge (5)
(5) edge (6)
(7) edge (8)
(8) edge (9)
(10) edge (11)
(11) edge (12)
(4) edge (9)
(7) edge (12)
(10) edge (6)
(1) edge (4)
(2) edge (4)
(2) edge (6)
(3) edge (6)
(2) edge (7)
(3) edge (7)
(3) edge (9)
(1) edge (9)
(3) edge (10)
(1) edge (10)
(1) edge (12)
(2) edge (12);
\end{tikzpicture}

\\
$G_{\ref{figKnownEG},1}$ & $G_{\ref{figKnownEG},2}$ & $G_{\ref{figKnownEG},3}$ & $G_{\ref{figKnownEG},4}$ 
\end{tabular}
\caption{Several known $\ell$-minimal graphs due to~\cite{LiptakT03, EscalanteMN06, AuT24b}.}
\label{figKnownEG}
\end{figure}
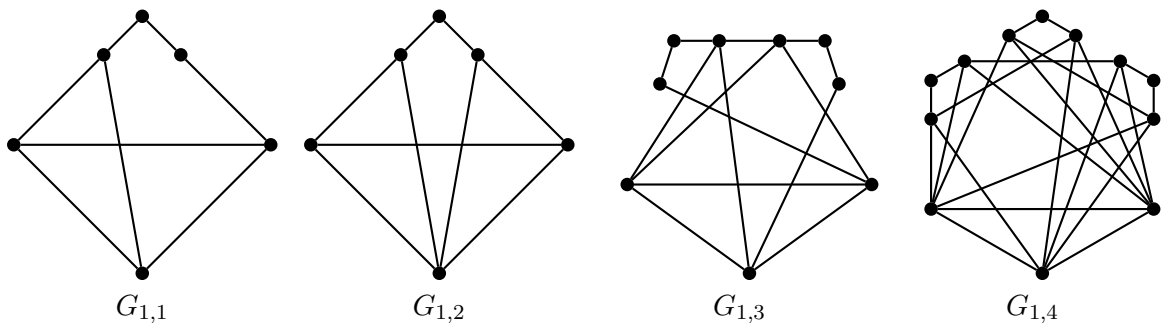

The examples in Figure~\ref{figKnownEG} already suggest vertex-stretching as a recurring mechanism for constructing $\ell$-minimal graphs. Given a graph $G$ and vertex $v \in V(G)$, define $\Gamma_G(v) \ce \set{ u : \set{u,v} \in E(G)}$ to be the \emph{(open) neighborhood} of $v$ in $G$. Then given a vertex $v \in V(G)$ and sets $A_1, \ldots,  A_k \subseteq \Gamma_G(v)$ where $\bigcup_{j=1}^k A_j  = \Gamma_G(v)$, we define the \emph{stretching} of $v$ in $G$ by applying the following sequence of transformations to $G$: 
\begin{itemize}
\item
replace $v$ by $k+1$ vertices: $v_0, v_1, \ldots, v_k$;
\item
for every $j \in [k]$, add an edge between $v_j$  and all vertices in $\set{v_0} \cup A_j$.
\end{itemize}

We say that a vertex-stretching operation is \emph{proper} if $\emptyset \neq A_j \subset \Gamma_G(v)$ for every $j \in [k]$. Also, we will call the operation \emph{$k$-stretching} when we need to specify $k$. For example, in Figure~\ref{figVertexStretch}, $G_1 = K_6$, $G_2$ is obtained from a proper 2-stretching of vertex $5$ in $G_1$, and $G_3$ is obtained from a proper 2-stretching of vertex $6$ in $G_2$. 

\def\y{0.7}
\def\sc{2}
\def\x{270 - 360/6}
\def\z{360/6}

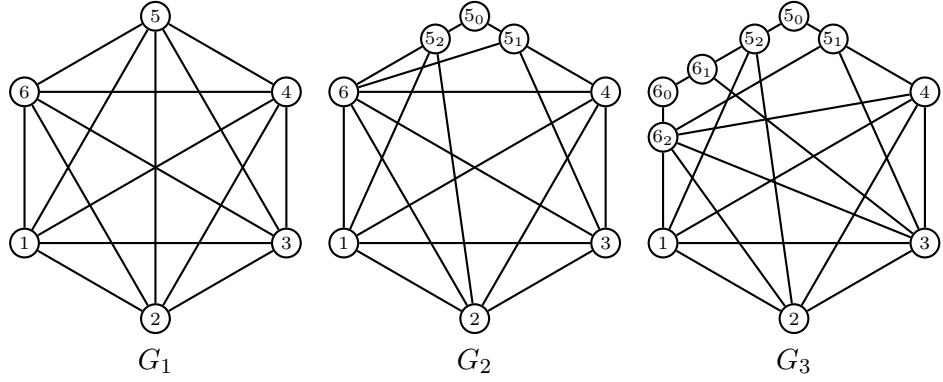
\begin{figure}[htbp]
\centering
\begin{tabular}{ccc}

\begin{tikzpicture}
[scale=\sc, thick,main node/.style={circle, minimum size=3.8mm, inner sep=0.1mm,draw,font=\tiny\sffamily}]
\node[main node] at ({cos(\x+(0)*\z)},{sin(\x+(0)*\z)}) (1) {$1$};
\node[main node] at ({cos(\x+(1)*\z)},{sin(\x+(1)*\z)}) (2) {$2$};
\node[main node] at ({cos(\x+(2)*\z)},{sin(\x+(2)*\z)}) (3) {$3$};
\node[main node] at ({cos(\x+(3)*\z)},{sin(\x+(3)*\z)}) (4) {$4$};
\node[main node] at ({cos(\x+(4)*\z)},{sin(\x+(4)*\z)}) (5) {$5$};
\node[main node] at ({cos(\x+(5)*\z)},{sin(\x+(5)*\z)}) (6) {$6$};

 \path[every node/.style={font=\sffamily}]
(1) edge (2) (1) edge (3) (2) edge (3) (1) edge (4)(2) edge (4)(3) edge (4)(1) edge (5)(2) edge (5)(3) edge (5)(4) edge (5)(1) edge (6)(2) edge (6)(3) edge (6)(4) edge (6)(5) edge (6);
\end{tikzpicture}
&
\begin{tikzpicture}
[scale=\sc, thick,main node/.style={circle, minimum size=3.8mm, inner sep=0.1mm,draw,font=\tiny\sffamily}]
\node[main node] at ({cos(\x+(0)*\z)},{sin(\x+(0)*\z)}) (1) {$1$};
\node[main node] at ({cos(\x+(1)*\z)},{sin(\x+(1)*\z)}) (2) {$2$};
\node[main node] at ({cos(\x+(2)*\z)},{sin(\x+(2)*\z)}) (3) {$3$};
\node[main node] at ({cos(\x+(3)*\z)},{sin(\x+(3)*\z)}) (4) {$4$};

\node[main node] at ({cos(\x+(4)*\z)},{sin(\x+(4)*\z)}) (5) {$5_0$};
\node[main node] at ({  cos(\x+(4)*\z) + (1-\y)*cos(330)},{ sin(\x+(4)*\z) + (1-\y)*sin(330)}) (6) {$5_1$};
\node[main node] at ({ cos(\x+(4)*\z) + (1-\y)*cos(210},{ sin(\x+(4)*\z) + (1-\y)*sin(210)}) (8) {$5_2$};

\node[main node] at ({cos(\x+(5)*\z)},{sin(\x+(5)*\z)}) (9) {$6$};

 \path[every node/.style={font=\sffamily}]
(1) edge (2)(1) edge (3)(2) edge (3)(1) edge (4)(2) edge (4)(3) edge (4)(6) edge (5)(8) edge (5)(1) edge (9)(2) edge (9)(3) edge (9)(4) edge (9)(3) edge (6)(4) edge (6)(9) edge (6)(1) edge (8)(2) edge (8)(9) edge (8);
\end{tikzpicture}
&
\begin{tikzpicture}
[scale=\sc, thick,main node/.style={circle, minimum size=3.8mm, inner sep=0.1mm,draw,font=\tiny\sffamily}]
\node[main node] at ({cos(\x+(0)*\z)},{sin(\x+(0)*\z)}) (1) {$1$};
\node[main node] at ({cos(\x+(1)*\z)},{sin(\x+(1)*\z)}) (2) {$2$};
\node[main node] at ({cos(\x+(2)*\z)},{sin(\x+(2)*\z)}) (3) {$3$};
\node[main node] at ({cos(\x+(3)*\z)},{sin(\x+(3)*\z)}) (4) {$4$};

\node[main node] at ({cos(\x+(4)*\z)},{sin(\x+(4)*\z)}) (5) {$5_0$};
\node[main node] at ({  cos(\x+(4)*\z) + (1-\y)*cos(330)},{ sin(\x+(4)*\z) + (1-\y)*sin(330)}) (6) {$5_1$};
\node[main node] at ({ cos(\x+(4)*\z) + (1-\y)*cos(210},{ sin(\x+(4)*\z) + (1-\y)*sin(210)}) (8) {$5_2$};

\node[main node] at ({cos(\x+(5)*\z)},{sin(\x+(5)*\z)}) (9) {$6_0$};
\node[main node] at ({  cos(\x+(5)*\z) + (1-\y)*cos(30)},{ sin(\x+(5)*\z) + (1-\y)*sin(30)}) (10) {$6_1$};
\node[main node] at ({ cos(\x+(5)*\z) + (1-\y)*cos(270},{  sin(\x+(5)*\z) + (1-\y)*sin(270)}) (11) {$6_2$};

 \path[every node/.style={font=\sffamily}]
(1) edge (2)(1) edge (3)(2) edge (3)(1) edge (4)(2) edge (4)(3) edge (4)(6) edge (5) (8) edge (5)(10) edge (9)(11) edge (9)(3) edge (6)(4) edge (6)(1) edge (8)(2) edge (8)(11) edge (1)(11) edge (2)(11) edge (3)(11) edge (4)(11) edge (6)(10) edge (3)(10) edge (8);
\end{tikzpicture}
\\
$G_1$ & $G_2$ & $G_3$
\end{tabular}
\caption{Illustrating the vertex-stretching operation.}
\label{figVertexStretch}
\end{figure}

An important property of the vertex-stretching operation in relation to the $\LS_+$-rank of a graph is the following.

\begin{theorem}\label{thmAuT250}
\cite[Lemma 2]{AuT25} Let $H$ be obtained from $G$ by stretching a vertex in $G$. Then $r_+(H) \geq r_+(G)$.
\end{theorem}

Theorem~\ref{thmAuT250} is noteworthy because the $\LS_+$-rank does not always behave well under the conventional graph minor operations. For instance, removing an edge from a graph could increase its $\LS_+$-rank~\cite[Figure 4]{LiptakT03}. In fact, given a general graph $G$, there are currently only two known ways to construct another graph $H$ with $r_+(H) \geq r_+(G)$: make $G$ an induced subgraph of $H$, or obtain $H$ by (possibly iterated) vertex-stretching. Although the full vertex-stretching operation was formulated only recently in~\cite{AuT25}, it is a natural generalization of several more restrictive variants that appeared earlier in the literature; see, for instance,~\cite{LiptakT03, AguileraEF14, BianchiENT17, AuT24b}.

For the present paper, the most relevant case is proper 2-stretching. The 2-minimal graphs $G_{\ref{figKnownEG},1}$ and $G_{\ref{figKnownEG},2}$ are each obtained by a proper 2-stretching of a vertex of $K_4$, so a single such operation can increase the $\LS_+$-rank from 1 to 2. Likewise, $G_{\ref{figKnownEG},3}$ is obtained from $K_5$ by properly 2-stretching two vertices, increasing the rank from 1 to 3. On the other hand, no example is currently known in which a single vertex-stretching operation raises the $\LS_+$-rank by more than 1. This makes proper 2-stretching especially attractive in the search for $\ell$-minimal graphs, since it may increase rank while adding only two vertices.

These insights motivated the authors' work in~\cite{AuT24, AuT24b}. In~\cite{AuT24b}, we studied the properties of the vertex-stretching operation, focusing on when the stretching is proper. In particular, the following result further highlights an important link between the 2-stretching operation and $\ell$-minimal graphs.

\begin{theorem}\label{thmAuT24b1}
\cite[Theorem 19]{AuT24b}
Let $G$ be an $\ell$-minimal graph where $\ell \geq 2$. Then there is a graph $H$ and a vertex $i \in V(H)$ such that
\begin{itemize}
\item[(i)]
$G$ is obtained from $H$ by a proper 2-stretching of $i$, and
\item[(ii)]
$H - i$ is an $(\ell-1)$-minimal graph.
\end{itemize}
\end{theorem}

We also showed the following when it comes to stretching vertices of a complete graph.

\begin{theorem}\label{thmAuT24b2}
\cite[Propositions 21 and 23]{AuT24b}
Let $n \geq 4$ be an integer, and let $G$ be obtained from $K_n$ by a proper 2-stretching of $i \in V(K_n)$. Then
\begin{itemize}
\item[(i)]
$r_+(G) = 2$.
\item[(ii)]
Let $H$ be a graph obtained by a proper 2-stretching of one of $i_0,i_1,i_2 \in V(G)$. Then $r_+(H) = 2$.
\end{itemize}
\end{theorem}

That is, while a proper 2-stretching of one of the original vertices $i \in V(K_n)$ is guaranteed to increase the graph's $\LS_+$-rank from 1 to 2, the rank would remain at 2 if we further stretch any of the three new vertices $i_0$, $i_1$, or $i_2$. These results suggest that a promising approach of obtaining relatively small graphs with high $\LS_+$-ranks is to 2-stretch some of the original vertices of $K_n$. Thus, given integers $n \geq 3$ and $d \geq 0$,  let $\K_{n,d}$ denote the set of graphs which can be obtained from $K_n$ by 2-stretching $d$ of its vertices.
Then notice that $G_{\ref{figKnownEG},1}, G_{\ref{figKnownEG},2} \in \K_{4,1}$, $G_{\ref{figKnownEG},3} \in \K_{5,2}$, and $G_{\ref{figKnownEG},4} \in \K_{6,3}$. Furthermore,  the authors exhibited in~\cite{AuT24} stretched-clique families with $r_+(G)=\Theta(|V(G)|)$, which is asymptotically tight by Theorem~\ref{thmLiptakT031}. These examples reinforce the idea that stretched-clique constructions are a natural source of graphs with high $\LS_+$-rank.

More recently, the authors were able to further build on these ideas and prove that $\ell$-minimal graphs indeed exist for every $\ell \in \mN$. To describe these results, we need some more notation related to graphs that are stretched cliques. Given $G \in \K_{n,d}$, we let $D(G) \subseteq V(K_n)$ be the set of original vertices from $K_n$ which were stretched to obtain $G$ (and thus $|D(G)| = d$). Furthermore, for every $i \in V(K_n)$, we define the vertices in $V(G)$ \emph{associated} with $i$ to be $i_0, i_1$, and $i_2$ if $i \in D(G)$, and simply the unstretched vertex $i$ otherwise. Every vertex in $G$ is associated with a unique original vertex in $K_n$. Also, by the definition of the vertex-stretching operation, if $G \in \K_{n,d}$ and $i,j \in [n]$ are distinct, then there is at least one edge in $G$ joining a vertex associated with $i$ to a vertex associated with $j$.

Next, define $\hat{\K}_{n,d} \subseteq \K_{n,d}$ to be the set of stretched cliques $G$ where, for every distinct $i,j \in D(G)$, there is \emph{exactly} one edge in $G$ that joins a vertex associated with $i$ with a vertex associated with $j$. For example, $G_3$ from Figure~\ref{figVertexStretch} does not belong to $\hat{\K}_{6,2}$, since $D(G_3) = \set{5,6}$ and there are two edges --- $\set{5_1, 6_2}$ and $\set{5_2,6_1}$ --- joining vertices associated with $5$ and $6$. On the other hand, each of the four known $\ell$-minimal graphs from Figure~\ref{figKnownEG} belongs to $\hat{\K}_{\ell+2,\ell-1}$.

Given a graph $G$, let $\a(G)$ denote the size of the largest stable set in $G$. We also let $\bar{e}$ denote the vector of all-ones (the dimension of which will be clear from the context). For every $G \in \K_{n,d}$, we have $\a(G) = d+1$~\cite[Lemma 4]{AuT25}, which implies that the inequality $\bar{e}^{\top}x \leq d+1$ is valid for $\STAB(G)$. Also, let $\omega(G)$ denote the size of the largest clique in $G$ (usually known as the \emph{clique number} of $G$). Then, we have the following.

\begin{theorem}\label{thmAuT251}
\cite[Theorem 19]{AuT25} Let $G \in \hat{\K}_{n,d}$ where $n \geq 3$ and $d \geq 0$, and let $k \ce \max \set{3, \omega(G)}$. Then the $\LS_+$-rank of $\bar{e}^{\top}x \leq d+1$ is at least $n-k+1$.
\end{theorem}

Given $\ell \in \mN$, observe that every $G \in \hat{\K}_{\ell+2, \ell-1}$ contains exactly $3\ell$ vertices, and Theorem~\ref{thmAuT251} assures that $r_+(G) = \ell$ as long as $\omega(G) \leq 3$. Using this, we were able to prove the following.

\begin{theorem}\label{thmAuT252}
\cite[Theorem 23]{AuT25} For every positive integer $\ell$, there are at least $2^{\ell-1}$ non-isomorphic $\ell$-minimal graphs.
\end{theorem}

While the bound in Theorem~\ref{thmAuT252} is tight for $\ell = 1$ and $\ell=2$, the number of $\ell$-minimal graphs likely far exceeds $2^{\ell-1}$ for $\ell \geq 3$. For instance, there are 13 non-isomorphic graphs in $\hat{\K}_{5,2}$ with $\omega(G) \leq 3$ (see Figure~\ref{fighatK52}), and it follows from Theorem~\ref{thmAuT251} that they are all 3-minimal. (These 13 graphs do include $G_{\ref{figKnownEG},3}$, as well as the 3-minimal graphs discovered earlier in~\cite{AuT24b}, up to isomorphism.) Moreover, there is at least one graph in $\K_{5,2} \setminus \hat{\K}_{5,2}$ which is also 3-minimal~\cite[Proposition 25]{AuT25}. Likewise, an exhaustive computational search finds that there are 588 non-isomorphic graphs in $\hat{\K}_{6,3}$ with $\omega(G) \leq 3$, which include the first known 4-minimal graphs discovered in~\cite{AuT24b}. Thus, prior to this work, there were a total of 14 known 3-minimal graphs and 588 known 4-minimal graphs.

\def\sc{0.85}
\def\y{0.7}
\def\z{360/ (5+4*\y)}
\def\x{270}
\def\placev{
\node[main node] at ({ \x+(-1 + 0*\y) *\z} : 1) (1) {};
\node[main node] at ({ \x+(0 + 0*\y) *\z} : 1) (2) {};
\node[main node] at ({ \x+(1 + 0*\y) *\z} : 1) (3) {};
\node[main node] at ({ \x+(2 + 0*\y) *\z} : 1) (41) {};
\node[main node] at ({ \x+(2 + 1*\y) *\z} : 1) (40) {};
\node[main node] at ({ \x+(2 + 2*\y) *\z} : 1) (42) {};
\node[main node] at ({ \x+(3 + 2*\y) *\z} : 1) (51) {};
\node[main node] at ({ \x+(3 + 3*\y) *\z} : 1) (50) {};
\node[main node] at ({ \x+(3 + 4*\y) *\z} : 1) (52) {};
}

\def\placecommone{  \path (1) edge (2) (2) edge (3) (3) edge (1) (40) edge (41) (40) edge (42) (50) edge (51) (50) edge (52);}

\begin{figure}[htbp]
\centering
\begin{tabular}{ccccccc}
\begin{tikzpicture}[scale=\sc, thick,main node/.style={circle,  minimum size=1.5mm, fill=black, inner sep=0.1mm,draw,font=\tiny\sffamily}] \placev \placecommone
 \path (1) edge (41) (2) edge (41) (3) edge (42)  (1) edge (51) (2) edge (51) (3) edge (52) (41) edge (52);
\end{tikzpicture}
& 
\begin{tikzpicture}[scale=\sc, thick,main node/.style={circle,  minimum size=1.5mm, fill=black, inner sep=0.1mm,draw,font=\tiny\sffamily}] \placev \placecommone
 \path (1) edge (41) (2) edge (41) (3) edge (42)  (1) edge (52) (2) edge (51) (3) edge (51) (41) edge (51);
\end{tikzpicture}
& 
\begin{tikzpicture}[scale=\sc, thick,main node/.style={circle,  minimum size=1.5mm, fill=black, inner sep=0.1mm,draw,font=\tiny\sffamily}] \placev \placecommone
 \path (1) edge (41) (2) edge (41) (3) edge (42)  (1) edge (52) (2) edge (51) (3) edge (51) (41) edge (52);
\end{tikzpicture}
& 
\begin{tikzpicture}[scale=\sc, thick,main node/.style={circle,  minimum size=1.5mm, fill=black, inner sep=0.1mm,draw,font=\tiny\sffamily}] \placev \placecommone
 \path (1) edge (41) (2) edge (41) (3) edge (42)  (1) edge (51) (2) edge (51) (3) edge (52) (42) edge (52);
\end{tikzpicture}
& 
\begin{tikzpicture}[scale=\sc, thick,main node/.style={circle,  minimum size=1.5mm, fill=black, inner sep=0.1mm,draw,font=\tiny\sffamily}] \placev \placecommone
 \path (1) edge (41) (2) edge (41) (3) edge (42)  (1) edge (52) (2) edge (51) (3) edge (51) (42) edge (52);
\end{tikzpicture}
& 
\begin{tikzpicture}[scale=\sc, thick,main node/.style={circle,  minimum size=1.5mm, fill=black, inner sep=0.1mm,draw,font=\tiny\sffamily}] \placev \placecommone
 \path (1) edge (41) (1) edge (42) (2) edge (41) (3) edge (42)  (1) edge (52) (2) edge (51) (3) edge (52) (41) edge (51);
\end{tikzpicture}
& 
\begin{tikzpicture}[scale=\sc, thick,main node/.style={circle,  minimum size=1.5mm, fill=black, inner sep=0.1mm,draw,font=\tiny\sffamily}] \placev \placecommone
 \path (1) edge (41) (1) edge (42) (2) edge (41) (3) edge (42)  (1) edge (51) (2) edge (52) (3) edge (51) (41) edge (51);
\end{tikzpicture}\\
\begin{tikzpicture}[scale=\sc, thick,main node/.style={circle,  minimum size=1.5mm, fill=black, inner sep=0.1mm,draw,font=\tiny\sffamily}] \placev \placecommone
 \path (1) edge (41) (1) edge (42) (2) edge (41) (3) edge (42)  (1) edge (51) (2) edge (52) (3) edge (51) (42) edge (52);
\end{tikzpicture}
& 
\begin{tikzpicture}[scale=\sc, thick,main node/.style={circle,  minimum size=1.5mm, fill=black, inner sep=0.1mm,draw,font=\tiny\sffamily}] \placev \placecommone
 \path (1) edge (41) (1) edge (42) (2) edge (41) (3) edge (42)  (1) edge (51) (2) edge (52) (3) edge (51) (41) edge (52);
\end{tikzpicture}
& 
\begin{tikzpicture}[scale=\sc, thick,main node/.style={circle,  minimum size=1.5mm, fill=black, inner sep=0.1mm,draw,font=\tiny\sffamily}] \placev \placecommone
 \path (1) edge (41) (1) edge (42) (2) edge (41) (3) edge (42)  (1) edge (52) (2) edge (51) (3) edge (51) (41) edge (52);
\end{tikzpicture}
& 
\begin{tikzpicture}[scale=\sc, thick,main node/.style={circle,  minimum size=1.5mm, fill=black, inner sep=0.1mm,draw,font=\tiny\sffamily}] \placev \placecommone
 \path (1) edge (41) (1) edge (42) (2) edge (41) (3) edge (42)  (1) edge (51) (1) edge (52)  (2) edge (52) (3) edge (51) (41) edge (51);
\end{tikzpicture}
& 
\begin{tikzpicture}[scale=\sc, thick,main node/.style={circle,  minimum size=1.5mm, fill=black, inner sep=0.1mm,draw,font=\tiny\sffamily}] \placev \placecommone
 \path (1) edge (41) (1) edge (42) (2) edge (41) (3) edge (42)  (1) edge (52)   (2) edge (51) (2) edge (52) (3) edge (51) (42) edge (51);
\end{tikzpicture}
& 
 \begin{tikzpicture}[scale=\sc, thick,main node/.style={circle,  minimum size=1.5mm, fill=black, inner sep=0.1mm,draw,font=\tiny\sffamily}] \placev \placecommone
 \path (1) edge (41) (1) edge (42) (2) edge (41) (3) edge (42)  (1) edge (52)   (2) edge (51) (2) edge (52) (3) edge (51) (41) edge (51);
\end{tikzpicture}
\end{tabular}
\caption{The $13$ graphs $G \in \hat{\K}_{5,2}$ with $\omega(G) \leq 3$.}
\label{fighatK52}
\end{figure}
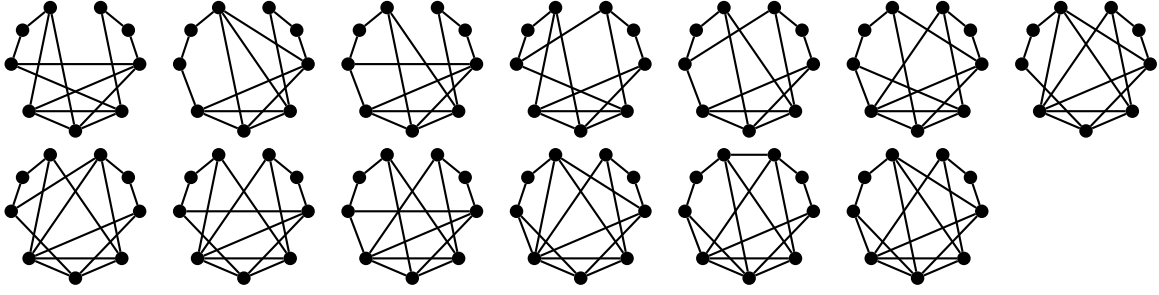

\subsection{Motivation and an outline for this manuscript}

With the extremal identity $n_+(\ell)=3\ell$ now established, the central question is no longer whether $\ell$-minimal graphs exist, but which graphs attain this bound and what this reveals about the $\LS_+$ hierarchy. Our goals are therefore twofold: to enlarge the known families of $\ell$-minimal graphs and extract structural patterns from them, and to develop a reliable certification framework for the semidefinite lower-bound arguments that arise in this setting.

First, extremal graphs are natural obstruction instances for the hierarchy, so understanding them helps clarify which structural features are genuinely responsible for high $\LS_+$-rank. In particular, new examples can test the limits of the stretched-clique paradigm, indicate which facet directions remain difficult after several rounds of $\LS_+$, and provide a controlled source of instances for future studies of related lift-and-project relaxations and polyhedral phenomena. In this regard, highly symmetric graphs offer a complementary perspective: although $\ell$-minimal graphs are necessarily irregular, vertex-transitive graphs provide a natural extremal testbed for understanding how graphs with high $\LS_+$-rank can arise under strong symmetry constraints.

The methodological contribution is equally important. For a graph $G$ on $3\ell$ vertices, the upper bound $r_+(G) \leq \ell$ follows immediately from Theorem~\ref{thmLiptakT031}; the main difficulty is proving the matching lower bound. In practice, this means showing that some point belongs to $\LS_+^{\ell-1}(G)$ while violating a valid inequality of $\STAB(G)$. Historically, such arguments have depended on layered numerical certificates, and this is precisely where reproducibility can become fragile. Our certificate-package framework addresses that issue by encoding the relevant semidefinite and cone-membership checks in integer data that can be verified independently and systematically.

The rest of the paper is organized as follows. Section~\ref{sec02} develops the tools used throughout the manuscript and introduces the certificate-package framework. Sections~\ref{sec03} and~\ref{sec04} contain the main computational results on 3-minimal and 4-minimal graphs, respectively; beyond the headline counts, these sections record the new structural phenomena revealed by the search and formulate several conjectures suggested by the data. Section~\ref{sec05} turns to vertex-transitive graphs and determines the smallest examples of ranks $2$, $3$, and $4$. Section~\ref{sec06} summarizes the resulting structural picture and lists several open problems motivated by the new examples and by the certification framework.

\section{Tools for analyzing $\LS_+$-ranks}\label{sec02}

In this section, we collect a number of tools which are useful in analyzing the $\LS_+$-rank of a graph. We first mention several known relevant results, then describe a framework for presenting numerical certificates for $\LS_+$ relaxations.

\subsection{Some useful known tools}

First, the following is a well-known property of $\LS_+$. (See, for instance,~\cite[Lemma 5]{AuT24} for a proof.)

\begin{lemma}\label{lemfacet}
Let $P \subseteq [0,1]^n$ be a polyhedron, and $F$ be a face of $[0,1]^n$. Then
\[
\LS_+^{\ell}(P \cap F) = \LS_+^{\ell}(P) \cap F,
\]
for every $\ell \in \mN$.
\end{lemma}

Next, given a graph $G$ and a set of vertices $S \subseteq V(G)$, let $G[S]$ denote the subgraph of $G$ induced by $S$. The following is an immediate consequence of Lemma~\ref{lemfacet}. (See, for instance,~\cite[Lemma 6]{AuT24} for a proof.)

\begin{lemma}\label{lemfacet1}
Let $G$ be a graph. Then $r_+(G[S]) \leq r_+(G)$ for every $S \subseteq V(G)$.
\end{lemma}

Thus, the $\LS_+$-rank of an induced subgraph of a graph cannot exceed that of the graph itself. Next, given a vector $a \in \mR^n$, let $\supp(a)$ denote the \emph{support} of $a$ (i.e., the set of indices $i$ where $a_i \neq 0$). Then another implication of~Lemma~\ref{lemfacet} is the following.

\begin{lemma}\label{lemfacet2}
For every graph $G$,
\[
r_+(G) = \max\set{ r_+(G[\supp(a)]) : \tn{$a^{\top}x \leq \b$ is a facet-inducing inequality of $\STAB(G)$}}.
\]
\end{lemma}

\begin{proof}
$(\geq)$ follows immediately from Lemma~\ref{lemfacet1}, so it only remains to prove $(\leq)$. Let $\ell \ce r_+(G)$. If $\ell =0$, then $G$ is bipartite and every induced subgraph must also have $\LS_+$-rank $0$. Thus, assume $\ell \geq 1$, and so there exists $\bar{x} \in \LS_+^{\ell-1}(G) \setminus \STAB(G)$, which means that there exists a facet-inducing inequality $a^{\top}x \leq \b$ of $\STAB(G)$ which is violated by $\bar{x}$. Then, it follows from Lemma~\ref{lemfacet} that the projection of $\bar{x}$ onto $\supp(a)$ belongs to $\LS_+^{\ell-1}(G[\supp(a)]) \setminus \STAB(G[\supp(a)])$, and thus $r_+(G[\supp(a)]) \geq \ell$.
\end{proof}

For the sake of brevity, we will slightly abuse terminology and refer to a facet-inducing inequality simply as a facet from here on. With Lemma~\ref{lemfacet2}, we see that if $\STAB(G)$ does not have a full-support facet, then there exists a proper induced subgraph of $G$ which has the same $\LS_+$-rank as $G$. This immediately implies that the stable set polytope of an $\ell$-minimal graph must have a full-support facet.

Another situation where one can conclude that $r_+(G)$ is realized by a proper subgraph of $G$ is when the graph contains a \emph{cut clique} --- a clique whose removal from $G$ results in multiple components. More precisely, we have the following.

\begin{proposition}\label{propCliqueCut}
\cite[Lemma 5]{LiptakT03}
Let $G$ be a graph, and $S_1, S_2, K \subseteq V(G)$ are mutually disjoint subsets such that
\begin{itemize}
\item
$S_1 \cup S_2 \cup K = V(G)$;
\item
$K$ induces a clique in $G$;
\item
there is no edge $\set{i,j} \in E(G)$ where $i\in S_1, j \in S_2$.
\end{itemize}
Then $r_+(G) = \max \set{ r_+( G[S_1 \cup K]), r_+(G[S_2 \cup K])}$.
\end{proposition}

Next, given a graph $G$ and $S \subseteq V(G)$, we define 
\[
G - S \ce G[V(G) \setminus S],
\]
and refer to $G-S$ as the graph obtained from $G$ by the \emph{deletion} of $S$. When $S = \set{v}$, we will simply write $G - v$ instead of $G - \set{v}$ for convenience. Given a vertex $v \in V(G)$, we also define
\[
G \ominus v \ce G - ( \set{v} \cup \Gamma_G(v)),
\]
and call $G \ominus v$ the graph obtained from $G$ by the \emph{destruction} of $v$. The following result relates the $\LS_+$-rank of $G$ to that of subgraphs of $G$ obtained via the deletion or destruction of a vertex in $G$.

\begin{theorem}\label{thmDeleteDestroy}
For every graph $G$,
\begin{itemize}
\item[(i)]
\cite[Corollary 2.16]{LovaszS91} $r_+(G) \leq \max \set{ r_+(G \ominus i) : i \in V(G) } + 1$;
\item[(ii)]
\cite[Theorem 36]{LiptakT03} $r_+(G) \leq \min \set{ r_+(G - i) : i \in V(G) } + 1$.
\end{itemize}
\end{theorem}

Finally, given graphs $G$ and $H$ where $V(H) = V(G)$ and $E(H) \subseteq E(G)$, we say that $H$ is an \emph{edge subgraph} of $G$. Then $\FRAC(G) \subseteq \FRAC(H)$ in this case. Also, it follows readily from the definition of $\LS_+$ that the operator preserves containment (i.e., if $P \subseteq P'$, then $\LS_+(P) \subseteq \LS_+(P')$). Hence, we have the following.

\begin{lemma}\label{lem05subgraph}
Let $H$ be an edge subgraph of $G$, and let $\ell$ be a nonnegative integer.
\begin{itemize}
\item[(i)]
If $a^{\top}x \leq \b$ is valid for $\LS_+^{\ell}(H)$, then $a^{\top}x \leq \b$ is valid for $\LS_+^{\ell}(G)$.
\item[(ii)]
If $a^{\top}x \leq \b$ is not valid for $\LS_+^{\ell}(G)$, then $a^{\top}x \leq \b$ is not valid for $\LS_+^{\ell}(H)$.
\end{itemize}
\end{lemma}

A useful implication of Lemma~\ref{lem05subgraph} is that, if we have a graph $G$ and a valid inequality $a^{\top}x \leq \b$ of $\STAB(G)$ with  $\LS_+$-rank $\ell$, then every edge subgraph of $G$ where $a^{\top}x\leq \b$ is valid for its stable set polytope also has $\LS_+$-rank at least $\ell$. 

\subsection{Introducing $\LS_+$ certificate packages}

We now describe our framework of presenting numerical certificates for $\LS_+$ relaxations in this manuscript. In this section, we focus on certifying the membership of a point in $\LS_+^{\ell}(P)$ for the case $\ell=1$, which will help prepare for our discussion of the cases where $\ell \geq 2$ in subsequent sections.

Given a symmetric matrix $Y \in \mZ^{n \times n}$, we say that matrices $U,V,W$ form a \emph{$UVW$-certificate} of $Y$ if
\begin{itemize}
\item
the entries of $U,V$, and $W$ are all integers;
\item
$ W^{\top} \left( U^{\top} U + V \right)W = k Y$ for some positive integer $k$;
\item
$V$ is symmetric and diagonally dominant (i.e., $\sum_{j \neq i} |V_{ij}| \leq V_{ii}$ for all $i \in [n]$).
\end{itemize}

Then we have the following elementary fact.

\begin{lemma}\label{lem:UVW}
Suppose $Y \in \mZ^{n \times n}$ is a symmetric matrix. Then $Y \succeq 0$ if and only if $Y$ has a $UVW$-certificate.
\end{lemma}

\begin{proof}
Let $Y \succeq 0$ be given, and let $d \ce \rank(Y)$. If $d=0$, then $Y$ is the matrix of all zeros, and in this case $W \ce I_n$ and $U, V$ being the $n \times n$ matrix of all zeros would do. Next, assume that $d$ is positive. Then there exists a symmetric matrix $Y' \in \mZ^{d \times d}$ that is a principal submatrix of $Y$ where $\rank(Y') = d$. Hence, we can write $Y = W_1^{\top} Y' W_1$ for some rational matrix $W_1$. Next, let $\ld$ be the smallest eigenvalue of $Y'$. Since $Y'$ is a principal submatrix of $Y$ and has full rank,  $Y'$ must be positive definite, which implies that $\ld > 0$. Then $Y' - \ld I_d \succeq 0$ and so there exists a real matrix $U_0$ where $Y' = U_0^{\top} U_0 + \ld I_d$. Moreover, $\ld I_d$ is a positive multiple of the identity matrix, and so we can let $U_1$ be a rational approximation sufficiently close to $U_0$ such that $V_1 \ce Y' - U_1^{\top}U_1$ is diagonally dominant and has rational entries. Now, we have $Y = W_1^{\top}( U_1^{\top}U_1 + V_1) W_1$. Multiplying the rational matrices $U_1, V_1, W_1$ by a suitable integer yields integral matrices $U, V, W$ with the desired properties.

Conversely, suppose $Y$ has a $UVW$-certificate. Since $U^{\top}U \succeq 0$ for every $U$ and $V \succeq 0$ (as a property of diagonally dominant matrices). Then $U^{\top}U+V \succeq 0$, which implies that $Y = \frac{1}{k} W^{\top}(U^{\top}U+V)W \succeq 0$.
\end{proof}

The presence of a $UVW$-certificate allows us to easily and reliably verify the positive semidefiniteness of a given matrix by performing only elementary arithmetic operations involving whole numbers.

Next, recall that $e_i$ denotes the $i$-th unit vector. Similarly, we let $f_i \ce e_0 - e_i$ (we will be using this notation exclusively when working in the space of $\cone(P)$ for a set $P \subseteq [0,1]^n$, so it will always be clear what the 0th coordinate is). Given a graph $G$ with $n$ vertices, we define an \emph{$\LS_+$ certificate package} to be
\begin{itemize}
\item
A matrix $Y \in \mZ^{(n+1) \times (n+1)}$ where 
\begin{itemize}
\item
$Y=Y^{\top}$ and $Ye_0 = \diag(Y)$;
\item
$Ye_i, Yf_i \in \cone(\FRAC(G))$ for every $i \in [n]$.
\end{itemize}
\item
A $UVW$-certificate for $Y$.
\end{itemize}

Notice that the presence of an $\LS_+$ certificate package asserts that $Ye_0 \in \cone(\LS_+(G))$. Also, since $\FRAC(G)$ is a rational polytope, the conditions $Ye_i, Yf_i \in \cone(\FRAC(G))$ can be verified using elementary arithmetic operations on whole numbers. $\LS_+$ certificate packages are useful for helping establish that a given graph has $\LS_+$-rank at least 2. 

\begin{proposition}\label{prop:LS_+-certificate}
Let $G$ be a graph. Then  $r_+(G) \geq 2$ if and only if there exist a valid inequality $a^{\top}x \leq \beta$ for $\STAB(G)$
and an $\LS_+$ certificate package $(Y, U, V, W)$ for $G$ such that $\left(-\beta, a^{\top}\right)Ye_0 > 0$.
\end{proposition}

\begin{proof}
Suppose $r_+(G) \geq 2$. Then there exists $\bar{x} \in \LS_+(G)\setminus \STAB(G)$ and a facet $a^{\top} x \leq \beta$
of $\STAB(G)$ such that $a^{\top}\bar{x} > \beta$. Then, by the definition of $\LS_+$ and the density of rationals, there exists a positive semidefinite matrix $\tilde{Y}$ with rational entries satisfying the second and third conditions for an $\LS_+$ certificate package and such that $\left(-\beta, a^{\top}\right)\tilde{Y}e_0 > 0$. By a suitable positive integer scaling of $\tilde{Y}$, we arrive at 
a positive semidefinite integral matrix $Y$ satisfying all conditions for the existence of a $\LS_+$ certificate package $(Y,U,V,W)$
for $G$ such that $\left(-\beta, a^{\top}\right)Ye_0 > 0$ (the existence of $(U,V,W)$ satisfying the last condition of the $\LS_+$ certificate package follows from Lemma~\ref{lem:UVW}).

Now, suppose there exist a valid inequality $a^{\top}x \leq \beta$ for $\STAB(G)$
and an $\LS_+$ certificate package $(Y, U, V, W)$ for $G$ such that $\left(-\beta, a^{\top}\right)Ye_0 > 0$.
Then, by the definition of $\LS_+$ certificate package, $Ye_0 \in \cone(\LS_+(G))$. By assumption, $Ye_0$ violates a valid inequality for $\cone(\STAB(G))$. Therefore, $\cone(\STAB(G)) \subset \cone(\LS_+(G))$ which implies $r_+(G) \geq 2$. 

\end{proof}

As our first example, consider the graph $G_{\ref{fig7claw},1}$ in Figure~\ref{fig7claw}.

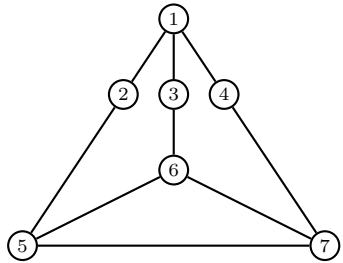
\begin{figure}[htbp]
\centering

\begin{tikzpicture}
[scale=1, thick,main node/.style={circle, minimum size=3.8mm, inner sep=0.1mm,draw,font=\tiny\sffamily}]
\node[main node] at (0,2) (1) {$1$};
\node[main node] at ({-2/3},1) (2) {$2$};
\node[main node] at (0,1) (3) {$3$};
\node[main node] at ({2/3},1) (4) {$4$};
\node[main node] at (-2,-1) (5) {$5$};
\node[main node] at (0,0) (6) {$6$};
\node[main node] at (2,-1) (7) {$7$};

  \path[every node/.style={font=\sffamily}]
(1) edge (2) (1) edge (3) (1) edge (4) (2) edge (5) (3) edge (6) (4) edge (7) (5) edge (6) (6) edge (7) (5) edge (7);
\end{tikzpicture}
\caption{$G_{\ref{fig7claw},1}$, a 7-vertex graph with $\LS_+$-rank $2$.}
\label{fig7claw}
\end{figure}

Then we have the following:

\begin{proposition}\label{prop7claw1}
The graph $G_{\ref{fig7claw},1}$ from Figure~\ref{fig7claw} has $\LS_+$-rank $2$.
\end{proposition}

\begin{proof}
For convenience, let $G \ce G_{\ref{fig7claw},1}$ throughout this proof. First, observe that $|V(G)| = 7$, and so $r_+(G) \leq 2$. Next, consider the matrices
\begin{align*}
Y &\ce
\begin{bmatrix}
76 & 25 & 40 & 40 & 40 & 20 & 20 & 20 \\ 
25 & 25 & 0 & 0 & 0 & 10 & 10 & 10 \\ 
40 & 0 & 40 & 30 & 30 & 0 & 10 & 10 \\ 
40 & 0 & 30 & 40 & 30 & 10 & 0 & 10 \\ 
40 & 0 & 30 & 30 & 40 & 10 & 10 & 0 \\ 
20 & 10 & 0 & 10 & 10 & 20 & 0 & 0 \\ 
20 & 10 & 10 & 0 & 10 & 0 & 20 & 0 \\ 
20 & 10 & 10 & 10 & 0 & 0 & 0 & 20
\end{bmatrix},~
U \ce
\begin{bmatrix}
0 & 0 & 0 & 0 & 0 & 0 & 0 \\ 
-1 & 12 & -4 & -4 & 14 & -12 & -12 \\ 
0 & 0 & -18 & 18 & 0 & -11 & 11 \\ 
10 & 16 & 18 & 18 & -40 & -21 & -21 \\ 
0 & 0 & 31 & -31 & 0 & -50 & 50 \\ 
37 & 74 & -38 & -38 & -29 & 31 & 31 \\ 
139 & 29 & 88 & 88 & 90 & 33 & 33
\end{bmatrix}
\end{align*}
\begin{align*}
V &\ce
\begin{bmatrix}
185 & -17 & 30 & 30 & 17 & -16 & -16 \\ 
-17 & 183 & 20 & 20 & 8 & -11 & -11 \\ 
30 & 20 & 227 & 37 & 34 & -44 & 12 \\ 
30 & 20 & 37 & 227 & 34 & 12 & -44 \\ 
17 & 8 & 34 & 34 & 303 & 17 & 17 \\ 
-16 & -11 & -44 & 12 & 17 & 264 & -14 \\ 
-16 & -11 & 12 & -44 & 17 & -14 & 264 
\end{bmatrix},~W \ce \begin{bmatrix}
5 & 0 & 0 & 0 & 0 & 0 & 0 & 10 \\ 
0 & 5 & 0 & 0 & 0 & 0 & 0 & -4 \\ 
0 & 0 & 5 & 0 & 0 & 0 & 0 & -3 \\ 
0 & 0 & 0 & 5 & 0 & 0 & 0 & -3 \\ 
0 & 0 & 0 & 0 & 5 & 0 & 0 & -3 \\ 
0 & 0 & 0 & 0 & 0 & 5 & 0 & -5 \\ 
0 & 0 & 0 & 0 & 0 & 0 & 5 & -5    
\end{bmatrix}.
\end{align*}
Then one can check that $6900 Y = W^{\top} ( U^{\top}U +V) W$, and the matrices indeed satisfy all conditions for an $\LS_+$ certificate package. Thus,
\[
\bar{x} \ce \frac{1}{76}(25,40,40,40,20,20,20)^{\top} \in \LS_+(G).
\]
On the other hand, $\bar{x}$ violates the inequality $(2,1,1,1,1,1,1)^{\top}x \leq 3$, which is valid for $\STAB(G)$. Thus, we conclude that $r_+(G) = 2$.
\end{proof}

The $\LS_+$ certificate package described in the proof of Proposition~\ref{prop7claw1} and other numerical data in support of the results in this manuscript are made publicly available at~\cite{AuT25data}.

Notably, $G_{\ref{fig7claw},1}$ contains a \emph{claw} (i.e., a stable set $S$ of size 3 with a vertex that is adjacent to all 3 vertices in $S$). In fact, using the characterization of claw-free graphs with $\LS_+$-rank at least 2 due to Bianchi et al.~\cite[Corollary 32]{BianchiENW23}, it follows that if $|V(H)| =7, r_+(H) =2$ and that $H$ does not contain $G_{\ref{figKnownEG},1}$ or $G_{\ref{figKnownEG},2}$ as an induced subgraph, then $H$ must contain a claw. Thus, Proposition~\ref{prop7claw1} shows that such an example indeed exists. The graph $G_{\ref{fig7claw},1}$ was also studied in~\cite{LiptakT03}, and Proposition~\ref{prop7claw1} proves that the subdivision-of-a-star operation mentioned therein can increase the $\LS_+$-rank of a graph. 

Also, with Proposition~\ref{prop7claw1}, we can prove a slight generalization of Theorem~\ref{thmAuT24b2}(i).

\begin{proposition}\label{propStretchedKk}
Let $n \geq 4$, and let $G$ be obtained from $K_n$ by a proper stretching of a vertex. Then $r_+(G) = 2$.
\end{proposition}

\begin{proof}
Suppose $G$ is obtained from $K_n$ by a proper $k$-stretching of $v \in V(K_n)$ for some integer $k \geq 2$. If $|A_i| \geq 2$ for at least one $i \in [k]$, then $G$ must contain $G_{\ref{figKnownEG},1}$ or $G_{\ref{figKnownEG},2}$ as an induced subgraph. Otherwise, $|A_i| = 1$ for all $i \in [k]$, which implies that $k \geq 3$ (since $n \geq 4$). In this case, $G$ must contain the graph $G_{\ref{fig7claw},1}$ as an induced subgraph. In either case, we have that $r_+(G) \geq 2$.

Also, notice that $G- v_0$ must be a perfect graph, and thus $r_+(G- v_0) \leq 1$, showing that $r_+(G) \leq 2$.
\end{proof}

Next, given a graph $G$, an integer $\ell \geq 1$ and a non-negative and non-zero vector $a \in \mR^{V(G)}$, define
\[
\gamma_{\ell}(G,a) \ce \frac{ \max \set{ a^{\top}x : x \in \LS_+^{\ell}(G)} }{ \max \set{ a^{\top}x : x \in \STAB(G)}}.
\]
In other words, $\gamma_{\ell}(G,a)$ is the \emph{integrality ratio} of $\LS_+^{\ell}(G)$ in the direction of the vector $a$. By imposing that $a \geq 0$ and $a \neq 0$, we ensure that $\max\set{a^{\top}x : x \in \STAB(G)} > 0$, and so $\gamma_{\ell}(G,a)$ is well-defined. 

It is apparent that if $\gamma_{\ell}(G,a) > 1$ for some integer $\ell$ and vector $a$, then $r_+(G) > \ell$. Throughout this paper, we will establish $\LS_+$-rank lower bounds of a graph using one of the following two approaches:
\begin{itemize}
\item
Provide an analytical proof for this rank lower bound (e.g., using results stated earlier in this section);
\item
Present a point $\bar{x} \not\in \STAB(G)$ with a $\LS_+^{\ell-1}$ certificate package showing that $\bar{x} \in \LS_+^{\ell-1}(G)$.
\end{itemize}

In both cases, we will  write $r_+(G) \geq \ell$. Likewise, we write $r_+(G) \leq \ell$ if there is an analytical proof for this bound. On the other hand, there are situations when a self-contained argument using the existing theoretical tools is not available. In this case, we will try to obtain a ``softer'' upper bound using CVX+SeDuMi~\cite{CVX,Sturm99}, a MATLAB-based modelling system for convex optimization. In our experience --- and especially in moderate to large size problem instances --- it is not uncommon for CVX+SeDuMi to return an integrality ratio between $1+10^{-7}$ and $1+10^{-6}$ when there is an analytical proof that the true value is $1$. Thus, given a graph $G$, we write that $r_+(G) \lesssim \ell$ if, for every facet $a^{\top}x \leq \b$ of $\STAB(G)$, either we have an analytical proof that $\gamma_{\ell}(G,a) = 1$, or 
\[
\gamma_{\ell} (G,a) \leq 1 + 10^{-6}
\]
according to CVX+SeDuMi. In such cases, it is conceivable that the true value of $\gamma_{\ell} (G,a)$ is $1$ for all facets of $\STAB(G)$, which would imply that $r_+(G) \leq \ell$. Furthermore, due to Theorem~\ref{thmLiptakT031}, $\gamma_{\ell}(G,a) = 1$ for all facets $a^{\top}x \leq \b$ of $\STAB(G)$ where $|\supp(a)| < 3\ell$. Thus, to conclude that $r_+(G) \lesssim \ell$, it suffices to compute $\max\set{ a^{\top}x : x \in \LS_+^{\ell}(G)}$ with CVX+SeDuMi only for the facets where $|\supp(a)| \geq 3\ell$.

For ease of reference, we will use the following evidence convention in the remainder of the paper. Statements of the form $r_+(G)=\ell$ or $r_+(G)\ge \ell$ are backed either by analytical arguments or by explicit certificate packages verifiable by exact integer arithmetic. Statements written as $r_+(G)\lesssim \ell$ are numerical upper-bound evidence obtained from CVX+SeDuMi and should be interpreted accordingly. Likewise, when we summarize search outcomes, phrases such as \emph{computations indicate}, \emph{computational evidence}, and \emph{heuristic} refer to conclusions that organize or motivate the search but are not themselves certified exact statements. In contrast, theorem-, proposition-, and corollary-level claims record only certified facts.

We remark that the computations for this work were mostly performed in MATLAB (R2023a) \cite{MATLAB} on a laptop computer equipped with an Intel Core i9-11950H processor (8 cores, 2.6 GHz clock speed) and 64 GB of RAM, running on the operating system Microsoft Windows 11 Education. 

\section{3-minimal graphs}\label{sec03}

In this section, we focus on studying 3-minimal graphs (i.e., graphs $G$ where $|V(G)| = 9$ and $r_+(G) = 3$). Again, at the time of this writing, the list of known 3-minimal graphs consists of the 13 in $\hat{\K}_{5,2}$ without a $K_4$ as an induced subgraph (Figure~\ref{fighatK52}), as well as one other graph in $\K_{5,2} \setminus \hat{\K}_{5,2}$~\cite[Proposition 25]{AuT25}.

Herein, we prove that there are at least 49 non-isomorphic 3-minimal graphs in total, including 18 in $\K_{5,2} \setminus \hat{\K}_{5,2}$, and another 18 graphs which are not in $\K_{5,2}$. We also record numerical patterns that may be useful in future searches for $\ell$-minimal graphs with $\ell \geq 4$.

\subsection{Certified lower bounds and new 3-minimal graphs}

To do so, we need to extend the notion of $\LS_+$ certificate packages to consider an analogous framework for verifying the membership of points in $\LS_+^2(G)$. First, we prove a simple lemma that helps explain one of the conditions in our $\LS_+^2$ certificate packages. Given a set $P \subseteq [0,1]^n$, we say that $P$ is \emph{lower-comprehensive} if, for every $x \in P$, $0 \leq y \leq x$ implies $y \in P$. Observe that $\FRAC(G)$ is lower-comprehensive for every graph $G$. It also follows readily from the definition of $\LS_+$ that if $P$ is lower-comprehensive, then so is $\LS_+(P)$. Thus, we know that $\LS_+^{\ell}(G)$ is lower-comprehensive for every graph $G$ and every non-negative integer $\ell$. Also, given $x^{(1)}, x^{(2)} \in \mR^{n+1}$, we say that $x^{(1)}$ \emph{dominates} $x^{(2)}$ if
\begin{itemize}
\item
$x^{(1)} = x^{(2)} = 0$, or
\item
$[x^{(1)}]_0 > 0$, $[x^{(2)}]_0 \geq 0$, and $[x^{(2)}]_0 \cdot x^{(1)} \geq [x^{(1)}]_0 \cdot x^{(2)}$.
\end{itemize}
Then we have the following.

\begin{lemma}\label{lemConeDominance}
Let $P \subseteq [0,1]^n$ be a lower-comprehensive set, and let $x^{(1)}, x^{(2)} \in \mR_+^{n+1}$. If $x^{(1)} \in \cone(P)$ and $x^{(1)}$ dominates $x^{(2)}$, then $x^{(2)} \in \cone(P)$.
\end{lemma}

\begin{proof}
The claim obviously holds if $x^{(1)} = x^{(2)} = 0$, so we may assume that $[x^{(1)}]_0 > 0$, $[x^{(2)}]_0 \geq 0$, and $[x^{(2)}]_0 \cdot x^{(1)} \geq [x^{(1)}]_0 \cdot x^{(2)}$. Since $x^{(1)} \in \cone(P)$, we have $[x^{(2)}]_0 \cdot x^{(1)} \in \cone(P)$ (as $\cone(P)$ is closed under non-negative scalar multiplication). Also, given that $P$ is lower-comprehensive, so is $\cone(P)$, and so it follows that  $[x^{(1)}]_0 \cdot x^{(2)} \in \cone(P)$. Using again the fact that $\cone(P)$ is closed under non-negative scalar multiplication (and that $[x^{(1)}]_0 > 0$), we conclude that $x^{(2)} \in \cone(P)$.
\end{proof}

Next, given a graph $G$ with $n$ vertices, we define an \emph{$\LS_+^2$ certificate package} to be as follows:
\begin{itemize}
\item
A set of matrices $\M_1 \ce \set{Y_{e_i}, Y_{f_i} : i \in [n]} \subseteq \mZ^{(n+1) \times (n+1)}$ such that, for every $M \in \M_1$,
\begin{itemize}
\item
$M = M^{\top}$ and $Me_0 = \diag(M)$;
\item
$Me_i, Mf_i \in \cone(\FRAC(G))$ for every $i \in [n]$.
\end{itemize}
\item
A matrix $Y \in \mZ^{(n+1) \times (n+1)}$ where
\begin{itemize}
\item
$Y = Y^{\top}$ and $Ye_0 = \diag(Y)$;
\item
for every $i \in [n]$,
\begin{itemize}
\item
$Y_{e_i}e_0$ dominates $ Ye_i$;
\item
$Y_{f_i}e_0$ dominates $Yf_i$.
\end{itemize}
\end{itemize}
\item
A $UVW$-certificate for every $M \in \M_1$ and $Y$.
\end{itemize}

Notice the conditions on the matrices in $\M_1$ certify that $Y_{e_i}e_0, Y_{f_i}e_0 \in \cone(\LS_+(G))$ for all $i \in [n]$. Next, using Lemma~\ref{lemConeDominance}, the domination conditions assure that, for every $i \in [n]$,
\begin{align*}
Y_{e_i}e_0 \in \cone(\LS_+(G)) & \Rightarrow Ye_i \in \cone(\LS_+(G)),\\
Y_{f_i}e_0 \in \cone(\LS_+(G)) & \Rightarrow Yf_i \in \cone(\LS_+(G)).
\end{align*}
Thus, together with other conditions on $Y$, $Ye_0 \in \cone(\LS_+^2(G))$. Generally, a $\LS_+^2$ certificate package for a vector in $\mR^n$ consists of $4(1+2n)$ matrices (the certificate matrices $\M_1 \cup \set{Y}$, plus a $UVW$-certificate of each of these matrices). Due to the above arguments and following a similar proof to that of Proposition~\ref{prop:LS_+-certificate}, we have the following fact.

\begin{proposition}\label{prop:LS_+^2-certificate}
Let $G$ be a graph. Then  $r_+(G) \geq 3$ if and only if there exist a valid inequality $a^{\top}x \leq \beta$ for $\STAB(G)$
and an $\LS_+^2$ certificate package $(Y, \mathcal{M}_1, \textup{ and } \textup{$UVW$-certificates})$ for $G$ such that $\left(-\beta, a^{\top}\right)Ye_0 > 0$.
\end{proposition}

We next prove that there are at least 49 non-isomorphic 3-minimal graphs. First, Figure~\ref{fig3minimal} gives eight of these graphs. Each graph has the property that $\deg(8) = 2$, and the vertices $\set{1,2,3,4,5,6}$ induce either $G_{\ref{figKnownEG},1}$ or $G_{\ref{figKnownEG},2}$.

\def\y{0.7}
\def\sc{1.6}
\def\z{360/ (5+4*\y)}
\def\x{270}
\def\w{1 / cos(\z*\y)}

\def\placev{
\node[main node] at ({ \x+(-1 + 0*\y) *\z} : 1) (1) {$1$};
\node[main node] at ({ \x+(0 + 0*\y) *\z} : 1) (2) {$2$};
\node[main node] at ({ \x+(1 + 0*\y) *\z} : 1) (3) {$3$};
\node[main node] at ({ \x+(2 + 0*\y) *\z} : 1) (4) {$4$};
\node[main node] at ({ \x+(2 + 1*\y) *\z} : 1) (5) {$5$};
\node[main node] at ({ \x+(2 + 2*\y) *\z} : 1) (6) {$6$};
\node[main node] at ({ \x+(3 + 2*\y) *\z} : 1) (7) {$7$};
\node[main node] at ({ \x+(3 + 3*\y) *\z} : 1) (8) {$8$};
\node[main node] at ({ \x+(3 + 4*\y) *\z} : 1) (9) {$9$};
}

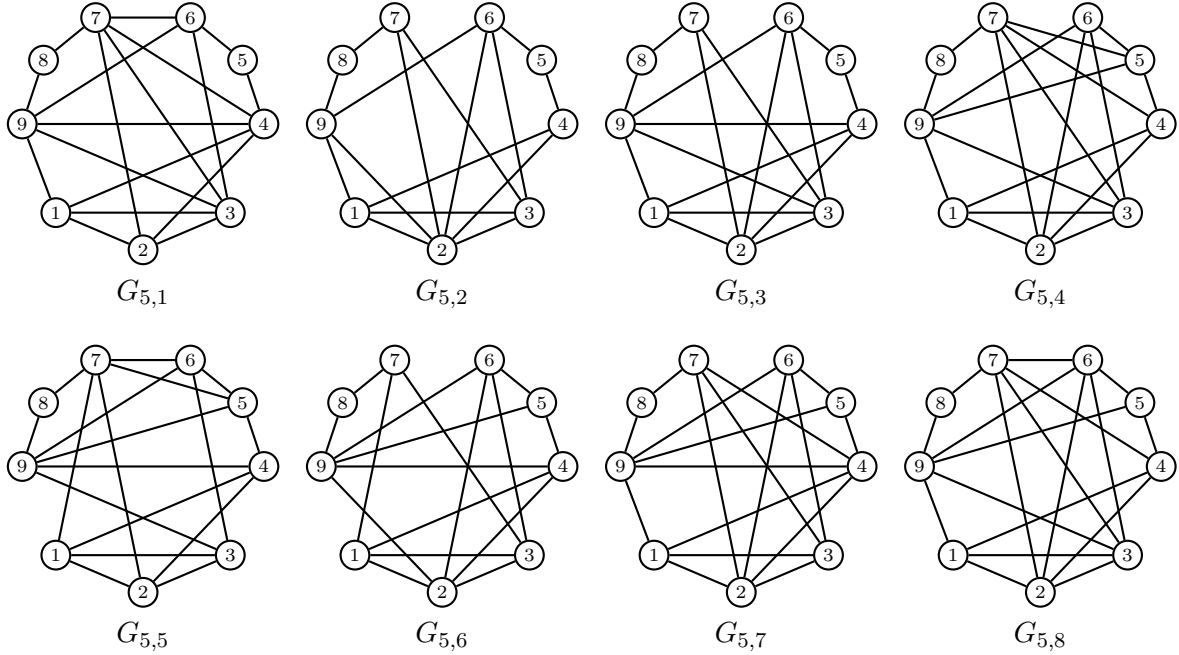
\begin{figure}[htbp]
\centering
\begin{tabular}{cccc}
\begin{tikzpicture}
[scale=\sc, thick,main node/.style={circle, minimum size=3.8mm, inner sep=0.1mm,draw,font=\tiny\sffamily}]
\placev

\path[every node/.style={font=\sffamily}]
(1) edge(2) (1) edge(3) (1) edge(4) (1) edge(9) (2) edge(3) (2) edge(4) (2) edge(7) (3) edge(6) (3) edge(7) (3) edge(9) (4) edge(5) (4) edge(7) (4) edge(9) (5) edge(6) (6) edge(7) (6) edge(9) (7) edge(8) (8) edge(9) ;
\end{tikzpicture}
&
\begin{tikzpicture}
[scale=\sc, thick,main node/.style={circle, minimum size=3.8mm, inner sep=0.1mm,draw,font=\tiny\sffamily}]
\placev
\path[every node/.style={font=\sffamily}]
(1) edge(2) (1) edge(3) (1) edge(4) (1) edge(9) (2) edge(3) (2) edge(4) (2) edge(6) (2) edge(7) (2) edge(9) (3) edge(6) (3) edge(7) (4) edge(5) (5) edge(6) (6) edge(9) (7) edge(8) (8) edge(9) ;
\end{tikzpicture}
&
\begin{tikzpicture}
[scale=\sc, thick,main node/.style={circle, minimum size=3.8mm, inner sep=0.1mm,draw,font=\tiny\sffamily}]
\placev
\path[every node/.style={font=\sffamily}]
(1) edge(2) (1) edge(3) (1) edge(4) (1) edge(9) (2) edge(3) (2) edge(4) (2) edge(6) (2) edge(7) (3) edge(6) (3) edge(7) (3) edge(9) (4) edge(5) (4) edge(9) (5) edge(6) (6) edge(9) (7) edge(8) (8) edge(9) ;
\end{tikzpicture}
&
\begin{tikzpicture}
[scale=\sc, thick,main node/.style={circle, minimum size=3.8mm, inner sep=0.1mm,draw,font=\tiny\sffamily}]
\placev
\path[every node/.style={font=\sffamily}]
(1) edge(2) (1) edge(3) (1) edge(4) (1) edge(9) (2) edge(3) (2) edge(4) (2) edge(6) (2) edge(7) (3) edge(6) (3) edge(7) (3) edge(9) (4) edge(5) (4) edge(7) (5) edge(6) (5) edge(7) (5) edge(9) (6) edge(9) (7) edge(8) (8) edge(9) ;
\end{tikzpicture}
\\
$G_{\ref{fig3minimal},1}$ & $G_{\ref{fig3minimal},2}$ & $G_{\ref{fig3minimal},3}$ & $G_{\ref{fig3minimal},4}$ \\ 
\\
\begin{tikzpicture}
[scale=\sc, thick,main node/.style={circle, minimum size=3.8mm, inner sep=0.1mm,draw,font=\tiny\sffamily}]
\placev
\path[every node/.style={font=\sffamily}]
(1) edge(2) (1) edge(3) (1) edge(4) (1) edge(7) (2) edge(3) (2) edge(4) (2) edge(7) (3) edge(6) (3) edge(9) (4) edge(5) (4) edge(9) (5) edge(6) (5) edge(7) (5) edge(9) (6) edge(7) (6) edge(9) (7) edge(8) (8) edge(9) ;
\end{tikzpicture}
&
\begin{tikzpicture}
[scale=\sc, thick,main node/.style={circle, minimum size=3.8mm, inner sep=0.1mm,draw,font=\tiny\sffamily}]
\placev
\path[every node/.style={font=\sffamily}]
(1) edge(2) (1) edge(3) (1) edge(4) (1) edge(7) (2) edge(3) (2) edge(4) (2) edge(6) (2) edge(9) (3) edge(6) (3) edge(7) (4) edge(5) (4) edge(9) (5) edge(6) (5) edge(9) (6) edge(9) (7) edge(8) (8) edge(9) ;
\end{tikzpicture}
&
\begin{tikzpicture}
[scale=\sc, thick,main node/.style={circle, minimum size=3.8mm, inner sep=0.1mm,draw,font=\tiny\sffamily}]
\placev
\path[every node/.style={font=\sffamily}]
(1) edge(2) (1) edge(3) (1) edge(4) (1) edge(9) (2) edge(3) (2) edge(4) (2) edge(6) (2) edge(7) (3) edge(6) (3) edge(7) (4) edge(5) (4) edge(7) (4) edge(9) (5) edge(6) (5) edge(9) (6) edge(9) (7) edge(8) (8) edge(9) ;
\end{tikzpicture}
&
\begin{tikzpicture}
[scale=\sc, thick,main node/.style={circle, minimum size=3.8mm, inner sep=0.1mm,draw,font=\tiny\sffamily}]
\placev
\path[every node/.style={font=\sffamily}]
(1) edge(2) (1) edge(3) (1) edge(4) (1) edge(9) (2) edge(3) (2) edge(4) (2) edge(6) (2) edge(7) (3) edge(6) (3) edge(7) (3) edge(9) (4) edge(5) (4) edge(7) (5) edge(6) (5) edge(9) (6) edge(7) (6) edge(9) (7) edge(8) (8) edge(9) ;
\end{tikzpicture}

\\
$G_{\ref{fig3minimal},5}$ & $G_{\ref{fig3minimal},6}$ & $G_{\ref{fig3minimal},7}$ & $G_{\ref{fig3minimal},8}$
\end{tabular}
\caption{Illustrating the eight edge-maximal 3-minimal graphs.}\label{fig3minimal}

\end{figure}

\begin{proposition}\label{prop3minimal1}
Each of the eight graphs in Figure~\ref{fig3minimal} is 3-minimal.
\end{proposition}

\begin{proof}
First, for $i \in \set{1,2,3}$, $\bar{e}^{\top}x \leq 3$ is valid for $\STAB(G_{\ref{fig3minimal},i})$, and for each of these graphs we provide an $\LS_+^2$ certificate package~\cite{AuT25data} for a point that violates this inequality, thus by Proposition~\ref{prop:LS_+^2-certificate}, $r_+(G_{\ref{fig3minimal},i}) \geq 3$.

Similarly, for $i \in \set{4,5,6,7,8}$, $(1,1,1,1,1,1,1,1,2)^{\top}x \leq 3$ is valid for $\STAB(G_{\ref{fig3minimal},i})$, and we also provide $\LS_+^2$ certificate packages for points violating this inequality.

Since $|V(G)|= 9$ readily implies $r_+(G) \leq 3$ (by Theorem~\ref{thmLiptakT031}), we have $r_+(G) = 3$ in all eight cases.
\end{proof}

Next, it follows from Lemma~\ref{lem05subgraph} that if the inequality $a^{\top}x \leq \b$ has $\LS_+$-rank $3$ for $\STAB(G)$ and is also valid for $\STAB(H)$ where $H$ is an edge subgraph of $G$, then $r_+(H) \geq 3$. Thus, many edge subgraphs of $G_{\ref{fig3minimal},1}, \ldots, G_{\ref{fig3minimal},8}$ are also 3-minimal. After removing isomorphic graphs, this results in a list of 49 3-minimal graphs, which we list in Figure~\ref{fig3minimalall}. (Note that, to reduce cluttering and with every vertex having a single-digit vertex label, we use $ij$ to denote the edge $\set{i,j}$.) 

\def\sc{1.2}
\def\y{0.7}
\def\z{360/ (5+4*\y)}
\def\x{270}
\def\placev{
\node[main node] at ({ \x+(-1 + 0*\y) *\z} : 1) (1) {};
\node[main node] at ({ \x+(0 + 0*\y) *\z} : 1) (2) {};
\node[main node] at ({ \x+(1 + 0*\y) *\z} : 1) (3) {};
\node[main node] at ({ \x+(2 + 0*\y) *\z} : 1) (4) {};
\node[main node] at ({ \x+(2 + 1*\y) *\z} : 1) (5) {};
\node[main node] at ({ \x+(2 + 2*\y) *\z} : 1) (6) {};
\node[main node] at ({ \x+(3 + 2*\y) *\z} : 1) (7) {};
\node[main node] at ({ \x+(3 + 3*\y) *\z} : 1) (8) {};
\node[main node] at ({ \x+(3 + 4*\y) *\z} : 1) (9) {};
}

\begin{figure}[htbp]
\centering
\scalebox{0.75}{
\tiny{
\begin{tabular}{ccccccc}
\begin{tikzpicture} 
[scale=\sc, thick,main node/.style={circle, minimum size=1.5mm, fill=black, inner sep=0.1mm,draw,font=\tiny\sffamily}] \placev \path[every node/.style={font=\sffamily}]
(1) edge(2) (1) edge(3) (1) edge(4) (1) edge(9) (2) edge(3) (2) edge(4) (2) edge(7) (3) edge(6) (3) edge(7) (3) edge(9) (4) edge(5) (4) edge(7) (4) edge(9) (5) edge(6) (6) edge(7) (6) edge(9) (7) edge(8) (8) edge(9) ;
\end{tikzpicture}
&
\begin{tikzpicture} 
[scale=\sc, thick,main node/.style={circle, minimum size=1.5mm, fill=black, inner sep=0.1mm,draw,font=\tiny\sffamily}] \placev \path[every node/.style={font=\sffamily}]
(1) edge(2) (1) edge(3) (1) edge(4) (1) edge(9) (2) edge(3) (2) edge(4) (2) edge(7) (3) edge(6) (3) edge(9) (4) edge(5) (4) edge(7) (4) edge(9) (5) edge(6) (6) edge(7) (6) edge(9) (7) edge(8) (8) edge(9) ;
\end{tikzpicture}
&
\begin{tikzpicture} 
[scale=\sc, thick,main node/.style={circle, minimum size=1.5mm, fill=black, inner sep=0.1mm,draw,font=\tiny\sffamily}] \placev \path[every node/.style={font=\sffamily}]
(1) edge(2) (1) edge(3) (1) edge(4) (1) edge(9) (2) edge(3) (2) edge(4) (2) edge(7) (3) edge(6) (3) edge(7) (3) edge(9) (4) edge(5) (4) edge(9) (5) edge(6) (6) edge(7) (6) edge(9) (7) edge(8) (8) edge(9) ;
\end{tikzpicture}
&
\begin{tikzpicture} 
[scale=\sc, thick,main node/.style={circle, minimum size=1.5mm, fill=black, inner sep=0.1mm,draw,font=\tiny\sffamily}] \placev \path[every node/.style={font=\sffamily}]
(1) edge(2) (1) edge(3) (1) edge(4) (1) edge(9) (2) edge(3) (2) edge(4) (2) edge(7) (3) edge(6) (3) edge(7) (3) edge(9) (4) edge(5) (4) edge(7) (4) edge(9) (5) edge(6) (6) edge(9) (7) edge(8) (8) edge(9) ;
\end{tikzpicture}
&
\begin{tikzpicture} 
[scale=\sc, thick,main node/.style={circle, minimum size=1.5mm, fill=black, inner sep=0.1mm,draw,font=\tiny\sffamily}] \placev \path[every node/.style={font=\sffamily}]
(1) edge(2) (1) edge(3) (1) edge(4) (1) edge(9) (2) edge(3) (2) edge(4) (2) edge(7) (3) edge(6) (3) edge(9) (4) edge(5) (4) edge(9) (5) edge(6) (6) edge(7) (6) edge(9) (7) edge(8) (8) edge(9) ;
\end{tikzpicture}
&
\begin{tikzpicture} 
[scale=\sc, thick,main node/.style={circle, minimum size=1.5mm, fill=black, inner sep=0.1mm,draw,font=\tiny\sffamily}] \placev \path[every node/.style={font=\sffamily}]
(1) edge(2) (1) edge(3) (1) edge(4) (1) edge(9) (2) edge(3) (2) edge(4) (2) edge(7) (3) edge(6) (3) edge(9) (4) edge(5) (4) edge(7) (4) edge(9) (5) edge(6) (6) edge(9) (7) edge(8) (8) edge(9) ;
\end{tikzpicture}
&
\begin{tikzpicture} 
[scale=\sc, thick,main node/.style={circle, minimum size=1.5mm, fill=black, inner sep=0.1mm,draw,font=\tiny\sffamily}] \placev \path[every node/.style={font=\sffamily}]
(1) edge(2) (1) edge(3) (1) edge(4) (1) edge(9) (2) edge(3) (2) edge(4) (2) edge(7) (3) edge(6) (3) edge(7) (4) edge(5) (4) edge(9) (5) edge(6) (6) edge(7) (6) edge(9) (7) edge(8) (8) edge(9) ;
\end{tikzpicture}
\\

$G_{\ref{fig3minimal},1}$&
$G_{\ref{fig3minimal},1} - \set{37} $&
$G_{\ref{fig3minimal},1} - \set{47}$&
$G_{\ref{fig3minimal},1} -\set{67} $&
$G_{\ref{fig3minimal},1} - \set{37,47} $&
$G_{\ref{fig3minimal},1} - \set{37,67}$&
$G_{\ref{fig3minimal},1} - \set{39,47}$\\
\\
\begin{tikzpicture} 
[scale=\sc, thick,main node/.style={circle, minimum size=1.5mm, fill=black, inner sep=0.1mm,draw,font=\tiny\sffamily}] \placev \path[every node/.style={font=\sffamily}]
(1) edge(2) (1) edge(3) (1) edge(4) (1) edge(9) (2) edge(3) (2) edge(4) (2) edge(7) (3) edge(6) (3) edge(7) (3) edge(9) (4) edge(5) (5) edge(6) (6) edge(7) (6) edge(9) (7) edge(8) (8) edge(9) ;
\end{tikzpicture}
&
\begin{tikzpicture} 
[scale=\sc, thick,main node/.style={circle, minimum size=1.5mm, fill=black, inner sep=0.1mm,draw,font=\tiny\sffamily}] \placev \path[every node/.style={font=\sffamily}]
(1) edge(2) (1) edge(3) (1) edge(4) (1) edge(9) (2) edge(3) (2) edge(4) (2) edge(7) (3) edge(6) (3) edge(7) (3) edge(9) (4) edge(5) (4) edge(9) (5) edge(6) (6) edge(7)  (7) edge(8) (8) edge(9) ;
\end{tikzpicture}
&
\begin{tikzpicture} 
[scale=\sc, thick,main node/.style={circle, minimum size=1.5mm, fill=black, inner sep=0.1mm,draw,font=\tiny\sffamily}] \placev \path[every node/.style={font=\sffamily}]
(1) edge(2) (1) edge(3) (1) edge(4) (1) edge(9) (2) edge(3) (2) edge(4) (2) edge(7) (3) edge(6)(3) edge(9) (4) edge(5) (4) edge(9) (5) edge(6) (6) edge(7) (7) edge(8) (8) edge(9) ;
\end{tikzpicture}
&
\begin{tikzpicture} 
[scale=\sc, thick,main node/.style={circle, minimum size=1.5mm, fill=black, inner sep=0.1mm,draw,font=\tiny\sffamily}] \placev \path[every node/.style={font=\sffamily}]
(1) edge(2) (1) edge(3) (1) edge(4) (1) edge(9) (2) edge(3) (2) edge(4) (2) edge(7) (3) edge(6) (3) edge(7) (4) edge(5)  (4) edge(9) (5) edge(6) (6) edge(7)  (7) edge(8) (8) edge(9) ;
\end{tikzpicture}
&
\begin{tikzpicture}  
[scale=\sc, thick,main node/.style={circle, minimum size=1.5mm, fill=black, inner sep=0.1mm,draw,font=\tiny\sffamily}] \placev \path[every node/.style={font=\sffamily}]
(1) edge(2) (1) edge(3) (1) edge(4) (1) edge(9) (2) edge(3) (2) edge(4) (2) edge(6) (2) edge(7) (2) edge(9) (3) edge(6) (3) edge(7) (4) edge(5) (5) edge(6) (6) edge(9) (7) edge(8) (8) edge(9) ;
\end{tikzpicture}
&
\begin{tikzpicture} 
[scale=\sc, thick,main node/.style={circle, minimum size=1.5mm, fill=black, inner sep=0.1mm,draw,font=\tiny\sffamily}] \placev \path[every node/.style={font=\sffamily}]
(1) edge(2) (1) edge(3) (1) edge(4) (1) edge(9) (2) edge(3) (2) edge(4) (2) edge(6) (2) edge(7) (3) edge(6) (3) edge(7) (3) edge(9) (4) edge(5) (4) edge(9) (5) edge(6) (6) edge(9) (7) edge(8) (8) edge(9) ;
\end{tikzpicture}
&
\begin{tikzpicture} 
[scale=\sc, thick,main node/.style={circle, minimum size=1.5mm, fill=black, inner sep=0.1mm,draw,font=\tiny\sffamily}] \placev \path[every node/.style={font=\sffamily}]
(1) edge(2) (1) edge(3) (1) edge(4) (1) edge(9) (2) edge(3)  (2) edge(6) (2) edge(7) (3) edge(6) (3) edge(7) (3) edge(9) (4) edge(5) (4) edge(9) (5) edge(6) (6) edge(9) (7) edge(8) (8) edge(9) ;
\end{tikzpicture}
\\

$G_{\ref{fig3minimal},1} - \set{47,49}$&
$G_{\ref{fig3minimal},1} - \set{47,69}$&
$G_{\ref{fig3minimal},1} - \set{37,47,69}$&
$G_{\ref{fig3minimal},1} - \set{39,47,69}$&
$G_{\ref{fig3minimal},2}$&
$G_{\ref{fig3minimal},3}$&
$G_{\ref{fig3minimal},3} - \set{24}$\\
\\
\begin{tikzpicture} 
[scale=\sc, thick,main node/.style={circle, minimum size=1.5mm, fill=black, inner sep=0.1mm,draw,font=\tiny\sffamily}] \placev \path[every node/.style={font=\sffamily}]
(1) edge(2) (1) edge(3) (1) edge(4) (1) edge(9) (2) edge(3) (2) edge(4) (2) edge(7) (3) edge(6) (3) edge(7) (3) edge(9) (4) edge(5) (4) edge(9) (5) edge(6) (6) edge(9) (7) edge(8) (8) edge(9) ;
\end{tikzpicture}
&
\begin{tikzpicture} 
[scale=\sc, thick,main node/.style={circle, minimum size=1.5mm, fill=black, inner sep=0.1mm,draw,font=\tiny\sffamily}] \placev \path[every node/.style={font=\sffamily}]
(1) edge(2) (1) edge(3) (1) edge(4) (1) edge(9) (2) edge(3) (2) edge(4) (2) edge(6) (2) edge(7) (3) edge(6) (3) edge(9) (4) edge(5) (4) edge(9) (5) edge(6) (6) edge(9) (7) edge(8) (8) edge(9) ;
\end{tikzpicture}
&
\begin{tikzpicture} 
[scale=\sc, thick,main node/.style={circle, minimum size=1.5mm, fill=black, inner sep=0.1mm,draw,font=\tiny\sffamily}] \placev \path[every node/.style={font=\sffamily}]
(1) edge(2) (1) edge(3) (1) edge(4) (1) edge(9) (2) edge(3) (2) edge(4) (2) edge(6) (2) edge(7) (3) edge(6) (3) edge(7) (4) edge(5) (4) edge(9) (5) edge(6) (6) edge(9) (7) edge(8) (8) edge(9) ;
\end{tikzpicture}
&
\begin{tikzpicture} 
[scale=\sc, thick,main node/.style={circle, minimum size=1.5mm, fill=black, inner sep=0.1mm,draw,font=\tiny\sffamily}] \placev \path[every node/.style={font=\sffamily}]
(1) edge(2) (1) edge(3) (1) edge(4) (1) edge(9) (2) edge(3) (2) edge(4) (2) edge(6) (2) edge(7) (3) edge(6) (3) edge(7) (3) edge(9) (4) edge(5) (5) edge(6) (6) edge(9) (7) edge(8) (8) edge(9) ;
\end{tikzpicture}
&
\begin{tikzpicture} 
[scale=\sc, thick,main node/.style={circle, minimum size=1.5mm, fill=black, inner sep=0.1mm,draw,font=\tiny\sffamily}] \placev \path[every node/.style={font=\sffamily}]
(1) edge(2) (1) edge(3) (1) edge(4) (1) edge(9) (2) edge(3) (2) edge(4) (2) edge(6) (2) edge(7) (3) edge(6) (3) edge(7) (3) edge(9) (4) edge(5) (4) edge(9) (5) edge(6) (7) edge(8) (8) edge(9) ;
\end{tikzpicture}
&
\begin{tikzpicture} 
[scale=\sc, thick,main node/.style={circle, minimum size=1.5mm, fill=black, inner sep=0.1mm,draw,font=\tiny\sffamily}] \placev \path[every node/.style={font=\sffamily}]
(1) edge(2) (1) edge(3) (1) edge(4) (1) edge(9) (2) edge(3) (2) edge(6) (2) edge(7) (3) edge(6) (3) edge(7) (4) edge(5) (4) edge(9) (5) edge(6) (6) edge(9) (7) edge(8) (8) edge(9) ;
\end{tikzpicture}
&
\begin{tikzpicture} 
[scale=\sc, thick,main node/.style={circle, minimum size=1.5mm, fill=black, inner sep=0.1mm,draw,font=\tiny\sffamily}] \placev \path[every node/.style={font=\sffamily}]
(1) edge(2) (1) edge(3) (1) edge(4) (1) edge(9) (2) edge(3) (2) edge(6) (2) edge(7) (3) edge(6) (3) edge(7) (3) edge(9) (4) edge(5) (5) edge(6) (6) edge(9) (7) edge(8) (8) edge(9) ;
\end{tikzpicture}
\\

$G_{\ref{fig3minimal},3} - \set{26}$&
$G_{\ref{fig3minimal},3} - \set{37}$&
$G_{\ref{fig3minimal},3} - \set{39}$&
$G_{\ref{fig3minimal},3} - \set{49}$&
$G_{\ref{fig3minimal},3} - \set{69}$&
$G_{\ref{fig3minimal},3} - \set{24,39}$&
$G_{\ref{fig3minimal},3} - \set{24,49}$\\
\\
\begin{tikzpicture} 
[scale=\sc, thick,main node/.style={circle, minimum size=1.5mm, fill=black, inner sep=0.1mm,draw,font=\tiny\sffamily}] \placev \path[every node/.style={font=\sffamily}]
(1) edge(2) (1) edge(3) (1) edge(4) (1) edge(9) (2) edge(3) (2) edge(6) (2) edge(7) (3) edge(6) (3) edge(7) (3) edge(9) (4) edge(5) (4) edge(9) (5) edge(6)  (7) edge(8) (8) edge(9) ;
\end{tikzpicture}
&
\begin{tikzpicture} 
[scale=\sc, thick,main node/.style={circle, minimum size=1.5mm, fill=black, inner sep=0.1mm,draw,font=\tiny\sffamily}] \placev \path[every node/.style={font=\sffamily}]
(1) edge(2) (1) edge(3) (1) edge(4) (1) edge(9) (2) edge(3) (2) edge(4) (2) edge(7) (3) edge(6) (3) edge(7) (4) edge(5) (4) edge(9) (5) edge(6) (6) edge(9) (7) edge(8) (8) edge(9) ;
\end{tikzpicture}
&
\begin{tikzpicture} 
[scale=\sc, thick,main node/.style={circle, minimum size=1.5mm, fill=black, inner sep=0.1mm,draw,font=\tiny\sffamily}] \placev \path[every node/.style={font=\sffamily}]
(1) edge(2) (1) edge(3) (1) edge(4) (1) edge(9) (2) edge(3) (2) edge(4) (2) edge(7) (3) edge(6) (3) edge(7) (3) edge(9) (4) edge(5) (5) edge(6) (6) edge(9) (7) edge(8) (8) edge(9) ;
\end{tikzpicture}
&
\begin{tikzpicture} 
[scale=\sc, thick,main node/.style={circle, minimum size=1.5mm, fill=black, inner sep=0.1mm,draw,font=\tiny\sffamily}] \placev \path[every node/.style={font=\sffamily}]
(1) edge(2) (1) edge(3) (1) edge(4) (1) edge(9) (2) edge(3) (2) edge(4) (2) edge(6) (2) edge(7) (3) edge(6) (3) edge(9) (4) edge(5) (5) edge(6) (6) edge(9) (7) edge(8) (8) edge(9) ;
\end{tikzpicture}
&
\begin{tikzpicture} 
[scale=\sc, thick,main node/.style={circle, minimum size=1.5mm, fill=black, inner sep=0.1mm,draw,font=\tiny\sffamily}] \placev \path[every node/.style={font=\sffamily}]
(1) edge(2) (1) edge(3) (1) edge(4) (1) edge(9) (2) edge(3) (2) edge(4) (2) edge(6) (2) edge(7) (3) edge(6) (3) edge(7) (4) edge(5) (5) edge(6) (6) edge(9) (7) edge(8) (8) edge(9) ;
\end{tikzpicture}
&
\begin{tikzpicture} 
[scale=\sc, thick,main node/.style={circle, minimum size=1.5mm, fill=black, inner sep=0.1mm,draw,font=\tiny\sffamily}] \placev \path[every node/.style={font=\sffamily}]
(1) edge(2) (1) edge(3) (1) edge(4) (1) edge(9) (2) edge(3) (2) edge(6) (2) edge(7) (3) edge(6) (3) edge(9) (4) edge(5)  (5) edge(6) (6) edge(9) (7) edge(8) (8) edge(9) ;
\end{tikzpicture}
&
\begin{tikzpicture} 
[scale=\sc, thick,main node/.style={circle, minimum size=1.5mm, fill=black, inner sep=0.1mm,draw,font=\tiny\sffamily}] \placev \path[every node/.style={font=\sffamily}]
(1) edge(2) (1) edge(3) (1) edge(4) (1) edge(9) (2) edge(3) (2) edge(6) (2) edge(7) (3) edge(6) (3) edge(7) (4) edge(5) (5) edge(6) (6) edge(9) (7) edge(8) (8) edge(9) ;
\end{tikzpicture}
\\

$G_{\ref{fig3minimal},3} -\set{24,69} $&
$G_{\ref{fig3minimal},3} - \set{26,39}$&
$G_{\ref{fig3minimal},3} - \set{26,49}$&
$G_{\ref{fig3minimal},3} - \set{37,49}$&
$G_{\ref{fig3minimal},3} - \set{39,49}$&
$G_{\ref{fig3minimal},3} - \set{24,37,49}$&
$G_{\ref{fig3minimal},3} - \set{24,39,49}$\\
\\
\begin{tikzpicture} 
[scale=\sc, thick,main node/.style={circle, minimum size=1.5mm, fill=black, inner sep=0.1mm,draw,font=\tiny\sffamily}] \placev \path[every node/.style={font=\sffamily}]
(1) edge(2) (1) edge(3) (1) edge(4) (1) edge(9) (2) edge(3)  (2) edge(6) (2) edge(7) (3) edge(6) (3) edge(7) (4) edge(5) (4) edge(9) (5) edge(6) (7) edge(8) (8) edge(9) ;
\end{tikzpicture}
&
\begin{tikzpicture} 
[scale=\sc, thick,main node/.style={circle, minimum size=1.5mm, fill=black, inner sep=0.1mm,draw,font=\tiny\sffamily}] \placev \path[every node/.style={font=\sffamily}]
(1) edge(2) (1) edge(3) (1) edge(4) (1) edge(9) (2) edge(3) (2) edge(4) (2) edge(7) (3) edge(6)(3) edge(9) (4) edge(5) (5) edge(6) (6) edge(9) (7) edge(8) (8) edge(9) ;
\end{tikzpicture}
&
\begin{tikzpicture} 
[scale=\sc, thick,main node/.style={circle, minimum size=1.5mm, fill=black, inner sep=0.1mm,draw,font=\tiny\sffamily}] \placev \path[every node/.style={font=\sffamily}]
(1) edge(2) (1) edge(3) (1) edge(4) (1) edge(9) (2) edge(3) (2) edge(4) (2) edge(7) (3) edge(6) (3) edge(7)  (4) edge(5) (5) edge(6) (6) edge(9) (7) edge(8) (8) edge(9) ;
\end{tikzpicture}
&
\begin{tikzpicture} 
[scale=\sc, thick,main node/.style={circle, minimum size=1.5mm, fill=black, inner sep=0.1mm,draw,font=\tiny\sffamily}] \placev \path[every node/.style={font=\sffamily}]
(1) edge(2) (1) edge(3) (1) edge(4) (1) edge(9) (2) edge(3) (2) edge(4) (2) edge(6) (2) edge(7) (3) edge(6) (3) edge(7) (3) edge(9) (4) edge(5) (4) edge(7) (5) edge(6) (5) edge(7) (5) edge(9) (6) edge(9) (7) edge(8) (8) edge(9) ;
\end{tikzpicture}
&
\begin{tikzpicture} 
[scale=\sc, thick,main node/.style={circle, minimum size=1.5mm, fill=black, inner sep=0.1mm,draw,font=\tiny\sffamily}] \placev \path[every node/.style={font=\sffamily}]
(1) edge(2) (1) edge(3) (1) edge(4) (1) edge(9) (2) edge(3) (2) edge(4)  (2) edge(7) (3) edge(6) (3) edge(7) (3) edge(9) (4) edge(5) (4) edge(7) (5) edge(6) (5) edge(7) (5) edge(9) (6) edge(9) (7) edge(8) (8) edge(9) ;
\end{tikzpicture}
&
\begin{tikzpicture} 
[scale=\sc, thick,main node/.style={circle, minimum size=1.5mm, fill=black, inner sep=0.1mm,draw,font=\tiny\sffamily}] \placev \path[every node/.style={font=\sffamily}]
(1) edge(2) (1) edge(3) (1) edge(4) (1) edge(9) (2) edge(3) (2) edge(4) (2) edge(6) (2) edge(7) (3) edge(6) (3) edge(9) (4) edge(5) (4) edge(7) (5) edge(6) (5) edge(7) (5) edge(9) (6) edge(9) (7) edge(8) (8) edge(9) ;
\end{tikzpicture}
&
\begin{tikzpicture} 
[scale=\sc, thick,main node/.style={circle, minimum size=1.5mm, fill=black, inner sep=0.1mm,draw,font=\tiny\sffamily}] \placev \path[every node/.style={font=\sffamily}]
(1) edge(2) (1) edge(3) (1) edge(4) (1) edge(9) (2) edge(3) (2) edge(4) (2) edge(7) (3) edge(6) (3) edge(9) (4) edge(5) (4) edge(7) (5) edge(6) (5) edge(7) (5) edge(9) (6) edge(9) (7) edge(8) (8) edge(9) ;
\end{tikzpicture}
\\

$G_{\ref{fig3minimal},3} - \set{24,39,69}$&
$G_{\ref{fig3minimal},3} - \set{26,37,49}$&
$G_{\ref{fig3minimal},3} - \set{26,39,49}$&
$G_{\ref{fig3minimal},4}$ &
$G_{\ref{fig3minimal},4} - \set{26}$ &
$G_{\ref{fig3minimal},4} -\set{37} $ &
$G_{\ref{fig3minimal},4} - \set{26,37}$ \\
\\
\begin{tikzpicture} 
[scale=\sc, thick,main node/.style={circle, minimum size=1.5mm, fill=black, inner sep=0.1mm,draw,font=\tiny\sffamily}] \placev \path[every node/.style={font=\sffamily}]
(1) edge(2) (1) edge(3) (1) edge(4) (1) edge(7) (2) edge(3) (2) edge(4) (2) edge(7) (3) edge(6) (3) edge(9) (4) edge(5) (4) edge(9) (5) edge(6) (5) edge(7) (5) edge(9) (6) edge(7) (6) edge(9) (7) edge(8) (8) edge(9) ;
\end{tikzpicture}
&
\begin{tikzpicture} 
[scale=\sc, thick,main node/.style={circle, minimum size=1.5mm, fill=black, inner sep=0.1mm,draw,font=\tiny\sffamily}]  \placev \path[every node/.style={font=\sffamily}]
(1) edge(2) (1) edge(3) (1) edge(4) (1) edge(7) (2) edge(3) (2) edge(4) (2) edge(6) (2) edge(9) (3) edge(6) (3) edge(7) (4) edge(5) (4) edge(9) (5) edge(6) (5) edge(9) (6) edge(9) (7) edge(8) (8) edge(9) ;
\end{tikzpicture}
&
\begin{tikzpicture} 
[scale=\sc, thick,main node/.style={circle, minimum size=1.5mm, fill=black, inner sep=0.1mm,draw,font=\tiny\sffamily}]  \placev \path[every node/.style={font=\sffamily}]
(1) edge(2) (1) edge(3) (1) edge(4) (1) edge(7) (2) edge(3) (2) edge(4) (2) edge(9) (3) edge(6) (3) edge(7) (4) edge(5) (4) edge(9) (5) edge(6) (5) edge(9) (6) edge(9) (7) edge(8) (8) edge(9) ;
\end{tikzpicture}
&
\begin{tikzpicture} 
[scale=\sc, thick,main node/.style={circle, minimum size=1.5mm, fill=black, inner sep=0.1mm,draw,font=\tiny\sffamily}] \placev \path[every node/.style={font=\sffamily}]
(1) edge(2) (1) edge(3) (1) edge(4) (1) edge(9) (2) edge(3) (2) edge(4) (2) edge(6) (2) edge(7) (3) edge(6) (3) edge(7) (4) edge(5) (4) edge(7) (4) edge(9) (5) edge(6) (5) edge(9) (6) edge(9) (7) edge(8) (8) edge(9) ;
\end{tikzpicture}
&
\begin{tikzpicture} 
[scale=\sc, thick,main node/.style={circle, minimum size=1.5mm, fill=black, inner sep=0.1mm,draw,font=\tiny\sffamily}] \placev \path[every node/.style={font=\sffamily}]
(1) edge(2) (1) edge(3) (1) edge(4) (1) edge(9) (2) edge(3) (2) edge(6) (2) edge(7) (3) edge(6) (3) edge(7) (4) edge(5) (4) edge(7) (4) edge(9) (5) edge(6) (5) edge(9) (6) edge(9) (7) edge(8) (8) edge(9) ;
\end{tikzpicture}
&
\begin{tikzpicture} 
[scale=\sc, thick,main node/.style={circle, minimum size=1.5mm, fill=black, inner sep=0.1mm,draw,font=\tiny\sffamily}] \placev \path[every node/.style={font=\sffamily}]
(1) edge(2) (1) edge(3) (1) edge(4) (1) edge(9) (2) edge(3) (2) edge(4) (2) edge(7) (3) edge(6) (3) edge(7) (4) edge(5) (4) edge(7) (4) edge(9) (5) edge(6) (5) edge(9) (6) edge(9) (7) edge(8) (8) edge(9) ;
\end{tikzpicture}
&
\begin{tikzpicture} 
[scale=\sc, thick,main node/.style={circle, minimum size=1.5mm, fill=black, inner sep=0.1mm,draw,font=\tiny\sffamily}] \placev \path[every node/.style={font=\sffamily}]
(1) edge(2) (1) edge(3) (1) edge(4) (1) edge(9) (2) edge(3) (2) edge(4) (2) edge(6) (2) edge(7) (3) edge(6) (3) edge(7) (4) edge(5) (4) edge(9) (5) edge(6) (5) edge(9) (6) edge(9) (7) edge(8) (8) edge(9) ;
\end{tikzpicture}
\\

$G_{\ref{fig3minimal},5}$&
$G_{\ref{fig3minimal},6}$&
$G_{\ref{fig3minimal},6} - \set{26}$&
$G_{\ref{fig3minimal},7}$&
$G_{\ref{fig3minimal},7} - \set{24}$&
$G_{\ref{fig3minimal},7} - \set{26}$&
$G_{\ref{fig3minimal},7} - \set{47}$\\
\\
\begin{tikzpicture} 
[scale=\sc, thick,main node/.style={circle, minimum size=1.5mm, fill=black, inner sep=0.1mm,draw,font=\tiny\sffamily}] \placev \path[every node/.style={font=\sffamily}]
(1) edge(2) (1) edge(3) (1) edge(4) (1) edge(9) (2) edge(3) (2) edge(6) (2) edge(7) (3) edge(6) (3) edge(7) (4) edge(5) (4) edge(9) (5) edge(6) (5) edge(9) (6) edge(9) (7) edge(8) (8) edge(9) ;
\end{tikzpicture}
&
\begin{tikzpicture} 
[scale=\sc, thick,main node/.style={circle, minimum size=1.5mm, fill=black, inner sep=0.1mm,draw,font=\tiny\sffamily}] \placev \path[every node/.style={font=\sffamily}]
(1) edge(2) (1) edge(3) (1) edge(4) (1) edge(9) (2) edge(3) (2) edge(4) (2) edge(6) (2) edge(7) (3) edge(6) (3) edge(7) (3) edge(9) (4) edge(5) (4) edge(7) (5) edge(6) (5) edge(9) (6) edge(7) (6) edge(9) (7) edge(8) (8) edge(9) ;
\end{tikzpicture}&
\begin{tikzpicture} 
[scale=\sc, thick,main node/.style={circle, minimum size=1.5mm, fill=black, inner sep=0.1mm,draw,font=\tiny\sffamily}] \placev \path[every node/.style={font=\sffamily}]
(1) edge(2) (1) edge(3) (1) edge(4) (1) edge(9) (2) edge(3) (2) edge(4) (2) edge(7) (3) edge(6) (3) edge(7) (3) edge(9) (4) edge(5) (4) edge(7) (5) edge(6) (5) edge(9) (6) edge(7) (6) edge(9) (7) edge(8) (8) edge(9) ;
\end{tikzpicture}&
\begin{tikzpicture} 
[scale=\sc, thick,main node/.style={circle, minimum size=1.5mm, fill=black, inner sep=0.1mm,draw,font=\tiny\sffamily}] \placev \path[every node/.style={font=\sffamily}]
(1) edge(2) (1) edge(3) (1) edge(4) (1) edge(9) (2) edge(3) (2) edge(4) (2) edge(6) (2) edge(7) (3) edge(6) (3) edge(9) (4) edge(5) (4) edge(7) (5) edge(6) (5) edge(9) (6) edge(7) (6) edge(9) (7) edge(8) (8) edge(9) ;
\end{tikzpicture}&
\begin{tikzpicture} 
[scale=\sc, thick,main node/.style={circle, minimum size=1.5mm, fill=black, inner sep=0.1mm,draw,font=\tiny\sffamily}] \placev \path[every node/.style={font=\sffamily}]
(1) edge(2) (1) edge(3) (1) edge(4) (1) edge(9) (2) edge(3) (2) edge(4) (2) edge(6) (2) edge(7) (3) edge(6) (3) edge(7) (3) edge(9) (4) edge(5) (4) edge(7) (5) edge(6) (5) edge(9) (6) edge(9) (7) edge(8) (8) edge(9) ;
\end{tikzpicture}&
\begin{tikzpicture} 
[scale=\sc, thick,main node/.style={circle, minimum size=1.5mm, fill=black, inner sep=0.1mm,draw,font=\tiny\sffamily}] \placev \path[every node/.style={font=\sffamily}]
(1) edge(2) (1) edge(3) (1) edge(4) (1) edge(9) (2) edge(3) (2) edge(4) (2) edge(7) (3) edge(6) (3) edge(9) (4) edge(5) (4) edge(7) (5) edge(6) (5) edge(9) (6) edge(7) (6) edge(9) (7) edge(8) (8) edge(9) ;
\end{tikzpicture}&
 \begin{tikzpicture} 
[scale=\sc, thick,main node/.style={circle, minimum size=1.5mm, fill=black, inner sep=0.1mm,draw,font=\tiny\sffamily}] \placev \path[every node/.style={font=\sffamily}]
(1) edge(2) (1) edge(3) (1) edge(4) (1) edge(9) (2) edge(3) (2) edge(4) (2) edge(6) (2) edge(7) (3) edge(6) (3) edge(9) (4) edge(5) (4) edge(7) (5) edge(6) (5) edge(9) (6) edge(9) (7) edge(8) (8) edge(9) ;
\end{tikzpicture}
\\
$G_{\ref{fig3minimal},7} - \set{24,47}$&
$G_{\ref{fig3minimal},8}$ &
$G_{\ref{fig3minimal},8} - \set{26}$ &
$G_{\ref{fig3minimal},8} - \set{37}$ &
$G_{\ref{fig3minimal},8} - \set{67}$ &
$G_{\ref{fig3minimal},8} - \set{26,37}$ &
$G_{\ref{fig3minimal},8} - \set{37,67}$ \\
\end{tabular}
}}
\caption{The 49 non-isomorphic 3-minimal graphs identified in this work.}
\label{fig3minimalall}
\end{figure}
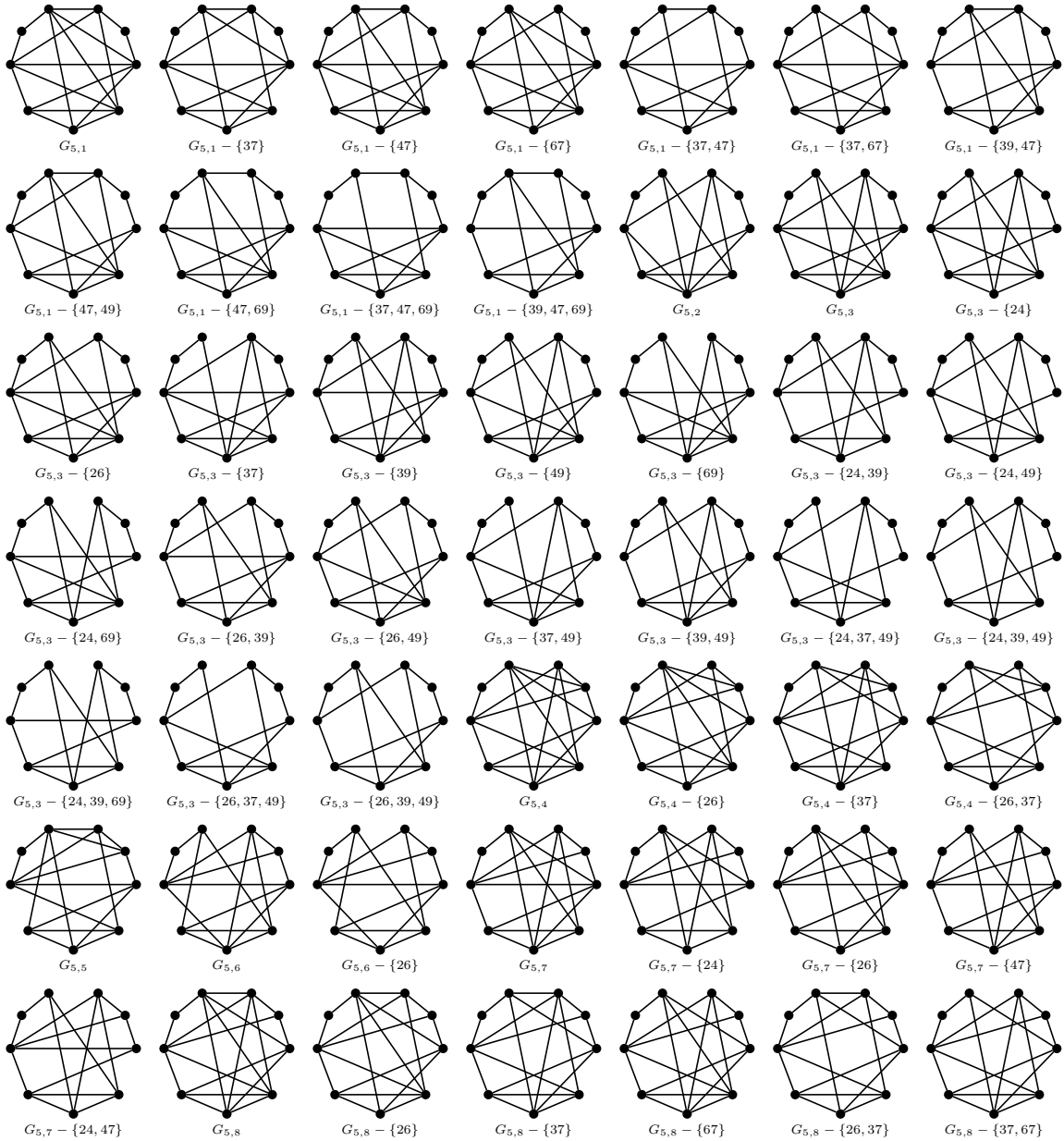

In particular, this proves the following.

\begin{theorem}\label{thm3minimal2}
There are at least 49 non-isomorphic 3-minimal graphs.
\end{theorem}

We remark that the 49 graphs described in Figure~\ref{fig3minimalall} contain (up to isomorphism) every previously known 3-minimal graph described in~\cite{EscalanteMN06, AuT24, AuT25}. Thus, for some of these graphs, we now have multiple independent proofs of their 3-minimality.

\subsection{Structural consequences of the new 3-minimal graphs}

Next, let us relate Theorem~\ref{thm3minimal2} to the existing findings about 3-minimal graphs. First, notice that $G_{\ref{fig3minimal},1}$, $G_{\ref{fig3minimal},2}$, $G_{\ref{fig3minimal},3}$ and their edge subgraphs described in Figure~\ref{fig3minimalall} all belong to $\K_{5,2}$. 13 of these graphs further belong to $\hat{\K}_{5,2}$ and are exactly those shown in Figure~\ref{fighatK52}. Our list also contains 18 3-minimal graphs which belong to $\K_{5,2} \setminus \hat{\K}_{5,2}$.

On the other hand, $G_{\ref{fig3minimal},4}, \ldots, G_{\ref{fig3minimal},8}$, as well as their edge subgraphs described in Figure~\ref{fig3minimalall}, do not belong to $\K_{5,2}$. (One way to see this is that every graph $\K_{n,d}$ has at least $d$ vertices of degree $2$.) Thus, these graphs provide the first known instances of $\ell$-minimal graphs which cannot be obtained by stretching the vertices of a clique. Also, with 19 edges, $G_{\ref{fig3minimal},4}$ and $G_{\ref{fig3minimal},5}$ are the densest known 3-minimal graphs yet. The sparsest possible 3-minimal graphs contain 14 edges~\cite[Proposition 28]{AuT24b}, which is attained by several graphs in $\hat{\K}_{5,2}$.

Now recall Theorem~\ref{thmAuT24b1}, which assures that every $\ell$-minimal graph can be obtained from an $(\ell-1)$-minimal graph by applying the following two graph operations:
\begin{itemize}
\item
\emph{1-Join}: Adding a new vertex and joining it to some (or all) vertices of the existing graph;
\item
\emph{2-stretch}: Applying a proper 2-stretching operation to one of the existing vertices of the graph.
\end{itemize}
Applying this observation iteratively, we obtain the following:

\begin{corollary}\label{corJoinStretch}
Let $G$ be an $\ell$-minimal graph for some positive integer $\ell$. Then there exists graphs $G_1, \ldots, G_{\ell}$ and $H_1, \ldots, H_{\ell-1}$ such that
\begin{itemize}
\item
$G_1 = K_3$ and $G_{\ell} = G$;
\item
for every $i \in [\ell-1]$, $H_i$ can be obtained from $G_i$ by adding a new vertex and joining it to some (or all) vertices of $G_i$;
\item
for every $i \in \set{2, \ldots, \ell}$, $G_i$ can be obtained from $H_{i-1}$ by a proper 2-stretching of a vertex.
\end{itemize}
\end{corollary}

Thus, every $\ell$-minimal graph can be constructed by starting with $K_3$ (the unique 1-minimal graph), and then applying 1-join and 2-stretch operations alternatively. For some $\ell$-minimal graphs, there is some flexibility in the sequencing of these 1-join and 2-stretch operations --- for example, an $\ell$-minimal graph in $\hat{\K}_{\ell+2, \ell-1}$ can be constructed by applying $\ell-1$ 1-join operations in a row to $K_3$ to obtain $K_{\ell+2}$, and then applying $\ell-1$ 2-stretch operations. However, the discovery of the 3-minimal graphs outside of $\K_{5,2}$ shows that, for some $\ell$-minimal graphs, there is no flexibility in the order of these operations. For instance, Figure~\ref{figJoinStretch3min} shows the unique sequence of 1-join and 2-stretch operations with which we can construct $G_{\ref{fig3minimal},4}$ from $K_3$.

\def\sc{1.4}
\def\y{0.7}
\def\z{360/ (5+4*\y)}
\def\x{270}

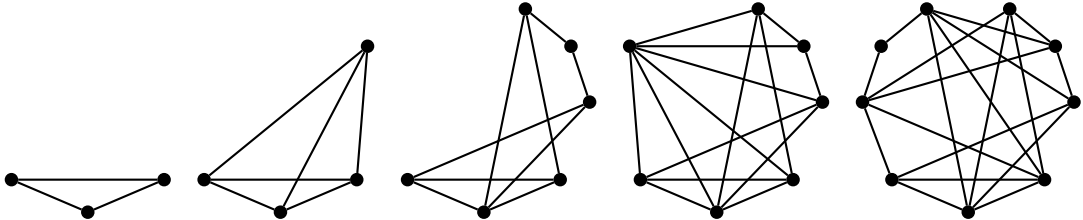
\begin{figure}[htbp]
\centering
\begin{tabular}{ccccc}

\begin{tikzpicture}[scale=\sc, thick,main node/.style={circle,  minimum size=1.5mm, fill=black, inner sep=0.1mm,draw,font=\tiny\sffamily}]
\node[main node] at ({ \x+(-1 + 0*\y) *\z} : 1) (1) {};
\node[main node] at ({ \x+(0 + 0*\y) *\z} : 1) (2) {};
\node[main node] at ({ \x+(1 + 0*\y) *\z} : 1) (3) {};

 \path (1) edge (2) (1) edge (3) (2) edge (3);
\end{tikzpicture}
&

\begin{tikzpicture}[scale=\sc, thick,main node/.style={circle,  minimum size=1.5mm, fill=black, inner sep=0.1mm,draw,font=\tiny\sffamily}]
\node[main node] at ({ \x+(-1 + 0*\y) *\z} : 1) (1) {};
\node[main node] at ({ \x+(0 + 0*\y) *\z} : 1) (2) {};
\node[main node] at ({ \x+(1 + 0*\y) *\z} : 1) (3) {};
\node[main node] at ({ \x+(2 + 1*\y) *\z} : 1) (40) {};

 \path (1) edge (2) (1) edge (3) (2) edge (3) (40) edge (1) (40) edge (2) (40) edge (3) ;
\end{tikzpicture}
&

\begin{tikzpicture}[scale=\sc, thick,main node/.style={circle,  minimum size=1.5mm, fill=black, inner sep=0.1mm,draw,font=\tiny\sffamily}]
\node[main node] at ({ \x+(-1 + 0*\y) *\z} : 1) (1) {};
\node[main node] at ({ \x+(0 + 0*\y) *\z} : 1) (2) {};
\node[main node] at ({ \x+(1 + 0*\y) *\z} : 1) (3) {};
\node[main node] at ({ \x+(2 + 0*\y) *\z} : 1) (41) {};
\node[main node] at ({ \x+(2 + 1*\y) *\z} : 1) (40) {};
\node[main node] at ({ \x+(2 + 2*\y) *\z} : 1) (42) {};

 \path (1) edge (2) (1) edge (3) (2) edge (3) (40) edge (41) (40) edge (42) (41) edge (1) (41) edge (2) (42) edge (2) (42) edge (3) ;
\end{tikzpicture}
&

\begin{tikzpicture}[scale=\sc, thick,main node/.style={circle,  minimum size=1.5mm, fill=black, inner sep=0.1mm,draw,font=\tiny\sffamily}]
\node[main node] at ({ \x+(-1 + 0*\y) *\z} : 1) (1) {};
\node[main node] at ({ \x+(0 + 0*\y) *\z} : 1) (2) {};
\node[main node] at ({ \x+(1 + 0*\y) *\z} : 1) (3) {};
\node[main node] at ({ \x+(2 + 0*\y) *\z} : 1) (41) {};
\node[main node] at ({ \x+(2 + 1*\y) *\z} : 1) (40) {};
\node[main node] at ({ \x+(2 + 2*\y) *\z} : 1) (42) {};
\node[main node] at ({ \x+(3 + 3*\y) *\z} : 1) (50) {};

 \path (1) edge (2) (1) edge (3) (2) edge (3) (40) edge (41) (40) edge (42) (41) edge (1) (41) edge (2) (42) edge (2) (42) edge (3) (50) edge (1) (50) edge (2) (50) edge (3) (50) edge (40) (50) edge (41) (50) edge (42);
\end{tikzpicture}
&

\begin{tikzpicture}[scale=\sc, thick,main node/.style={circle,  minimum size=1.5mm, fill=black, inner sep=0.1mm,draw,font=\tiny\sffamily}]
\node[main node] at ({ \x+(-1 + 0*\y) *\z} : 1) (1) {};
\node[main node] at ({ \x+(0 + 0*\y) *\z} : 1) (2) {};
\node[main node] at ({ \x+(1 + 0*\y) *\z} : 1) (3) {};
\node[main node] at ({ \x+(2 + 0*\y) *\z} : 1) (41) {};
\node[main node] at ({ \x+(2 + 1*\y) *\z} : 1) (40) {};
\node[main node] at ({ \x+(2 + 2*\y) *\z} : 1) (42) {};
\node[main node] at ({ \x+(3 + 2*\y) *\z} : 1) (51) {};
\node[main node] at ({ \x+(3 + 3*\y) *\z} : 1) (50) {};
\node[main node] at ({ \x+(3 + 4*\y) *\z} : 1) (52) {};

 \path (1) edge (2) (1) edge (3) (2) edge (3) (40) edge (41) (40) edge (42) (41) edge (1) (41) edge (2) (42) edge (2) (42) edge (3) (50) edge (51) (50) edge (52) (51) edge (2) (51) edge (3) (51) edge (41) (51) edge (40) (52) edge (1) (52) edge (3) (52) edge (40) (52) edge (42);
\end{tikzpicture}
\end{tabular}
\caption{Obtaining $G_{\ref{fig3minimal}, 4}$ from $K_3$ via 1-join and 2-stretch operations.}
\label{figJoinStretch3min}
\end{figure}

Also, observe that a common feature of all 3-minimal graphs shown in Figure~\ref{fig3minimalall} is that they all contain  $K_5$ as a graph minor. In particular, for each graph, contracting the edges $\set{4,5}$, $\set{5,6}$, $\set{7,8}$, and $\set{8,9}$ (and then removing parallel edges) would result in $K_5$. This shows that each of the 49 graphs in Figure~\ref{fig3minimalall} contains a stretched clique in $\K_{5,2}$ as an edge subgraph. As we shall see in the next section, this pattern no longer holds for 4-minimal graphs.

Next, we turn our attention to the clique number of a graph, which has been shown in~\cite{AuT25} to be relevant in determining whether a graph is $\ell$-minimal under some circumstances. First, we see that among the 49 3-minimal graphs shown in Figure~\ref{fig3minimalall}, 48 of them have $\omega(G) =3$, and one has $\omega(G) = 4$ ($G_{\ref{fig3minimal}, 8}$, with the vertices $\set{2,3,6,7}$ inducing a $K_4$). Thus, we see that $\omega(G) \leq 3$ is not a necessary condition for a graph to be $\ell$-minimal in general.

The situation seems to be more interesting if we restrict our discussion to stretched cliques. Before we go further, the following lemma will be helpful.

\begin{lemma}\label{lemSparseStretchedClique}
Given $G \in \K_{n,d}$ where $n \geq 3$ and $d \geq 0$, we have
\begin{itemize}
\item[(i)]
$|E(G)| \geq \frac{n(n-1)}{2} + 2d$.
\item[(ii)]
If $\omega(G) \geq 3$, then $G$ contains an edge subgraph $H \in \K_{n,d}$ where $|E(H)| = \frac{n(n-1)}{2} + 2d$ and $\omega(H) = \omega(G)$.
\end{itemize}
\end{lemma}

\begin{proof}
We first prove (i). Given distinct $i, j \in [n]$, it follows from the definition of the vertex-stretching operation that there exists at least one edge in $G$ which joins a vertex associated with $i$ and a vertex associated with $j$, giving a total of at least $\frac{n(n-1)}{2}$ distinct edges. Also, for every $i \in D(G)$, we have the edges $\set{i_0, i_1}$ and $\set{i_0,i_2}$, which yields an additional total of $2  |D(G)| = 2d$ edges. Thus, $|E(G)| \geq \frac{n(n-1)}{2} + 2d$.

For (ii), suppose $K \subseteq V(G)$ induces a clique of size at least three in $G$. Then observe that no two vertices in $K$ can be associated with the same index in $[n]$. Thus, if there are multiple edges joining vertices associated with distinct $i,j \in [n]$, we can remove all but one of them without deleting edges that join vertices in $K$. Doing this for all distinct $i,j \in [n]$ would yield the desired graph $H$, which is an edge subgraph of $G$ with $|E(H)| = \frac{n(n-1)}{2} + 2d$ and $\omega(H) = \omega(G)$ (since $K$ still induces a clique in $H$).
\end{proof}

Therefore, given integers $n, d$ where $n \geq 3$ and $d \geq 0$, we say that $G \in \K_{n,d}$ is a \emph{sparse stretched clique} if $|E(G)| = \frac{n(n-1)}{2} + 2d$. Observe that a sparse stretched clique necessarily belongs to $\hat{\K}_{n,d}$.

Next, recall Theorem~\ref{thmAuT251}, which states that given $G \in \hat{\K}_{5,2}$, $\omega(G) \leq 3$ is sufficient for $G$ to be 3-minimal. On the other hand, numerical evidence suggests that $\omega(G) \leq 3$ is not sufficient for $G \in \K_{5,2}$ to be 3-minimal. In Figure~\ref{figK52}, we list the seven graphs for which $\bar{e}^{\top}x \leq 3$ is the lone full-support facet, but CVX+SeDuMi computations indicate that $r_+(G) \lesssim 2$ in all seven cases.

\def\sc{0.85}
\def\y{0.7}
\def\z{360/ (5+4*\y)}
\def\x{270}
\def\placev{
\node[main node] at ({ \x+(-1 + 0*\y) *\z} : 1) (1) {};
\node[main node] at ({ \x+(0 + 0*\y) *\z} : 1) (2) {};
\node[main node] at ({ \x+(1 + 0*\y) *\z} : 1) (3) {};
\node[main node] at ({ \x+(2 + 0*\y) *\z} : 1) (41) {};
\node[main node] at ({ \x+(2 + 1*\y) *\z} : 1) (40) {};
\node[main node] at ({ \x+(2 + 2*\y) *\z} : 1) (42) {};
\node[main node] at ({ \x+(3 + 2*\y) *\z} : 1) (51) {};
\node[main node] at ({ \x+(3 + 3*\y) *\z} : 1) (50) {};
\node[main node] at ({ \x+(3 + 4*\y) *\z} : 1) (52) {};
}

\def\placecommone{  \path (1) edge (2) (2) edge (3) (3) edge (1) (40) edge (41) (40) edge (42) (50) edge (51) (50) edge (52);}

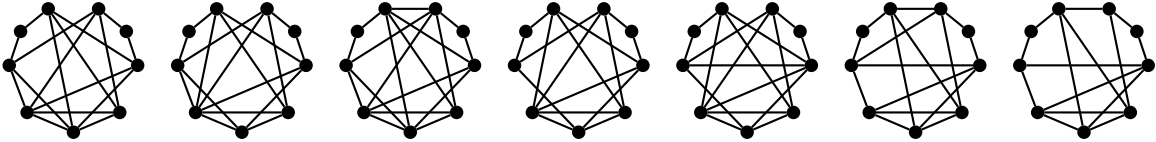
\begin{figure}[htbp]
\centering
\begin{tabular}{ccccccc}
\begin{tikzpicture}[scale=\sc, thick,main node/.style={circle,  minimum size=1.5mm, fill=black, inner sep=0.1mm,draw,font=\tiny\sffamily}] \placev \placecommone
 \path (1) edge (41) (1) edge (42) (2) edge (41) (3) edge (42)  (1) edge (52)   (2) edge (51) (2) edge (52) (3) edge (51) (41) edge (51) (42) edge (52);
\end{tikzpicture}
& 
\begin{tikzpicture}[scale=\sc, thick,main node/.style={circle,  minimum size=1.5mm, fill=black, inner sep=0.1mm,draw,font=\tiny\sffamily}] \placev \placecommone
 \path (1) edge (41) (1) edge (42) (2) edge (41) (3) edge (42)  (1) edge (51) (1) edge (52)  (2) edge (52) (3) edge (51) (41) edge (51) (42) edge (52);
\end{tikzpicture}
& 
\begin{tikzpicture}[scale=\sc, thick,main node/.style={circle,  minimum size=1.5mm, fill=black, inner sep=0.1mm,draw,font=\tiny\sffamily}] \placev \placecommone
 \path (1) edge (41) (1) edge (42) (2) edge (41) (3) edge (42)  (1) edge (52)   (2) edge (51) (2) edge (52) (3) edge (51) (41) edge (51) (42) edge (51) (42) edge (52);
\end{tikzpicture}
& 
\begin{tikzpicture}[scale=\sc, thick,main node/.style={circle,  minimum size=1.5mm, fill=black, inner sep=0.1mm,draw,font=\tiny\sffamily}] \placev \placecommone
 \path (1) edge (41) (1) edge (42) (2) edge (41) (3) edge (42)  (1) edge (51) (2) edge (52) (3) edge (51) (41) edge (51) (42) edge (52);
\end{tikzpicture}
& 
\begin{tikzpicture}[scale=\sc, thick,main node/.style={circle,  minimum size=1.5mm, fill=black, inner sep=0.1mm,draw,font=\tiny\sffamily}] \placev \placecommone
 \path (1) edge (41) (1) edge (42) (2) edge (41) (3) edge (42)  (1) edge (51) (2) edge (52) (3) edge (51) (41) edge (51) (41) edge (52) (42) edge (52);
\end{tikzpicture}
& 
 \begin{tikzpicture}[scale=\sc, thick,main node/.style={circle,  minimum size=1.5mm, fill=black, inner sep=0.1mm,draw,font=\tiny\sffamily}] \placev \placecommone
 \path (1) edge (41) (2) edge (41) (3) edge (42)  (1) edge (52)   (2) edge (51)  (3) edge (51) (41) edge (52) (42) edge (51)(42) edge (52);
\end{tikzpicture}
& 
 \begin{tikzpicture}[scale=\sc, thick,main node/.style={circle,  minimum size=1.5mm, fill=black, inner sep=0.1mm,draw,font=\tiny\sffamily}] \placev \placecommone
 \path (1) edge (41) (2) edge (41) (3) edge (42)  (1) edge (52)   (2) edge (51)  (3) edge (51) (41) edge (52) (42) edge (51);
\end{tikzpicture}

\end{tabular}
\caption{The seven graphs $G \in \K_{5,2}$ with $\omega(G) \leq 3$ and $r_+(G) \lesssim 2$.}
\label{figK52}
\end{figure}

Thus, given $G \in \K_{5,2}$, the inequality $\bar{e}^{\top}x \leq 3$ may seem to have $\LS_+$-rank 2 or 3 when $\omega(G)=3$. On the other hand, we show that this inequality cannot have $\LS_+$-rank 3 if $\omega(G) \geq 4$.

\begin{proposition}\label{propK52K4}
Let $G \in \K_{5,2}$. If $\omega(G) \geq 4$, then the inequality $\bar{e}^{\top}x \leq 3$ has $\LS_+$-rank at most 2.
\end{proposition}

\begin{proof}
Given $G \in \K_{5,2}$ with $\omega(G) \geq 4$, we know from Lemma~\ref{lemSparseStretchedClique} that $G$ has an edge subgraph $H \in \K_{5,2}$ where $|E(H)| = 14$ and $\omega(H) = \omega(G)$. By Lemma~\ref{lem05subgraph}, it suffices to show that $\bar{e}^{\top}x \leq 3$ has $\LS_+$-rank at most $2$ for $\STAB(H)$.

Next, let us focus on the graph $H$. If $\bar{e}^{\top} x \leq 3$ is not a facet of $\STAB(H)$, then it follows from the proof of~\cite[Lemma 7]{AuT25} that $\bar{e}^{\top}x \leq 3$ can be expressed as the sum of facets of $\STAB(H)$ that do not have full support, in which case Lemma~\ref{lemfacet2} implies that $r_+(H) \leq 2$. Thus, we may assume that $\bar{e}^{\top}x \leq 3$ is indeed a facet of $\STAB(H)$. (For a complete characterization of when $\bar{e}^{\top}x \leq d+1$ is a facet of the stable set polytope of a graph in $\K_{n,d}$, see~\cite[Lemma 6]{AuT25}.)

Now, an exhaustive search shows that there are exactly three non-isomorphic graphs $H \in \K_{5,2}$ where $|E(H)| = 14$, $\omega(H) \geq 4$, and $\bar{e}^{\top}x \leq 3$ is a facet of $\STAB(H)$,  which are shown in Figure~\ref{figSparsePath}. We now prove that these three graphs all have $\LS_+$-rank at most 2. First, notice that $G_1 - 4$ is a perfect graph. Thus, $r_+(G_1-4) \leq 1$, which implies that $r_+(G_1) \leq 2$. Likewise, $G_2-3$ and $G_3-4$ are also perfect. Therefore, $G_1$, $G_2$, and $G_3$ (and thus, in particular, the inequality $\bar{e}^{\top}x \leq 3$ of their stable set polytopes) all have $\LS_+$-rank at most $2$. This proves our claim.
\end{proof}

\def\y{0.7}
\def\sc{1.5}
\def\z{360/ (5+4*\y)}
\def\x{270}
\def\w{1 / cos(\z*\y)}

\def\placev{
\node[main node] at ({ \x+(-1 + 0*\y) *\z} : 1) (1) {$1$};
\node[main node] at ({ \x+(0 + 0*\y) *\z} : 1) (2) {$2$};
\node[main node] at ({ \x+(1 + 0*\y) *\z} : 1) (3) {$3$};
\node[main node] at ({ \x+(2 + 0*\y) *\z} : 1) (4) {$4$};
\node[main node] at ({ \x+(2 + 1*\y) *\z} : 1) (5) {$5$};
\node[main node] at ({ \x+(2 + 2*\y) *\z} : 1) (6) {$6$};
\node[main node] at ({ \x+(3 + 2*\y) *\z} : 1) (7) {$7$};
\node[main node] at ({ \x+(3 + 3*\y) *\z} : 1) (8) {$8$};
\node[main node] at ({ \x+(3 + 4*\y) *\z} : 1) (9) {$9$};
}

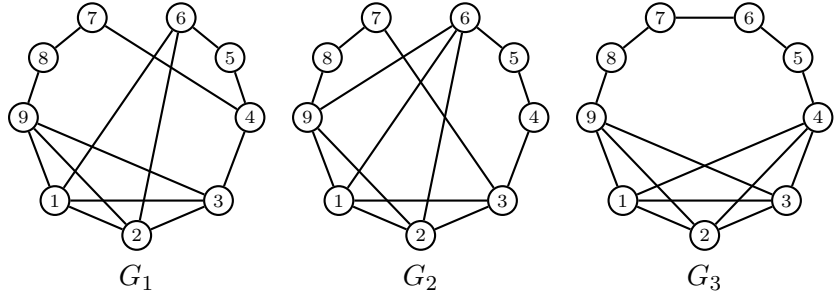
\begin{figure}[htbp]
\centering
\begin{tabular}{ccc}
\begin{tikzpicture}
[scale=\sc, thick,main node/.style={circle, minimum size=3.8mm, inner sep=0.1mm,draw,font=\tiny\sffamily}]
\placev
 \path[every node/.style={font=\sffamily}]
(1) edge (2) (2) edge (3) (1) edge (3) (4) edge (5) (5) edge (6) (7) edge (8) (8) edge (9) (4) edge (3) (6) edge (1) (6) edge (2) (1) edge (9) (2) edge (9) (3) edge (9) (4) edge (7);
\end{tikzpicture}
&
\begin{tikzpicture}
[scale=\sc, thick,main node/.style={circle, minimum size=3.8mm, inner sep=0.1mm,draw,font=\tiny\sffamily}]
\placev
 \path[every node/.style={font=\sffamily}]
(6) edge (2) (2) edge (3) (3) edge (4) (4) edge (5) (5) edge (6) (7) edge (8) (8) edge (9) (6) edge (1) (1) edge (2) (1) edge (3) (7) edge (3) (9) edge (1) (9) edge (2) (9) edge (6);
\end{tikzpicture}
&
\begin{tikzpicture}
[scale=\sc, thick,main node/.style={circle, minimum size=3.8mm, inner sep=0.1mm,draw,font=\tiny\sffamily}]
\placev
 \path[every node/.style={font=\sffamily}]
(1) edge (2) (2) edge (3) (3) edge (1) (4) edge (5) (5) edge (6) (7) edge (8) (8) edge (9) (4) edge (1) (4) edge (2) (4) edge (3) (9) edge (3) (9) edge (1) (9) edge (2) (7) edge (6);
\end{tikzpicture}
\\

$G_1$ & $G_2$  & $G_3$
\end{tabular}
\caption{The three non-isomorphic graphs $H \in \K_{5,2}$ where $|E(H)| =14$, $\omega(H) \geq 4$, and $\bar{e}^{\top}x \leq 3$ is a facet of $\STAB(H)$.}
\label{figSparsePath}
\end{figure}

\subsection{Numerical evidence and heuristic observations from the 3-minimal search}
Next, we provide some details about our computational search that turned up the 49 3-minimal graphs we listed in Figure~\ref{fig3minimalall}. First, given a 3-minimal graph $G$, we may assume (due to Theorem~\ref{thmAuT24b1}) that $G$ can be obtained from a proper 2-stretching of $i \in V(H)$ for some 7-vertex graph $H$ where $H-i$ is isomorphic to $G_{\ref{figKnownEG},1}$ or $G_{\ref{figKnownEG},2}$ (which, again, are the only 2-minimal graphs). An exhaustive search found 1,115 non-isomorphic graphs which satisfy these conditions. Among them, there are 540 instances of $(G,a)$ where $a \in \mR^9$ defines a full-support facet for $\STAB(G)$. For convenience, let $\X_3$ denote the set consisting of these 540 ordered pairs $(G,a)$. Note that there are 9 graphs which appear in two elements in $\X_3$ for having two distinct full-support facets, and none of them belong to $\K_{5,2}$. If it is indeed true that no graph in $\K_{5,2}$ has a full-support facet that is different from $\bar{e}^{\top}x \leq 3$, then it would follow from Proposition~\ref{propK52K4} that $r_+(G) \leq 2$ for every $G \in \K_{5,2}$ where $\omega(G) \geq 4$.

\begin{figure}[htbp]
\centering
\includegraphics[width=15cm]{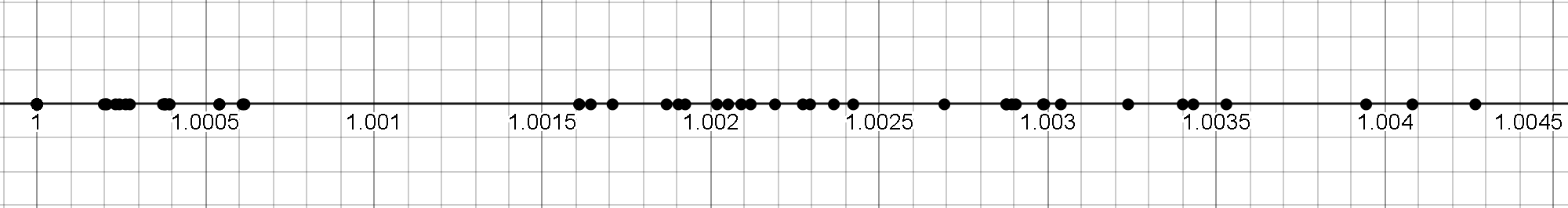}

\caption{Plotting $\gamma_2(G,a)$ for all $(G,a) \in \X_3$.}
\label{figX3gamma2}
\end{figure}

Using CVX+SeDuMi, we computed the values of $\gamma_2(G,a)$ for all 540 $(G,a) \in \X_3$, and plotted them on the number line as shown in Figure~\ref{figX3gamma2}. The $540$ data points can be categorized into the following three groups:
\begin{itemize}
\item
31 elements with $\gamma_2(G,a) > 1.0016$. They correspond exactly to $G_{\ref{fig3minimal},1}$, $G_{\ref{fig3minimal},2}$, $G_{\ref{fig3minimal},3}$, and their edge subgraphs listed in Figure~\ref{fig3minimalall}. These 31 graphs share the common property that they are all in $\K_{5,2}$ and have $\bar{e}^{\top}x \leq 3$ as the facet with $\LS_+$-rank three.
\item
18 elements with $1.00019 < \gamma_2(G,a) < 1.00062$. These correspond to $G_{\ref{fig3minimal},4}, \ldots ,G_{\ref{fig3minimal},8}$ and their edge subgraphs listed in Figure~\ref{fig3minimalall}. These 18 graphs do not belong to $\K_{5,2}$, and their facet with $\LS_+$-rank three is $(1,1,1,1,1,1,1,1,2)^{\top}x \leq 3$.
\item
491 elements with $\gamma_2(G,a) \leq 1 + 4\cdot 10^{-7}$. They all visually correspond to the dot at $1$ in Figure~\ref{figX3gamma2}. 
\end{itemize}

Thus, we see that aside from the 49 3-minimal graphs we previously described, every graph in our search satisfies $r_+(G) \lesssim 2$. This provides strong computational evidence that the list in Figure~\ref{fig3minimalall} is complete, which motivates the following conjecture.

\begin{conjecture}\label{conj3minimal}
There are exactly 49 non-isomorphic 3-minimal graphs.
\end{conjecture}

We conclude this section by presenting more computational findings for the graphs and facets in $\X_3$. Recall that we generated the collection $\X_3$ by exhaustively checking among a pool of candidate graphs for full-support facets. However, the approach of casting a wide net and searching exhaustively within may not be viable for 4-minimal graphs and beyond when both the number of candidate graphs and the time demand for each optimization problem increase substantially. Therefore, it is worthwhile to examine the numerical data from the 3-minimal search for patterns that may remain informative in higher-rank searches.

Thus, in addition to $\gamma_2(G,a)$, we also computed $\gamma_1(G,a)$ for every $(G,a) \in \X_3$, as we are interested in whether $\gamma_1(G,a)$ has some predictive value for $\gamma_p(G,a)$ for $p\geq 2$. If so, then $\gamma_1(G,a)$ (which is much less computationally costly to obtain) could serve as a valuable heuristic for identifying $\ell$-minimal graphs.

Figure~\ref{fig_3minX31} (left) gives the scatterplot of $\gamma_1(G,a)$ ($x$-axis) versus $\gamma_2(G,a)$  ($y$-axis) for each of the $540$ elements in $\X_3$. All $(G,a) \in \X_3$ satisfy  $1.0036 \leq \gamma_1(G,a) \leq 1.0608$. It is not surprising that these integrality ratios are all comfortably above $1$ since every graph analyzed here contains either $G_{\ref{figKnownEG},1}$ or $G_{\ref{figKnownEG},2}$ as an induced graph, and thus must have $\LS_+$-rank at least 2. A simple linear regression shows a very weak ($r \approx 0.0947$) positive correlation between $\gamma_1(G,a)$ and $\gamma_2(G,a)$.

\begin{figure}[htbp]
\centering
\begin{tabular}{cc}
\includegraphics[width=7cm]{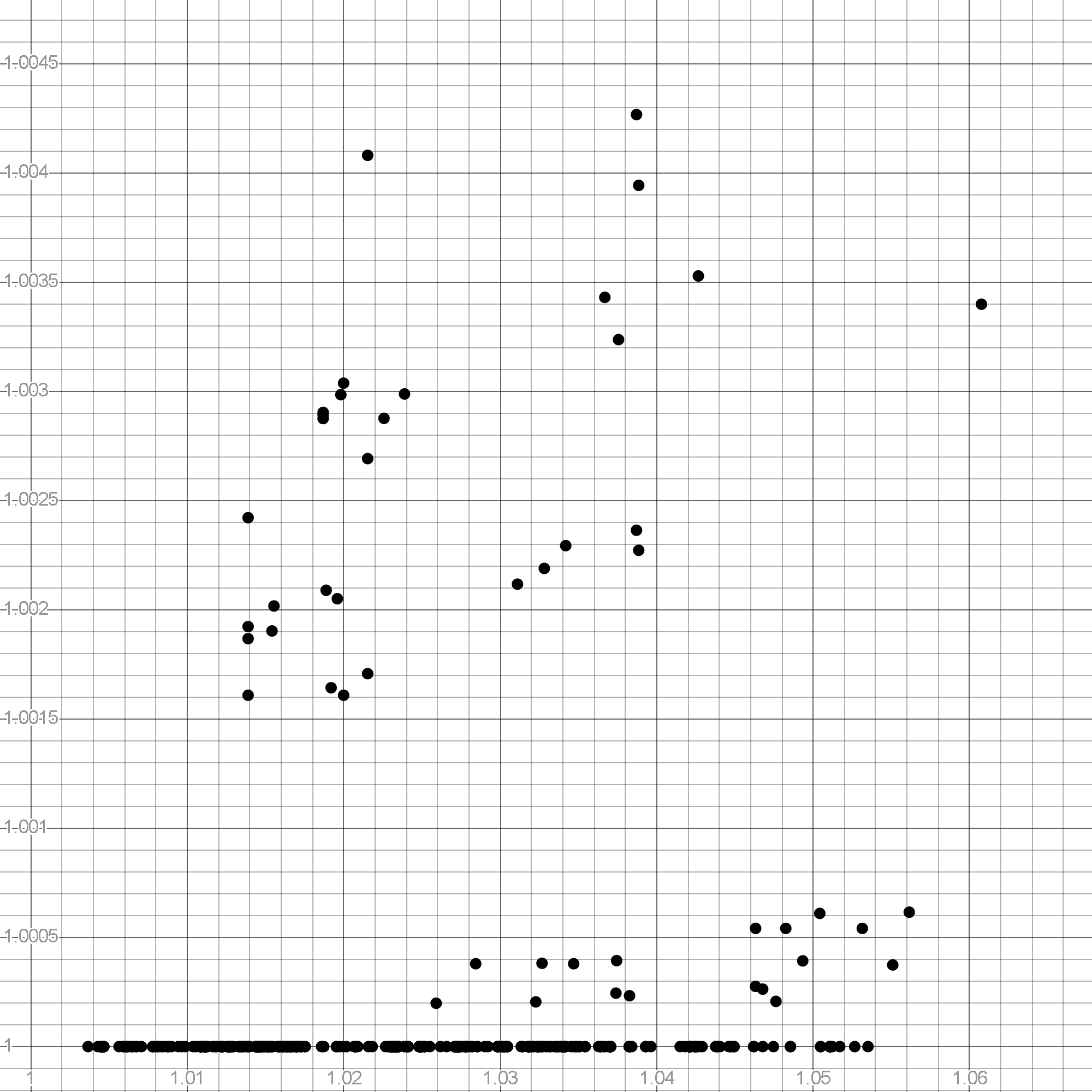} 
&
\includegraphics[width=7cm]{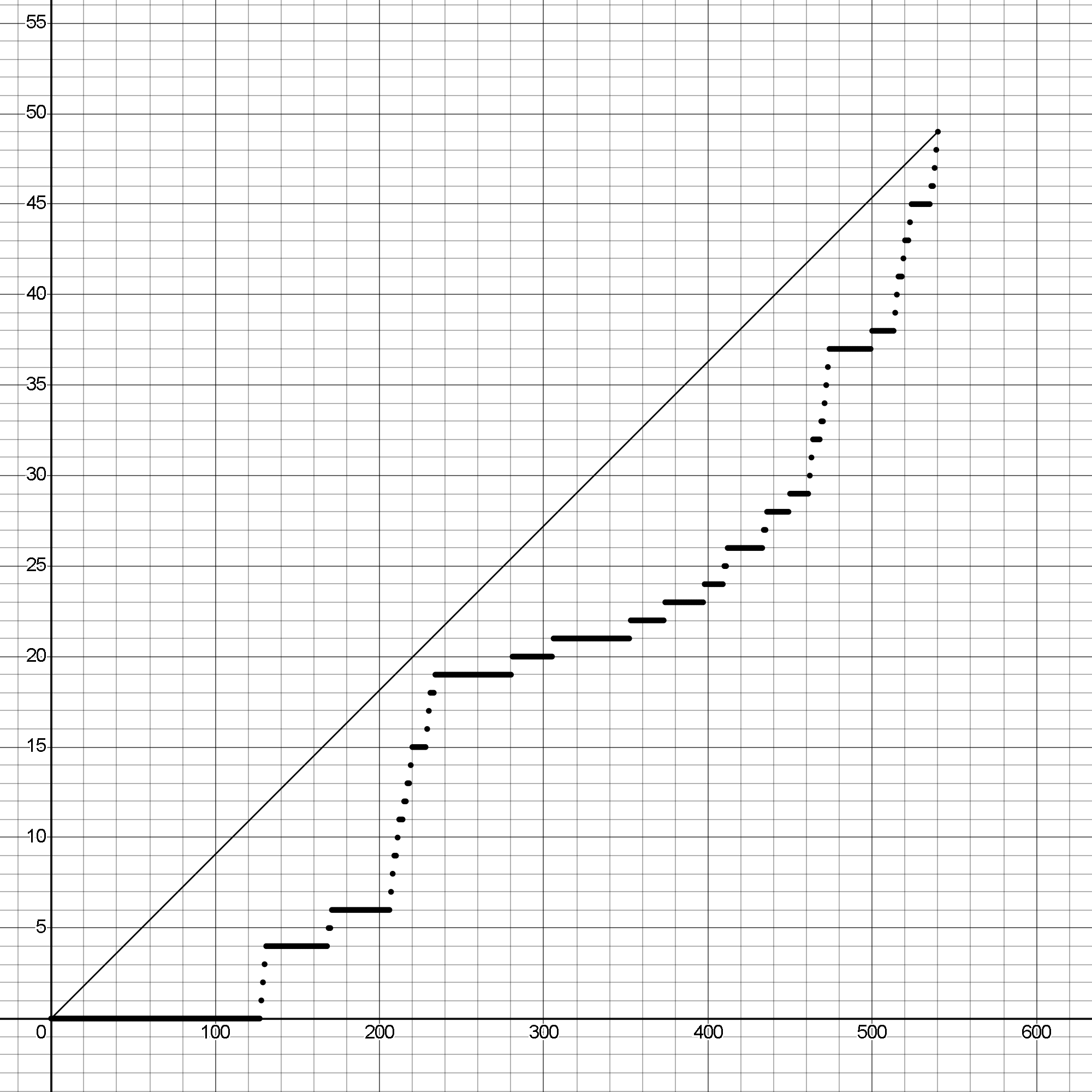} 
\end{tabular}
\caption{Relating $\gamma_1(G,a)$ and $\gamma_2(G,a)$ for elements in $\X_3$.}\label{fig_3minX31}

\end{figure}

Next, we analyze the relation between $\gamma_1(G,a)$ and $\gamma_2(G,a)$ from a different perspective. For every $n$ where $0 \leq n \leq 540$, we define $g(n)$ to be the number of the 49 3-minimal graphs listed in Figure~\ref{fig3minimalall} we would find among the $n$ facets with the lowest $\gamma_1(G,a)$. Obviously, $g(0) = 0$ and $g(540) = 49$, and Figure~\ref{fig_3minX31} (right) gives the plot of $g(n)$ for $0\leq n \leq 540$. Notice that the points $(n, g(n))$ all stay below the line $y = \frac{49}{540}x$, which indicates that the facets with higher $\gamma_1(G,a)$ indeed contain a higher concentration of facets with $\LS_+$-rank 3. In particular, notice that $g(n) = 0$ for all $n \leq 127$, which means that the 127 facets with the lowest values of $\gamma_1(G,a)$ all basically have $\gamma_2(G,a)=1$. 

\section{4-minimal graphs}\label{sec04}

We now turn to the case of 4-minimal graphs.

\subsection{Certified lower bounds and new 4-minimal graphs}

As with $\LS_+$ and $\LS_+^2$ certificate packages, here is an analogous framework for verifying the membership of points in $\LS_+^3(G)$. Given a graph $G$ with $n$ vertices, we define an \emph{$\LS_+^3$ certificate package} to be
\begin{itemize}
\item
A set of matrices $\M_2 \ce \set{Y_{e_ie_j}, Y_{e_if_j}, Y_{f_ie_j}, Y_{f_if_j} : i,j \in [n]}  \subseteq \mZ^{(n+1) \times (n+1)}$ such that, for every $M \in \M_2$,
\begin{itemize}
\item
$M = M^{\top}$ and $Me_0 = \diag(M)$;
\item
$Me_i, Mf_i \in \cone(\FRAC(G))$ for every $i \in [n]$.
\end{itemize}
\item
A set of matrices $\M_1 \ce \set{Y_{e_i}, Y_{f_i} : i \in [n] } \subseteq \mZ^{(n+1) \times (n+1)}$ such that, for every $M \in \M_1$,
\begin{itemize}
\item
$M = M^{\top}$ and $Me_0 = \diag(M)$;
\item
for every $i,j \in [n]$, 
\begin{itemize}
\item
$Y_{e_ie_j}e_0$ dominates $Y_{e_i}e_j$;
\item
$Y_{e_if_j}e_0$ dominates $Y_{e_i}f_j$;
\item
$Y_{f_ie_j}e_0$ dominates $Y_{f_i}e_j$;
\item
$Y_{f_if_j}e_0$ dominates $Y_{f_i}f_j$.
\end{itemize}
\end{itemize}
\item
A matrix $Y \in \mZ^{(n+1) \times (n+1)}$ where
\begin{itemize}
\item
$Y = Y^{\top}$ and $Ye_0 = \diag(Y)$;
\item
for every $i \in [n]$, 
\begin{itemize}
\item
$Y_{e_i}e_0$ dominates $Y e_i$;
\item
$Y_{f_i}e_0$ dominates $Y f_i$.
\end{itemize}
\end{itemize}
\item
A $UVW$-certificate for every non-zero matrix in $\M_2, \M_1$, and $Y$.
\end{itemize}

The conditions on $\M_2$ establish that $Me_0 \in \cone(\LS_+(G))$ for every $M \in \M_2$. Then the conditions imposed on $\M_1$ assure that $Me_0 \in \cone(\LS_+^2(G))$ for every $M \in \M_1$. Finally, the additional constraints on $Y$ ensure that $Ye_0 \in \cone(\LS_+^3(G))$. Now observe that $\M_2$ inherently contains many matrices of all zeros --- for instance, for every edge $\set{i,j} \in E(G)$, $\LS_+$ imposes the constraint  $[Ye_i]_j = 0$. Hence, $Y_{e_i}e_j$ is the zero vector, and consequently $Y_{e_ie_j}$ must be a matrix of all zeros. Matrices of all zeros are trivially PSD, and not including $UVW$-certificates for them helps somewhat reduce the size of these certificate packages. Thus, a $\LS_+^3$ certificate package for a vector in $\mR^n$ consists of up to $4(1+2n+4n^2)$ matrices (the certificate matrices $\M_2, \M_1$, and $Y$, as well as a $UVW$-certificate for each of these matrices which are non-zero). While verifying the validity of an $\LS_+^3$ certificate package generally requires checking a much greater number of matrices and conditions compared to $\LS_+$ and $\LS_+^2$ certificate packages, each condition can still be checked reliably as, again, they only depend on elementary arithmetic operations on integers. Using these ideas and following a similar proof to that of Proposition~\ref{prop:LS_+-certificate}, we have the following fact.

\begin{proposition}\label{prop:LS_+^3-certificate}
Let $G$ be a graph. Then  $r_+(G) \geq 4$ if and only if there exist a valid inequality $a^{\top}x \leq \beta$ for $\STAB(G)$
and an $\LS_+^3$ certificate package $(Y, \mathcal{M}_1, \mathcal{M}_2, \textup{ and } \textup{$UVW$-certificates})$ for $G$ such that $\left(-\beta, a^{\top}\right)Ye_0 > 0$.
\end{proposition}

Our computational search, combined with explicit $\LS_+^3$ certificate packages and edge-subgraph arguments, yields the following certified lower bound.

\begin{theorem}\label{thm4min}
There are at least 4,107 non-isomorphic 4-minimal graphs.
\end{theorem}

\begin{proof}
First, we are able to generate $\LS_+^3$ certificate packages for 570 graphs $G$, showing that there exists $\bar{x} \in \LS_+^3(G) \setminus \STAB(G)$ for these graphs. Next, Lemma~\ref{lem05subgraph} implies that many edge subgraphs of these 570 graphs are also 4-minimal. Collecting these edge subgraphs and checking for isomorphisms among them resulted in a total of 4,107 distinct 4-minimal graphs.
\end{proof}

The $\LS_+^3$ certificate packages, as well as a full list of 4-minimal graphs we found, are available at~\cite{AuT25data}. Our dataset also includes MATLAB code which can be used to verify all $\LS_+$, $\LS_+^2$, and $\LS_+^3$ certificate packages mentioned in this manuscript. Furthermore, all matrices in our data are stored as widely-supported CSV (comma-separated values) files, so one can also verify and analyze these certificate packages in other programming languages, such as Python and R.

\subsection{Structural consequences of the new 4-minimal graphs}

Next, we highlight a few certified 4-minimal graphs with notable features in Figure~\ref{fig4minimal}. For each graph, we list its graph6 encoding, the facet of its stable set polytope with $\LS_+$-rank four, and the integrality ratio of this inequality according to CVX+SeDuMi. (The reader may refer to~\cite{McKay} for a detailed description of graph6, an encoding of undirected graphs as strings of printable ASCII characters.)
\def\y{0.7}
\def\sc{1.6}
\def\z{360/ (6+6*\y)}
\def\x{270}

\def\placev{
\node[main node] at ({ \x+(-1 + 0*\y) *\z} : 1) (1) {$1$};
\node[main node] at ({ \x+(0 + 0*\y) *\z} : 1) (2) {$2$};
\node[main node] at ({ \x+(1 + 0*\y) *\z} : 1) (3) {$3$};
\node[main node] at ({ \x+(2 + 0*\y) *\z} : 1) (4) {$4$};
\node[main node] at ({ \x+(2 + 1*\y) *\z} : 1) (5) {$5$};
\node[main node] at ({ \x+(2 + 2*\y) *\z} : 1) (6) {$6$};
\node[main node] at ({ \x+(3 + 2*\y) *\z} : 1) (7) {$7$};
\node[main node] at ({ \x+(3 + 3*\y) *\z} : 1) (8) {$8$};
\node[main node] at ({ \x+(3 + 4*\y) *\z} : 1) (9) {$9$};
\node[main node] at ({ \x+(4 + 4*\y) *\z} : 1) (10) {$10$};
\node[main node] at ({ \x+(4 + 5*\y) *\z} : 1) (11) {$11$};
\node[main node] at ({ \x+(4 + 6*\y) *\z} : 1) (12) {$12$};
}

\begin{figure}[htbp]
\centering
\scriptsize{
\begin{tabular}{ccc}

\begin{tikzpicture}
[scale=\sc, thick,main node/.style={circle, minimum size=3.8mm, inner sep=0.1mm,draw,font=\tiny\sffamily}]
\placev
\path[every node/.style={font=\sffamily}]
(1) edge (2) (1) edge (3) (1) edge (6) (1) edge (7) (1) edge (10) (2) edge (3) (2) edge (6) (2) edge (9) (2) edge (12) (3) edge (4) (3) edge (7) (3) edge (9) (3) edge (10) (3) edge (12) (4) edge (5) (5) edge (6) (5) edge (7) (5) edge (9) (5) edge (10) (5) edge (12) (6) edge (7) (6) edge (9) (6) edge (12) (7) edge (8) (8) edge (9) (9) edge (10) (10) edge (11) (11) edge (12) ;
\end{tikzpicture}
&
\begin{tikzpicture}
[scale=\sc, thick,main node/.style={circle, minimum size=3.8mm, inner sep=0.1mm,draw,font=\tiny\sffamily}]
\placev
\path[every node/.style={font=\sffamily}]
(1) edge (2) (1) edge (3) (1) edge (4) (1) edge (6) (1) edge (9) (1) edge (10) (2) edge (3) (2) edge (6) (2) edge (7) (2) edge (12) (3) edge (4) (3) edge (7) (3) edge (9) (3) edge (12) (4) edge (5) (4) edge (7) (4) edge (10) (4) edge (12) (5) edge (6) (6) edge (7) (6) edge (10) (6) edge (12) (7) edge (8) (8) edge (9) (9) edge (10) (9) edge (12) (10) edge (11) (11) edge (12) ;
\end{tikzpicture}
&
\begin{tikzpicture}
[scale=\sc, thick,main node/.style={circle, minimum size=3.8mm, inner sep=0.1mm,draw,font=\tiny\sffamily}]
\placev
\path[every node/.style={font=\sffamily}]
(1) edge (2) (1) edge (3) (1) edge (4) (1) edge (9) (1) edge (10) (2) edge (3) (2) edge (6) (2) edge (7) (2) edge (12) (3) edge (4) (3) edge (9) (3) edge (12) (4) edge (5) (4) edge (7) (4) edge (12) (5) edge (6) (7) edge (8) (8) edge (9) (9) edge (12) (10) edge (11) (11) edge (12) ;
\end{tikzpicture}
\\
$G_{\ref{fig4minimal},1} \ce \verb+KxFLWDliG?l`+ $&
$G_{\ref{fig4minimal},2} \ce \verb+K|FJgE`dG?md+$ &
$G_{\ref{fig4minimal},3} \ce \verb+K|DI_E`_??mD+ $\\
$(1,1,1,1,1,1,1,1,1,1,1,1)^{\top}x \leq 4$ &
$(1,1,1,1,1,1,1,1,1,1,1,1)^{\top}x \leq 4$ &
$(1,1,1,1,1,1,1,1,1,1,1,1)^{\top}x \leq 4$ \\
$\gamma_3( G_{\ref{fig4minimal},1}, a) =1.0008772$ &
$\gamma_3( G_{\ref{fig4minimal},2}, a) =1.0007099$ &
$\gamma_3( G_{\ref{fig4minimal},3}, a) =1.0016810$ \\
\\
\begin{tikzpicture}
[scale=\sc, thick,main node/.style={circle, minimum size=3.8mm, inner sep=0.1mm,draw,font=\tiny\sffamily}]
\placev
\path[every node/.style={font=\sffamily}]
(1) edge (2) (1) edge (3) (1) edge (4) (1) edge (6) (1) edge (9) (1) edge (10) (2) edge (3) (2) edge (6) (2) edge (7) (2) edge (12) (3) edge (4) (3) edge (7) (3) edge (9) (3) edge (12) (4) edge (5) (4) edge (7) (4) edge (10) (4) edge (12) (5) edge (6) (5) edge (7) (5) edge (9) (6) edge (9) (6) edge (10) (6) edge (12) (7) edge (8) (8) edge (9) (9) edge (12) (10) edge (11) (11) edge (12) ;
\end{tikzpicture}
&
\begin{tikzpicture}
[scale=\sc, thick,main node/.style={circle, minimum size=3.8mm, inner sep=0.1mm,draw,font=\tiny\sffamily}]
\placev
\path[every node/.style={font=\sffamily}]
(1) edge (2) (1) edge (3) (1) edge (6) (1) edge (7) (1) edge (10) (2) edge (3) (2) edge (6) (2) edge (9) (2) edge (10) (2) edge (12) (3) edge (4) (3) edge (7) (3) edge (9) (3) edge (12) (4) edge (5) (5) edge (6) (5) edge (7) (5) edge (9) (6) edge (7) (6) edge (9) (6) edge (12) (7) edge (8) (7) edge (12) (8) edge (9) (8) edge (10) (8) edge (12) (9) edge (10) (10) edge (11) (11) edge (12) ;
\end{tikzpicture}
&
\begin{tikzpicture}
[scale=\sc, thick,main node/.style={circle, minimum size=3.8mm, inner sep=0.1mm,draw,font=\tiny\sffamily}]
\placev
\path[every node/.style={font=\sffamily}]
(1) edge (2) (1) edge (3) (1) edge (6) (1) edge (9) (1) edge (10) (2) edge (3) (2) edge (6) (2) edge (7) (2) edge (10) (3) edge (4) (3) edge (7) (3) edge (9) (3) edge (12) (4) edge (5) (4) edge (7) (5) edge (6) (5) edge (7) (6) edge (9) (6) edge (12) (7) edge (8) (7) edge (10) (8) edge (9) (8) edge (10) (9) edge (12) (10) edge (11) (11) edge (12) ;
\end{tikzpicture}
\\
$G_{\ref{fig4minimal},4} \ce \verb+K|FJoEld??md+$ &
$G_{\ref{fig4minimal},5} \ce \verb+KxFLWDloW?kx+$ &
$G_{\ref{fig4minimal},6} \ce \verb+KxFJoEdoo?cd+ $\\
$(1,1,1,1,1,1,2,1,1,1,1,1)^{\top}x \leq 4$ &
$(1,1,1,1,1,1,1,1,1,2,1,1)^{\top}x \leq 4$ &
$(1,1,1,1,1,1,2,1,1,2,1,1)^{\top}x \leq 4$ \\
$\gamma_3( G_{\ref{fig4minimal},4}, a) =1.0001388$ &
$\gamma_3( G_{\ref{fig4minimal},5}, a) =1.0001323$ &
$\gamma_3( G_{\ref{fig4minimal},6}, a) =1.0000859$ 
 \\
\\

\begin{tikzpicture}
[scale=\sc, thick,main node/.style={circle, minimum size=3.8mm, inner sep=0.1mm,draw,font=\tiny\sffamily}]
\placev
\path[every node/.style={font=\sffamily}]
(1) edge (2) (1) edge (3) (1) edge (6) (1) edge (7) (1) edge (10) (2) edge (3) (2) edge (6) (2) edge (9) (2) edge (10) (3) edge (4) (3) edge (9) (3) edge (10) (4) edge (5) (4) edge (10) (4) edge (12) (5) edge (6) (5) edge (7) (5) edge (9) (5) edge (10) (6) edge (9) (6) edge (12) (7) edge (8) (7) edge (10) (8) edge (9) (8) edge (12) (9) edge (10) (10) edge (11) (11) edge (12) ;
\end{tikzpicture}
&
\begin{tikzpicture}
[scale=\sc, thick,main node/.style={circle, minimum size=3.8mm, inner sep=0.1mm,draw,font=\tiny\sffamily}]
\placev
\path[every node/.style={font=\sffamily}]
(1) edge (2) (1) edge (3) (1) edge (6) (1) edge (7) (2) edge (3) (2) edge (6) (2) edge (9) (2) edge (10) (2) edge (12) (3) edge (4) (3) edge (7) (3) edge (9) (3) edge (12) (4) edge (5) (4) edge (10) (5) edge (6) (5) edge (9) (5) edge (12) (6) edge (9) (7) edge (8) (7) edge (12) (8) edge (9) (8) edge (10) (9) edge (12) (10) edge (11) (11) edge (12) ;
\end{tikzpicture}
&
\begin{tikzpicture}
[scale=\sc, thick,main node/.style={circle, minimum size=3.8mm, inner sep=0.1mm,draw,font=\tiny\sffamily}]
\placev
\path[every node/.style={font=\sffamily}]
(1) edge (2) (1) edge (3) (1) edge (6) (1) edge (9) (1) edge (10) (2) edge (3) (2) edge (6) (2) edge (7) (2) edge (12) (3) edge (4) (3) edge (9) (3) edge (10) (4) edge (5) (4) edge (7) (4) edge (10) (5) edge (6) (5) edge (7) (5) edge (10) (6) edge (7) (6) edge (12) (7) edge (8) (7) edge (10) (8) edge (9) (10) edge (11) (11) edge (12) ;
\end{tikzpicture}
\\
$G_{\ref{fig4minimal},7} \ce \verb+KxFKODl}g?ah+$ &
$G_{\ref{fig4minimal},8} \ce \verb+KxFL?DlSO?lT+$ &
$G_{\ref{fig4minimal},9} \ce \verb+KxFIwE`m_?g`+$ \\
$(1,1,1,1,1,1,1,1,1,2,1,2)^{\top}x \leq 4$ &
$(1,1,1,1,1,1,1,1,1,2,1,1)^{\top}x \leq 4$ &
$(1,1,1,1,1,1,2,1,1,2,1,1)^{\top}x \leq 4$ \\
$\gamma_3( G_{\ref{fig4minimal},7}, a) =1.0001203$ &
$\gamma_3( G_{\ref{fig4minimal},8}, a) =1.0003666$ &
$\gamma_3( G_{\ref{fig4minimal},9}, a) =1.0000789$ 
\\
\\

\begin{tikzpicture}
[scale=\sc, thick,main node/.style={circle, minimum size=3.8mm, inner sep=0.1mm,draw,font=\tiny\sffamily}]
\placev
\path[every node/.style={font=\sffamily}]
(1) edge (2) (1) edge (3) (1) edge (6) (1) edge (7) (1) edge (10) (2) edge (3) (2) edge (6) (2) edge (9) (3) edge (4) (3) edge (9) (4) edge (5) (4) edge (12) (5) edge (6) (5) edge (7) (5) edge (9) (5) edge (10) (6) edge (9) (6) edge (12) (7) edge (8) (7) edge (10) (8) edge (9) (8) edge (12) (10) edge (11) (11) edge (12) ;
\end{tikzpicture}
&
\begin{tikzpicture}
[scale=\sc, thick,main node/.style={circle, minimum size=3.8mm, inner sep=0.1mm,draw,font=\tiny\sffamily}]
\placev
\path[every node/.style={font=\sffamily}]
(1) edge (2) (1) edge (3) (1) edge (6) (1) edge (7) (2) edge (3) (2) edge (6) (2) edge (9) (2) edge (12) (3) edge (4) (3) edge (7) (3) edge (10) (4) edge (5) (4) edge (12) (5) edge (6) (5) edge (9) (5) edge (10) (6) edge (9) (7) edge (8) (7) edge (10) (8) edge (9) (8) edge (12) (9) edge (10) (10) edge (11) (11) edge (12) ;
\end{tikzpicture}
&
\begin{tikzpicture}
[scale=\sc, thick,main node/.style={circle, minimum size=3.8mm, inner sep=0.1mm,draw,font=\tiny\sffamily}]
\placev
\path[every node/.style={font=\sffamily}]
(1) edge (2) (1) edge (3) (1) edge (6) (1) edge (7) (1) edge (10) (2) edge (3) (2) edge (6) (2) edge (9) (3) edge (4) (3) edge (9) (3) edge (12) (4) edge (5) (4) edge (10) (5) edge (6) (5) edge (7) (5) edge (12) (6) edge (7) (6) edge (9) (7) edge (8) (7) edge (12) (8) edge (9) (8) edge (10) (10) edge (11) (11) edge (12) ;
\end{tikzpicture}
\\
$G_{\ref{fig4minimal},10} \ce \verb+KxFKODla_?ah+$ &
$G_{\ref{fig4minimal},11} \ce \verb+KxFL?DLIg?iH+$ &
$G_{\ref{fig4minimal},12} \ce \verb+KxFKWDdcO?dP+$ \\
$(1,1,1,1,1,1,1,1,1,1,1,2)^{\top}x \leq 4$ &
$(1,1,1,1,1,1,1,1,1,1,1,2)^{\top}x \leq 4$ &
$(1,1,1,1,1,1,1,1,1,2,1,1)^{\top}x \leq 4$ \\
$\gamma_3( G_{\ref{fig4minimal},10}, a) =1.0003558$ &
$\gamma_3( G_{\ref{fig4minimal},11}, a) =1.0003679$ &
$\gamma_3( G_{\ref{fig4minimal},12}, a) =1.0002711$ 
\\
\\

\end{tabular}}
\caption{A sample of twelve 4-minimal graphs with notable features.}\label{fig4minimal}
\end{figure}
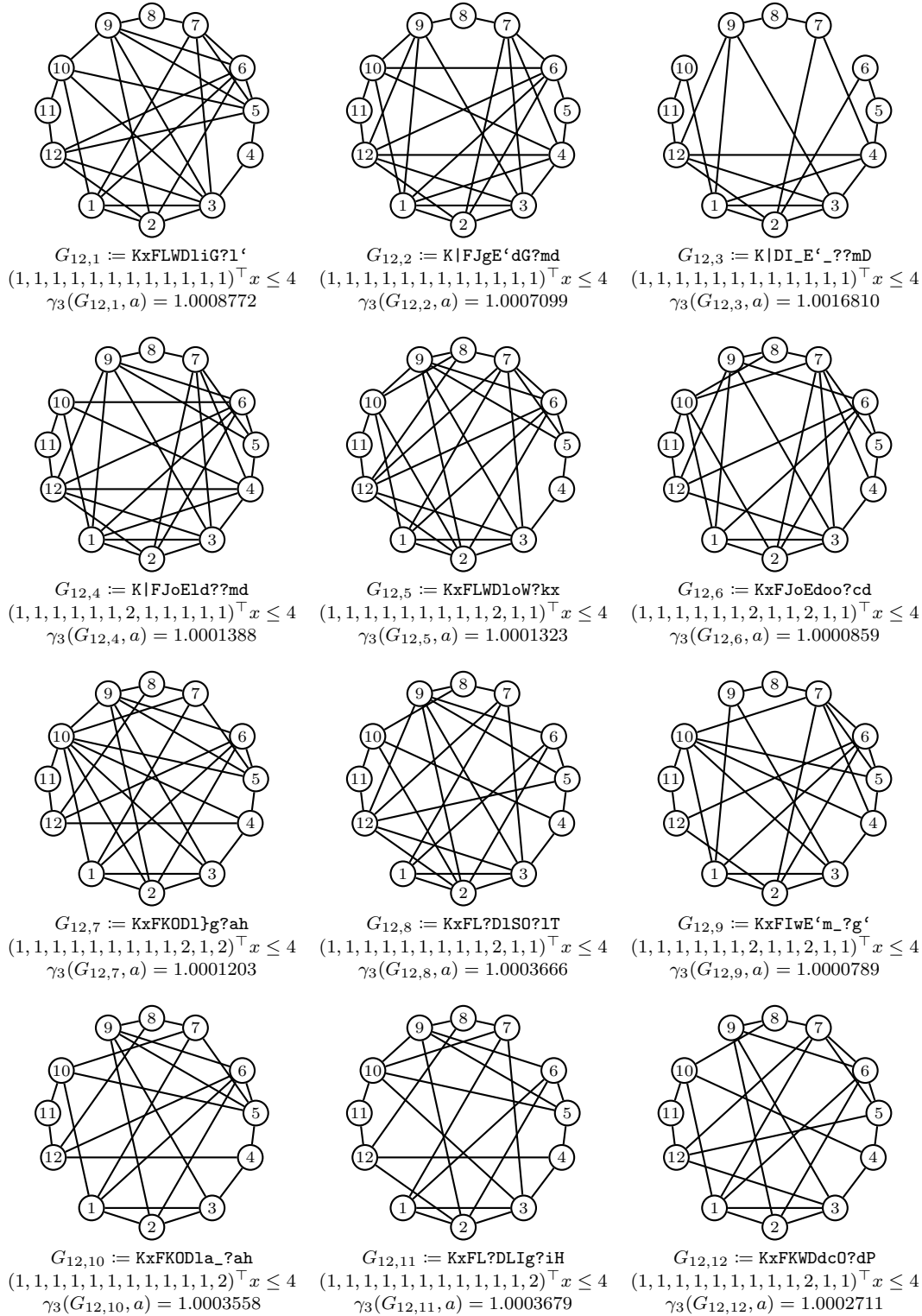

Among the 4,107 4-minimal graphs we found, 2,318 belong to $\K_{6,3}$, including the aforementioned 588 graphs in $\hat{\K}_{6,3}$ with clique number at most three. Thus, in addition to the analytical proof from~\cite{AuT25}, we now have an independent certificate-based proof that every graph $G \in \hat{\K}_{6,3}$ with $\omega(G) \leq 3$ is indeed 4-minimal. The densest 4-minimal stretched clique we found has 28 edges --- $G_{\ref{fig4minimal},1}$ and  $G_{\ref{fig4minimal},2}$ are two such examples. 

The densest 4-minimal graphs we discovered overall have 29 edges --- $G_{\ref{fig4minimal},4}$ and $G_{\ref{fig4minimal},5}$ give two such instances. The sparsest 4-minimal graphs we found have 21 edges; there are 40 such graphs, all sparse stretched cliques in $\hat{\K}_{6,3}$. One of these graphs is $G_{\ref{fig4minimal},3}$, which has the notable feature that its facet of $\LS_+$-rank four has the largest value of $\gamma_3(G,a)$  among all 12-vertex graphs we tested. These recurring structural patterns motivate the following conjectures.

\begin{conjecture}\label{conj02}
For every positive integer $\ell$, a sparsest $\ell$-minimal graph is a sparse stretched clique in $\hat{\K}_{\ell+2, \ell-1}$.
\end{conjecture}

\begin{conjecture}\label{conj02b}
For every positive integer $\ell$, the maximum value of $\gamma_{\ell-1} (G,a)$ among $\ell$-minimal graphs is attained by a sparse stretched clique $G \in \hat{\K}_{\ell+2, \ell-1}$ and $a \ce \bar{e}$.
\end{conjecture}

Next, $G_{\ref{fig4minimal},7}$, $G_{\ref{fig4minimal},8}$, and $G_{\ref{fig4minimal},9}$ give examples of 4-minimal graphs which contain $K_4$ as an induced subgraph. In particular, both $\set{1,2,3,10}$ and $\set{2,3,9,10}$ induce a $K_4$ in $G_{\ref{fig4minimal},7}$, which is the only 4-minimal graph we found with multiple induced $K_4$ subgraphs. Also, notice that these graphs do not belong to $\K_{6,3}$.

In other words, as with 3-minimal graphs, we did not find any 4-minimal graphs in $\K_{6,3}$ where $\omega(G) \geq 4$. In fact, one can exhaustively search and find that there are exactly 121 non-isomorphic sparse stretched cliques $G \in \hat{\K}_{6,3}$ where $\bar{e}^{\top}x \leq 4$ is a facet of $\STAB(G)$. 40 of them have $\omega(G) = 3$, all of which are among graphs which have been shown to be 4-minimal in Theorem~\ref{thm4min}. The remaining 81 all have $\omega(G) \geq 4$, and one can use similar ideas as in the proof of Proposition~\ref{propK52K4} to show that all 81 graphs have $\LS_+$-rank at most 3. Thus, we can use the same argument for Proposition~\ref{propK52K4} to show that $\bar{e}^{\top}x \leq 4$ has $\LS_+$-rank at most 3 for every graph $G \in \K_{6,3}$ with $\omega(G) \geq 4$. These observations motivate the following conjecture.

\begin{conjecture}\label{conj03}
For every positive integer $\ell$, if $G \in \K_{\ell+2,\ell-1}$ and $\omega(G) \geq 4$, then $r_+(G) \leq \ell-1$.
\end{conjecture}

If Conjecture~\ref{conj03} holds, then combining it with Theorem~\ref{thmAuT251} gives that a graph $G \in \hat{\K}_{\ell+2,\ell-1}$ is $\ell$-minimal if and only if $\omega(G) \leq 3$.

Next, recall our discussion around Corollary~\ref{corJoinStretch} about 3-minimal graphs where the only way to obtain these graphs from $K_3$ using 1-join and 2-stretch operations is to alternate between these two operations. We found that there are also many such examples for 4-minimal graphs --- $G_{\ref{fig4minimal},6}$ and $G_{\ref{fig4minimal},9}$ are two such instances.

Finally, recall that we mentioned earlier that each of the 49 3-minimal graphs in Figure~\ref{fig3minimalall} contains a stretched clique in $\K_{5,2}$ as an edge subgraph. This is no longer true for 4-minimal graphs, as $G_{\ref{fig4minimal},10}$, $G_{\ref{fig4minimal},11}$, and $G_{\ref{fig4minimal},12}$ provide instances which do not contain a stretched clique in $\K_{6,3}$ as an edge subgraph. In particular, $G_{\ref{fig4minimal},12}$ does not even contain $K_6$ as a graph minor, giving the first example of an $\ell$-minimal graph which does not contain $K_{\ell+2}$ as a graph minor.

\subsection{Numerical evidence and heuristic observations from the 4-minimal search}

Next, we go into more detail about our approach for searching for 4-minimal graphs. First, we used the criteria described in Theorem~\ref{thmAuT24b1} to construct a pool of viable candidates for being 4-minimal graphs. Let $\X_4$ denote the set of ordered pairs $(G,a)$ with the following properties:

\begin{itemize}
\item
$V(G) = [12]$, and the vertices in $[9] \subset V(G)$ induce one of the $49$ 3-minimal graphs listed in Figure~\ref{fig3minimalall}.
\item
$\deg(11) = 2$, with $\Gamma_G(11) = \set{10,12}$. Also, $\Gamma_G(10)$ and $\Gamma_G(12)$ are not subsets of each other. (This is to ensure that $G$ can be obtained from a proper 2-stretching of another graph.)
\item
$a$ is the direction of a full-support facet of $G$ using the criterion described in~\cite[Corollary 15]{AuT24b}.
\end{itemize}

An exhaustive search found that $|\X_4| = 2{,}038{,}174$ after eliminating redundant isomorphic graphs. A straightforward approach would then be to compute $\gamma_3(G,a)$ for each of these elements, and then proceed to generate $\LS_+^3$ certificate packages for graphs with integrality ratios comfortably above $1$. However, optimizing over (a straightforward formulation of) $\LS_+^3(G)$ with CVX+SeDuMi for a 12-vertex graph $G$ takes 8--10 minutes per instance on our machine, which means that computing $\gamma_3(G,a)$ for all elements in $\X_4$ this way would take more than 30 years. Thus, we need additional insights to narrow down our search.

Therefore, let us focus on a particular subset of $\X_4$ by imposing the edge subgraph partial order on this set: Given $(G,a), (G',a') \in \X_4$, we define $(G,a) \leq (G',a')$ if $a =a'$ and $G$ is an edge subgraph of $G'$. In this case, it follows from Lemma~\ref{lem05subgraph} that $r_+(G') = 4$ implies $r_+(G) = 4$. Thus, let $\bar{\X}_4$ be the set of elements $(G,a) \in \X_4$ which are minimal with respect to this partial order. Then we know that every 4-minimal graph contains at least one graph in $\bar{\X}_4$ as an edge subgraph. Also, we have $|\bar{\X}_4| = 6{,}822$, which is a much more manageable set to work with. Figure~\ref{figX4gamma3} shows the plot of the values of $\gamma_3(G,a)$ on the real number line for all $(G,a) \in \bar{\X}_4$.

\begin{figure}[htbp]
\centering
\includegraphics[width=15cm]{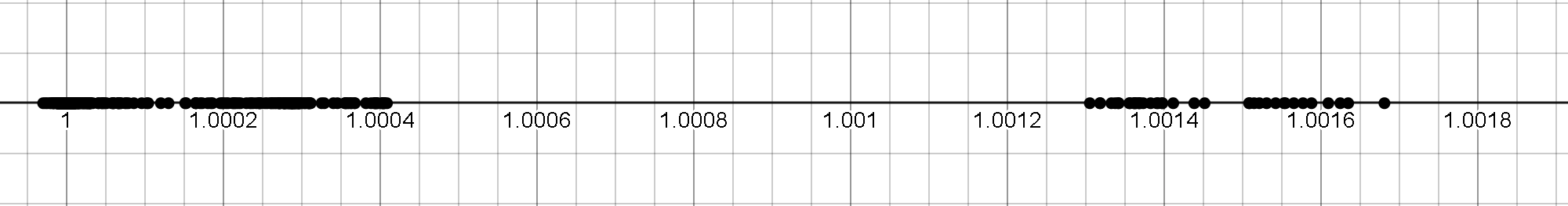}

\caption{Plotting $\gamma_3(G,a)$ for all $(G,a) \in \bar{\X}_4$.}
\label{figX4gamma3}
\end{figure}

Next, observe that for every $(G,a) \in \X_4$,
\begin{equation}\label{eqgammabound}
\gamma_3(G,a) \leq \max_{ (G',a) \in \bar{\X}_4} \set{ \gamma_3(G',a) : (G',a) \leq (G,a)}.
\end{equation}
Using~\eqref{eqgammabound}, we obtain an upper bound on $\gamma_3(G,a)$ for every $(G,a) \in \X_4$. Then we selectively computed the actual value of $\gamma_3(G,a)$ for elements which are promising, resulting in the collection of 4-minimal graphs described in Theorem~\ref{thm4min}. With this reduction, we ended up computing $\gamma_3(G,a)$ for around 16,000 graphs in $\X_4$.  At 8--10 minutes per graph, solving these 16,000 SDPs alone took CVX+SeDuMi more than 2,000 hours of computation time.

We next examine the numerical distribution in Figure~\ref{figX4gamma3}. First, notice that there is a cluster of values above $1.001$ --- they consist of 40 elements in $\bar{\X}_4$, which are exactly the aforementioned 40 sparse stretched cliques in $\K_{6,3}$ with clique number at most 3. In particular, the rightmost dot in Figure~\ref{figX4gamma3} represents the integrality ratio of $G_{\ref{fig4minimal},3}$.

In addition to these 40 sparse stretched cliques, there are another 134 graphs in $\bar{\X}_4$ which are certified to be 4-minimal in Theorem~\ref{thm4min}, with integrality ratios between $1.00004$ and $1.0005$. Unlike the case in Figure~\ref{figX3gamma2}, there is not a clean visual break indicating which graphs ``apparently'' have integrality ratios above $1$, and so we suspect that there could be more than 174 4-minimal graphs among elements of $\bar{\X}_4$.

Recall that in our analysis on $\X_3$, we examined the possible predictive value of $\gamma_1(G,a)$  (computationally easier to obtain) for $\gamma_2(G,a)$ (the quantity directly relevant to whether a graph has $\LS_+$-rank $3$). 
Likewise, it is natural to wonder if $\gamma_1(G,a)$ and/or $\gamma_2(G,a)$ can serve as a heuristic in our search for 4-minimal graphs. In particular, between $\gamma_1(G,a)$ and $\gamma_2(G,a)$, which has better predictive value for $\gamma_3(G,a)$? To that end, we plotted $\gamma_1(G,a)$ against $\gamma_3(G,a)$ (Figure~\ref{fig_4minX41}, left), as well as $\gamma_2(G,a)$ against $\gamma_3(G,a)$ (Figure~\ref{fig_4minX41}, right). A simple linear regression shows that $\gamma_2(G,a)$ ($r \approx 0.3158$) indeed has a stronger correlation with $\gamma_3(G,a)$ compared to $\gamma_1(G,a)$ ($r \approx 0.03976$) for this particular set of graphs and facets. This is not surprising, as $\gamma_2(G,a)$ (3--4 seconds per instance) is more computationally costly to obtain than $\gamma_1(G,a)$ (roughly 0.25 seconds per instance) and should reveal more information about the underlying graph. In particular, as the clusters of points on the lower-left corner of both scatterplot suggest, graphs with the lowest $\gamma_1(G,a)$ and $\gamma_2(G,a)$ values all seem to have $\gamma_3(G,a)$ very close to $1$.

\begin{figure}[htbp]
\centering
\begin{tabular}{cc}
\includegraphics[width=7cm]{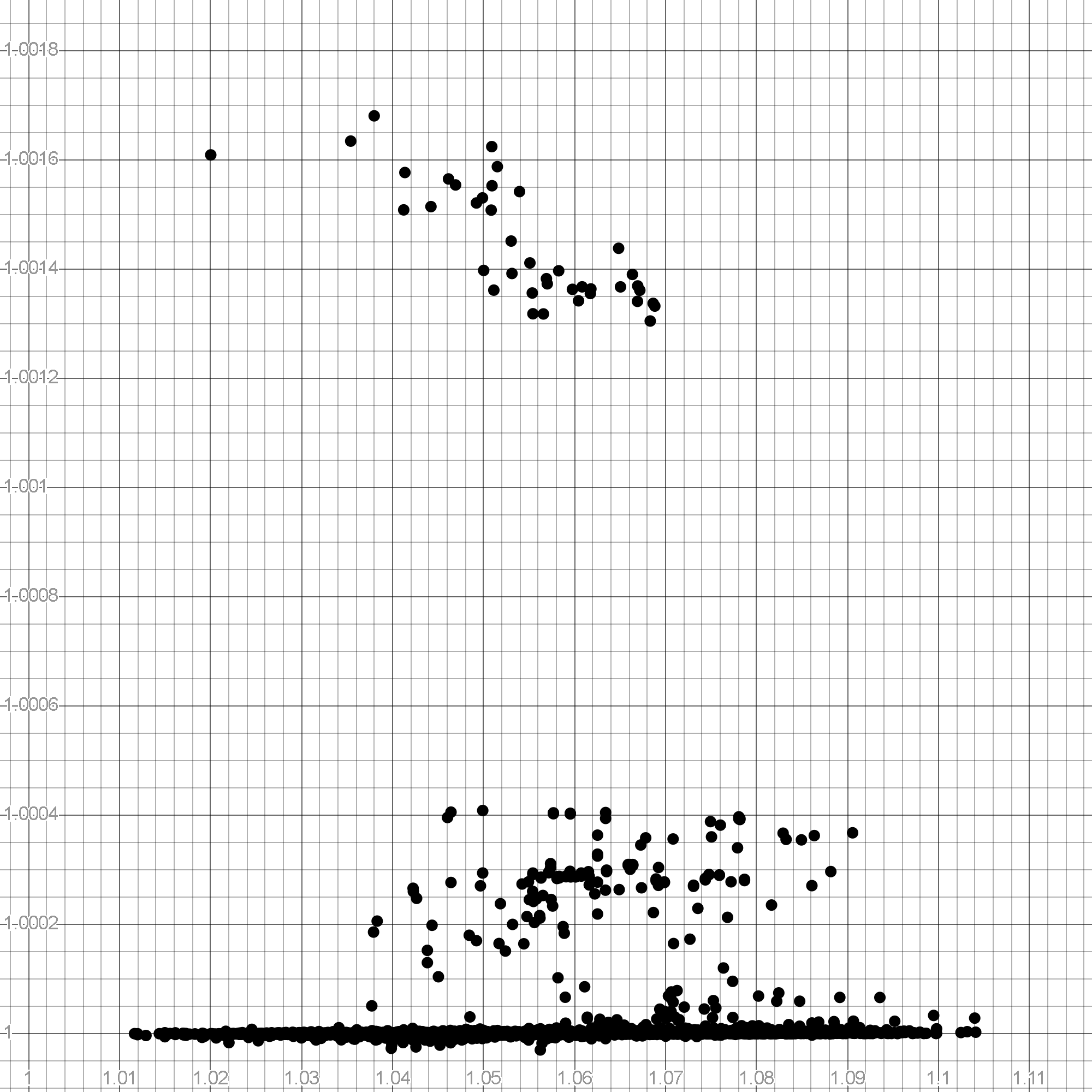} 
&
\includegraphics[width=7cm]{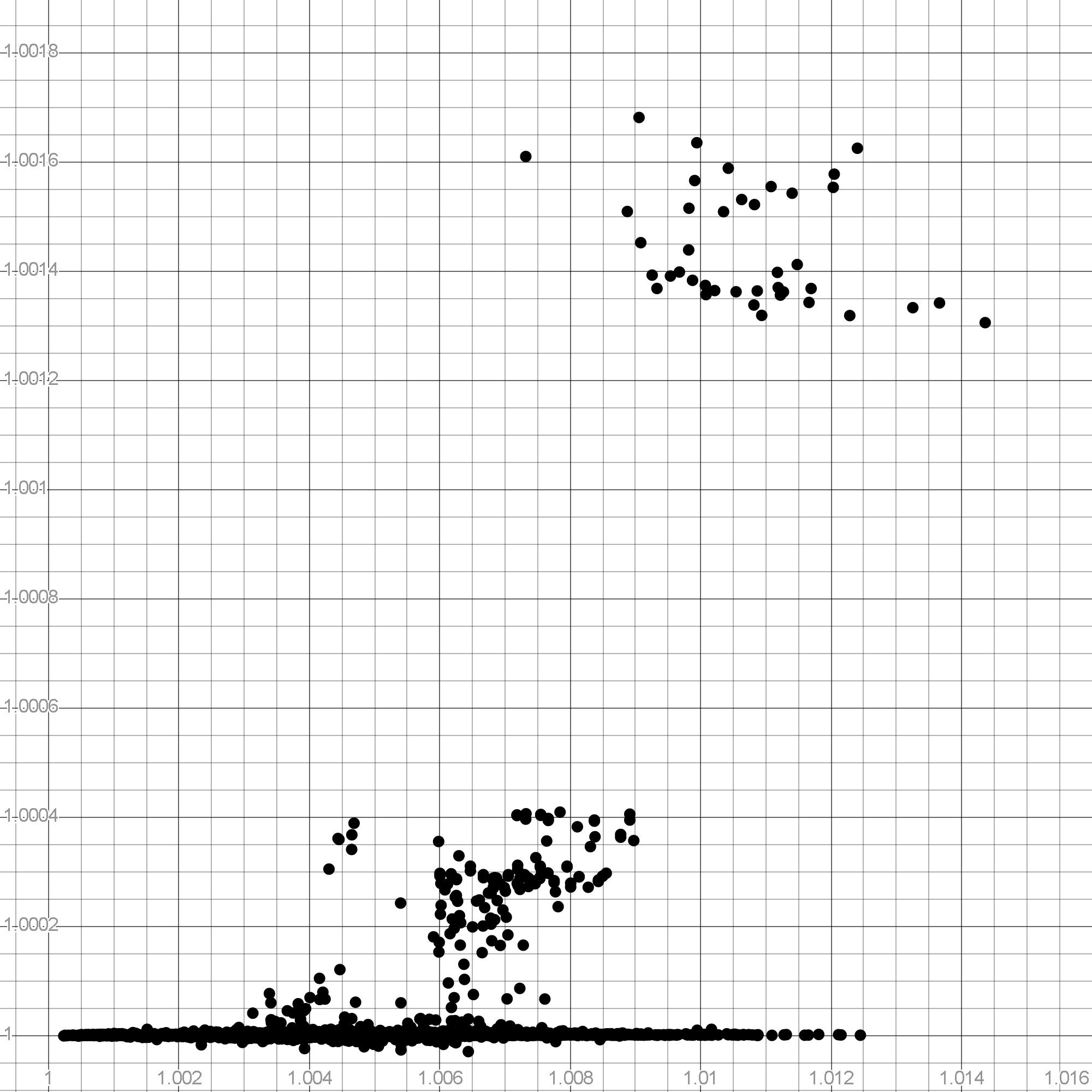} 
\end{tabular}
\caption{Scatterplots of $\gamma_1(G,a)$ versus $\gamma_3(G,a)$ (left) and $\gamma_2(G,a)$ versus $\gamma_3(G,a)$ (right) for all $(G,a) \in \bar{\X}_4$.}\label{fig_4minX41}

\end{figure}

\section{Vertex-transitive graphs}\label{sec05}

In this section, we look into the smallest vertex-transitive graphs with a given $\LS_+$-rank. As we saw, $\ell$-minimal graphs must contain some (but not all) vertices with degree 2, making them necessarily irregular. Thus, restricting ourselves to highly-symmetric graphs may expose new structures that make instances of the stable set problem challenging for $\LS_+$.

\subsection{General setup and preliminary bounds}

Given $\ell \in \mN$, let $\overline{n}_+(\ell)$ denote the minimum number of vertices among vertex-transitive graphs with $\LS_+$-rank exactly $\ell$. Theorem~\ref{thmLiptakT031} implies that $\overline{n}_+(\ell) \geq 3\ell$ for every $\ell \in \mN$, and the aforementioned result by Stephen and the second author~\cite{StephenT99} on the line graphs of odd cliques show that $\overline{n}_+(\ell) \leq \binom{2\ell+1}{2} = 2\ell^2+\ell$ for every $\ell \in \mN$. Both bounds are tight for $\ell =1$, as $\overline{n}_+(1) = 3$ is attained by $K_3$.

Applying Theorem~\ref{thmAuT251} and then adding vertices to symmetrize stretched cliques yields the following existence result for vertex-transitive graphs of high $\LS_+$-rank.

\begin{theorem}\label{thmAuT253}
\cite[Proposition 26]{AuT25} For every odd integer $\ell \geq 1$, there exists a vertex-transitive graph $G$ where $|V(G)| \leq 4\ell+8$ and $r_+(G) \geq \ell$.
\end{theorem}

Thus,
\[
\min \set{ \overline{n}_+(k) : k \geq \ell } \leq 4\ell + 8
\]
for every odd $\ell \in \mN$. Since the underlying construction only gives a lower bound on the rank of those graphs, this does not by itself determine $\overline{n}_+(\ell)$ for a specific $\ell$.

Next, given a graph $G$ and $\ell \in \mN$, define
\[
\a_{\LS_+^{\ell}}(G) \ce \max\set{ \bar{e}^{\top}x : x \in \LS_+^{\ell}(G)}.
\]
To show that $r_+(G) > \ell$, it is sufficient to show that $\a_{\LS_+^{\ell}}(G) > \a(G)$. The following result will also be helpful.

\begin{lemma}\label{lemnkl}
Suppose $G$ is a $k$-regular vertex-transitive graph with $n$ vertices and $r_+(G) = \ell \geq 2$. Then $3 \leq k \leq n +2 - 3\ell$.
\end{lemma}

\begin{proof}
If $k \leq 2$, then each component in $G$ is either a single vertex ($k=0$), an edge ($k=1$), or a cycle ($k=2$), and $r_+(G) \leq 1$ in all of these cases. Thus, it follows that $k \geq 3$. Next, since $G$ is vertex-transitive, $G \ominus i$ is isomorphic for all $i \in V(G)$, and so $r_+(G \ominus i) \geq \ell -1$ (due to Theorem~\ref{thmDeleteDestroy}), and thus $|V(G \ominus i)| \geq 3\ell -3$. Since the destruction of $i$ removes $k+1$ vertices from $G$, this shows that $n \geq 3\ell + k -2$, and the claim follows.
\end{proof}

A well-studied family of vertex-transitive graphs which we will frequently refer to is the circulant graphs. Given $S \subseteq [n]$, we define the \emph{circulant graph} $C_{n}^S$ where $V(C_{n}^S) \ce [n]$ and
\[
E(C_n^S) \ce \set{ \set{i,j} :~\tn{$(j-i)$ mod $n \in S$ or $(i-j)$ mod $n \in S$}}.
\]

\subsection{Exact results for ranks $2$, $3$, and $4$}

Next, we will show that the graphs $G_{\ref{figVertexTransitive234},1}$, $G_{\ref{figVertexTransitive234},2}$, and $G_{\ref{figVertexTransitive234},3}$ from Figure~\ref{figVertexTransitive234} are the smallest vertex-transitive graphs with $\LS_+$-rank 2, 3, and 4, respectively. This shows that $\overline{n}_+(2) = 8$, $\overline{n}_+(3) = 13$, and $\overline{n}_+(4) = 16$. Notice that $G_{\ref{figVertexTransitive234},1} = C_8^{\set{1,2}}$ and $G_{\ref{figVertexTransitive234},2} = C_{13}^{\set{1,5}}$ are both circulant graphs. For $G_{\ref{figVertexTransitive234},3}$, one way to make sense of the graph is to observe that it consists of a copy of $C_8^{\set{1,4}}$ (induced by vertices $\set{1,\ldots,8}$) and a copy of $C_8^{\set{3,4}}$ (induced by vertices $\set{9,\ldots,16}$) joined together by a perfect matching.

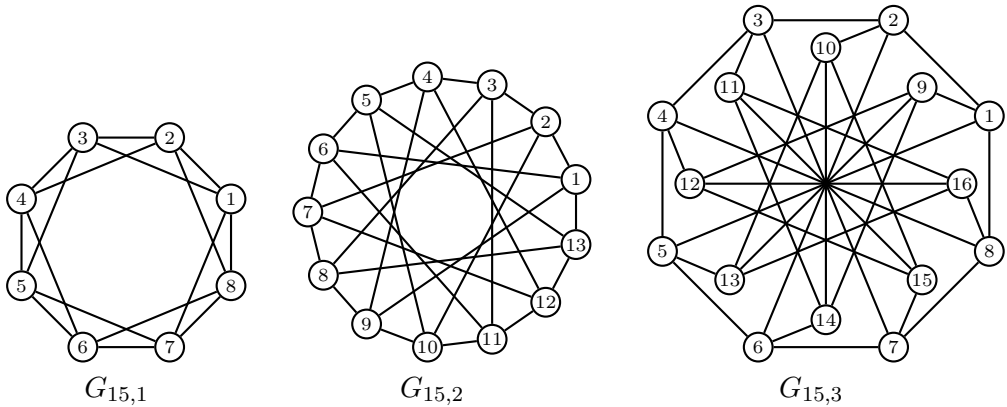
\begin{figure}[htbp]
\centering
\begin{tabular}{ccc}
\def\z{-360/16}
\def\x{360/8}
\begin{tikzpicture}
[scale=1.5, thick,main node/.style={circle, minimum size=3.8mm, inner sep=0.1mm,draw,font=\tiny\sffamily}]
\node[main node] at ({cos((-1)*\x+\z)},{sin((-1)*\x+\z)}) (7) {$7$};
\node[main node] at ({cos(0*\x+\z)},{sin((0)*\x+\z)}) (8) {$8$};
\node[main node] at ({cos(1*\x+\z)},{sin((1)*\x+\z)}) (1) {$1$};
\node[main node] at ({cos((2)*\x+\z)},{sin((2)*\x+\z)}) (2) {$2$};
\node[main node] at ({cos(3*\x+\z)},{sin((3)*\x+\z)}) (3) {$3$};
\node[main node] at ({cos(4*\x+\z)},{sin((4)*\x+\z)}) (4) {$4$};
\node[main node] at ({cos((5)*\x+\z)},{sin((5)*\x+\z)}) (5) {$5$};
\node[main node] at ({cos((6)*\x+\z)},{sin((6)*\x+\z)}) (6) {$6$};

  \path[every node/.style={font=\sffamily}]
(1) edge (2)
(1) edge (3)
(2) edge (3)
(2) edge (4)
(3) edge (4)
(3) edge (5)
(4) edge (5)
(4) edge (6)
(5) edge (6)
(5) edge (7)
(6) edge (7)
(6) edge (8)
(7) edge (8)
(7) edge (1)
(8) edge (1)
(8) edge (2);
\end{tikzpicture}
&

\def\z{-360/26}
\def\x{360/13}
\begin{tikzpicture}
[scale=1.8, thick,main node/.style={circle, minimum size=3.8mm, inner sep=0.1mm,draw,font=\tiny\sffamily}]
\node[main node] at ({ cos((0)*\x+\z)},{ sin((0)*\x+\z)}) (13) {$13$};
\node[main node] at ({ cos(1*\x+\z)},{ sin((1)*\x+\z)}) (1) {$1$};
\node[main node] at ({ cos(2*\x+\z)},{ sin((2)*\x+\z)}) (2) {$2$};
\node[main node] at ({ cos((3)*\x+\z)},{ sin((3)*\x+\z)}) (3) {$3$};
\node[main node] at ({ cos(4*\x+\z)},{ sin((4)*\x+\z)}) (4) {$4$};
\node[main node] at ({ cos(5*\x+\z)},{ sin((5)*\x+\z)}) (5) {$5$};
\node[main node] at ({ cos((6)*\x+\z)},{ sin((6)*\x+\z)}) (6) {$6$};
\node[main node] at ({ cos((7)*\x+\z)},{ sin((7)*\x+\z)}) (7) {$7$};
\node[main node] at ({ cos((8)*\x+\z)},{ sin((8)*\x+\z)}) (8) {$8$};
\node[main node] at ({ cos((9)*\x+\z)},{ sin((9)*\x+\z)}) (9) {$9$};
\node[main node] at ({ cos((10)*\x+\z)},{ sin((10)*\x+\z)}) (10) {$10$};
\node[main node] at ({ cos((11)*\x+\z)},{ sin((11)*\x+\z)}) (11) {$11$};
\node[main node] at ({ cos((12)*\x+\z)},{ sin((12)*\x+\z)}) (12) {$12$};

  \path[every node/.style={font=\sffamily}]
(1) edge (2)
(2) edge (3)
(3) edge (4)
(4) edge (5)
(5) edge (6)
(6) edge (7)
(7) edge (8)
(8) edge (9)
(9) edge (10)
(10) edge (11)
(11) edge (12)
(12) edge (13)
(13) edge (1)
(1) edge (6)
(2) edge (7)
(3) edge (8)
(4) edge (9)
(5) edge (10)
(6) edge (11)
(7) edge (12)
(8) edge (13)
(9) edge (1)
(10) edge (2)
(11) edge (3)
(12) edge (4)
(13) edge (5);
\end{tikzpicture}
&

\def\x{360/8}
\def\z{360/16}
\def\k{1.3}
\begin{tikzpicture}
[scale=1.8, thick,main node/.style={circle, minimum size=3.8mm, inner sep=0.1mm,draw,font=\tiny\sffamily}]
\node[main node] at ({\k*cos((0)*\x+\z)},{\k*sin((0)*\x+\z)}) (1a) {$1$};
\node[main node] at ({\k*cos(1*\x+\z)},{\k*sin((1)*\x+\z)}) (2a) {$2$};
\node[main node] at ({\k*cos(2*\x+\z)},{\k*sin((2)*\x+\z)}) (3a) {$3$};
\node[main node] at ({\k*cos((3)*\x+\z)},{\k*sin((3)*\x+\z)}) (4a) {$4$};
\node[main node] at ({\k*cos(4*\x+\z)},{\k*sin((4)*\x+\z)}) (5a) {$5$};
\node[main node] at ({\k*cos(5*\x+\z)},{\k*sin((5)*\x+\z)}) (6a) {$6$};
\node[main node] at ({\k*cos((6)*\x+\z)},{\k*sin((6)*\x+\z)}) (7a) {$7$};
\node[main node] at ({\k*cos((7)*\x+\z)},{\k*sin((7)*\x+\z)}) (8a) {$8$};

\node[main node] at ({cos((0.5)*\x+\z)},{sin((0.5)*\x+\z)}) (1b) {$9$};
\node[main node] at ({cos(1.5*\x+\z)},{sin((1.5)*\x+\z)}) (2b) {$10$};
\node[main node] at ({cos(2.5*\x+\z)},{sin((2.5)*\x+\z)}) (3b) {$11$};
\node[main node] at ({cos((3.5)*\x+\z)},{sin((3.5)*\x+\z)}) (4b) {$12$};
\node[main node] at ({cos(4.5*\x+\z)},{sin((4.5*\x+\z)}) (5b) {$13$};
\node[main node] at ({cos(5.5*\x+\z)},{sin((5.5)*\x+\z)}) (6b) {$14$};
\node[main node] at ({cos((6.5)*\x+\z)},{sin((6.5)*\x+\z)}) (7b) {$15$};
\node[main node] at ({cos((7.5)*\x+\z)},{sin((7.5)*\x+\z)}) (8b) {$16$};

  \path[every node/.style={font=\sffamily}]
(1a) edge (2a)
(2a) edge (3a)
(3a) edge (4a)
(4a) edge (5a)
(5a) edge (6a)
(6a) edge (7a)
(7a) edge (8a)
(8a) edge (1a)
(1a) edge (5a)
(2a) edge (6a)
(3a) edge (7a)
(4a) edge (8a)
(1b) edge (4b)
(4b) edge (7b)
(7b) edge (2b)
(2b) edge (5b)
(5b) edge (8b)
(8b) edge (3b)
(3b) edge (6b)
(6b) edge (1b)
(1b) edge (5b)
(2b) edge (6b)
(3b) edge (7b)
(4b) edge (8b)
(1a) edge (1b)
(2a) edge (2b)
(3a) edge (3b)
(4a) edge (4b)
(5a) edge (5b)
(6a) edge (6b)
(7a) edge (7b)
(8a) edge (8b);
\end{tikzpicture}\\
$G_{\ref{figVertexTransitive234},1}$ & $G_{\ref{figVertexTransitive234},2}$ & $G_{\ref{figVertexTransitive234},3}$ 
\end{tabular}
\caption{The smallest vertex-transitive graphs with $\LS_+$-rank $2$ (left), $3$ (centre), and $4$ (right).}\label{figVertexTransitive234}
\end{figure}

A significant portion of our work in this section involves proving $\LS_+$-rank upper bounds for vertex-transitive graphs up to a certain size, which is largely aided by the fact that vertex-transitive graphs on up to 47 vertices have been exhaustively listed (see~\cite{McKayR90, Skiena, RoyleH20}). 

We now determine $\overline{n}_+(2)$.

\begin{proposition}\label{propVT2}
$\overline{n}_+(2) = 8$. Moreover, the only vertex-transitive graph $G$ where $|V(G)| \leq 8$ and $r_+(G) = 2$ is $G_{\ref{figVertexTransitive234},1} \ce C_8^{\set{1,2}}$ (Figure~\ref{figVertexTransitive234}, left).
\end{proposition}

\begin{proof}
Suppose $G$ is a $k$-regular vertex-transitive graph on $n$ vertices, and $r_+(G) =2$. To find the smallest such graph, we may assume that $G$ is connected (otherwise $G$ contains a proper subgraph with the same $\LS_+$-rank). From Lemma~\ref{lemnkl}, we know that $3 \leq k \leq n-4$, and so $n \geq 7$. If $n=7$, then $k=3$, and no such graph exists. If $n=8$, then $3 \leq k \leq 4$. There are a total of 5 such graphs~\cite{McKayR90, Skiena}, as listed in Figure~\ref{figVT2}.

\def\sc{1.2}
\def\x{-22.5}
\def\z{360/8}
\def\placev{
\node[main node] at ({cos(\x+(1)*\z)},{sin(\x+(1)*\z)}) (1) {$1$}; 
\node[main node] at ({cos(\x+(2)*\z)},{sin(\x+(2)*\z)}) (2) {$2$}; 
\node[main node] at ({cos(\x+(3)*\z)},{sin(\x+(3)*\z)}) (3) {$3$}; 
\node[main node] at ({cos(\x+(4)*\z)},{sin(\x+(4)*\z)}) (4) {$4$}; 
\node[main node] at ({cos(\x+(5)*\z)},{sin(\x+(5)*\z)}) (5) {$5$}; 
\node[main node] at ({cos(\x+(6)*\z)},{sin(\x+(6)*\z)}) (6) {$6$}; 
\node[main node] at ({cos(\x+(7)*\z)},{sin(\x+(7)*\z)}) (7) {$7$}; 
\node[main node] at ({cos(\x+(8)*\z)},{sin(\x+(8)*\z)}) (8) {$8$}; 
}
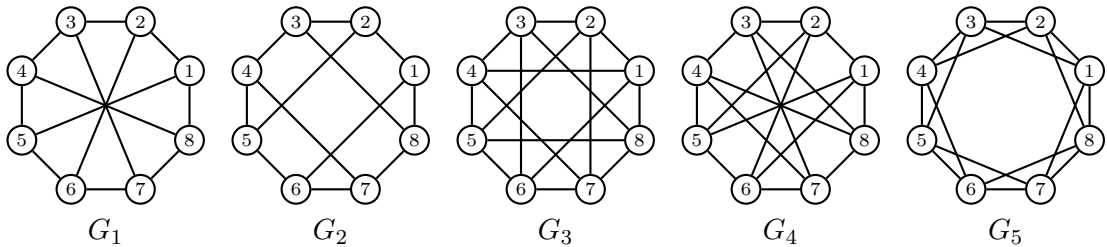
\begin{figure}[htbp]
\centering
\begin{tabular}{ccccc}
\begin{tikzpicture}[scale=\sc, thick,main node/.style={circle,  minimum size=3.8mm,  inner sep=0.1mm,draw,font=\tiny\sffamily}] \placev
\path[every node/.style={font=\sffamily}]
(1) edge (2) (2) edge (3) (3) edge (4) (4) edge (5) (5) edge (6) (6) edge (7) (7) edge (8) (8) edge (1)
 (4) edge (8) (5) edge (1) (6) edge (2) (7) edge (3);
\end{tikzpicture}
&
\begin{tikzpicture}[scale=\sc, thick,main node/.style={circle,  minimum size=3.8mm,  inner sep=0.1mm,draw,font=\tiny\sffamily}] \placev\path[every node/.style={font=\sffamily}]
(1) edge (2) (2) edge (3) (3) edge (4) (4) edge (5) (5) edge (6) (6) edge (7) (7) edge (8) (8) edge (1)
 (3) edge (8) (5) edge (2) (6) edge (1) (7) edge (4);
\end{tikzpicture}
&
\begin{tikzpicture}[scale=\sc, thick,main node/.style={circle,  minimum size=3.8mm,  inner sep=0.1mm,draw,font=\tiny\sffamily}] \placev\path[every node/.style={font=\sffamily}]
(1) edge (2) (2) edge (3) (3) edge (4) (4) edge (5) (5) edge (6) (6) edge (7) (7) edge (8) (8) edge (1)
 (3) edge (8) (5) edge (2) (6) edge (1) (7) edge (4)  (3) edge (6) (5) edge (8) (4) edge (1) (7) edge (2);
\end{tikzpicture}
&
\begin{tikzpicture}[scale=\sc, thick,main node/.style={circle,  minimum size=3.8mm,  inner sep=0.1mm,draw,font=\tiny\sffamily}] \placev\path[every node/.style={font=\sffamily}]
(1) edge (2) (2) edge (3) (3) edge (4) (4) edge (5) (5) edge (6) (6) edge (7) (7) edge (8) (8) edge (1)
 (3) edge (8) (5) edge (2) (6) edge (1) (7) edge (4)  (3) edge (7) (5) edge (1) (4) edge (8) (6) edge (2);
\end{tikzpicture}
&
\begin{tikzpicture}[scale=\sc, thick,main node/.style={circle,  minimum size=3.8mm,  inner sep=0.1mm,draw,font=\tiny\sffamily}] \placev\path[every node/.style={font=\sffamily}]
(1) edge (2) (2) edge (3) (3) edge (4) (4) edge (5) (5) edge (6) (6) edge (7) (7) edge (8) (8) edge (1)
 (1) edge (3) (2) edge (4) (3) edge (5) (4) edge (6)  (5) edge (7) (6) edge (8) (7) edge (1) (8) edge (2);
 \end{tikzpicture}
\\
$G_1$ & $G_2$ & $G_3$ & $G_4$ & $G_5$ 
\end{tabular}
\caption{The list of connected $k$-regular vertex-transitive graphs on 8 vertices with $3 \leq k \leq 4$.}
\label{figVT2}
\end{figure}

Observe that $G_1 = C_8^{\set{1,4}}$, and destroying any vertex yields a bipartite graph, and so $r_+(G_1)=1$. Next, $G_2$ (the 3-cube) and $G_3$ ($K_{4,4}$) are both bipartite and thus have $\LS_+$-rank 0. $G_4$ is  obtained from joining two disjoint copies of $K_4$ with a perfect matching, and thus is a perfect graph, which implies that $r_+(G_4) = 1$. This leaves us with $G_5$, which is the graph $G \ce G_{\ref{figVertexTransitive234},1}$ from Figure~\ref{figVertexTransitive234}. There are several ways to show that $r_+(G) \geq 2$:
\begin{itemize}
\item
An $\LS_+$ certificate package~\cite{AuT25data} shows that $\frac{4}{15} \bar{e} \in \LS_+(G)$, which implies that $\a_{\LS_+}(G) \geq \frac{32}{15} > \a(G) = 2$. Thus, $r_+(G) \geq 2$.
\item
Let $A(\overline{G})$ be the adjacency matrix of the complement graph of $G$, let $d \ce \frac{\sqrt{2}+4}{8(\sqrt{2}+1)}$, and let $Y \ce \begin{bmatrix} 1 & d \bar{e}^{\top} \\ d \bar{e} & d I_8 + \frac{d}{\sqrt{2}+1} A(\overline{G}) \end{bmatrix}$. Then one can check that $Y \in \widehat{\LS}_+(G)$ (see, for instance,~\cite[Proposition 4]{AuLT26} for a more detailed analysis for the $\LS_+$ certificates for regular graphs and some vertex-transitive graphs). This shows that $\a_{\LS_+}(G) \geq 8d \approx 2.24 > \a(G)$. 
\item
Observe that $G - \set{1,3}$ is isomorphic to $G_{\ref{figKnownEG},2}$. Since $G$ contains an induced subgraph with $\LS_+$-rank $2$, $r_+(G) \geq 2$.
\end{itemize}
Finally, since $|V(G)| = 8, r_+(G) \leq 2$ (by Theorem~\ref{thmLiptakT031}). It thus follows that $r_+(G) =2$.
\end{proof}

We remark that, for $n=9$, there is also exactly one vertex-transitive graph with $\LS_+$-rank $2$: $C_{9}^{\set{1,2}}$ (which contains $G_{\ref{figKnownEG},1}$ as an induced subgraph). The other two graphs which satisfy the criterion in Lemma~\ref{lemnkl} are $C_{9}^{\set{1,3}}$ and the Paley graph on 9 vertices, both of which would result in a bipartite graph upon the destruction of any vertex.

Next, we remark that while $\a_{\LS_+^{\ell}}(G) > \a(G)$ implies $r_+(G) > \ell$, the converse is not true. Before we describe such an example, we first prove a more general result.

\begin{proposition}\label{prop11}
Suppose $G$ is a graph with $n$ vertices and $\deg(i) \geq k$ for every vertex $i \in V(G)$. Then
\begin{itemize}
\item[(i)]
\[
\alpha_{\LS_+}(G) \leq n-k.
\]
\item[(ii)]
Moreover, if $\max\set{ \bar{e}^{\top}x : x \in \FRAC(G \ominus i)} \leq \frac{1}{2} |V(G\ominus i)|$ for every vertex $i \in V(G)$, then
\[
\alpha_{\LS_+}(G) \leq \frac{n-k+1}{2}.
\]
\end{itemize}
\end{proposition}

\begin{proof}
Let $Y \ce \begin{bmatrix} 1& x^{\top} \\ x & \tn{Diag}(x) + M \end{bmatrix}$ be a certificate matrix in $\widehat{\LS}_+(G)$. Then $Y \succeq 0$, and thus the Schur complement $\Diag(x) + M - xx^{\top}$ is also positive semidefinite. Hence,
\[
\bar{e}^{\top} ( \tn{Diag}(x) + M - xx^{\top}) \bar{e} = \bar{e}^{\top}x - (\bar{e}^{\top}x)^2 + \bar{e}^{\top}M\bar{e} \geq 0.
\]
Next, notice that since $Ye_0 = \diag(Y)$, $M[i,i] = 0$ for every $i \in V(G)$. Also, $M[i,j] = 0$ for all edges $\set{i,j} \in E(G)$. Thus, for every $i \in V(G)$, $Me_i$ must contain at most $n-\deg(i)-1 \leq n-k-1$ positive entries, each being at most $x_i$. Applying this for every $i \in V(G)$ yields $\bar{e}^{\top}M\bar{e} \leq (n-k-1)\bar{e}^{\top}x$. This gives $\bar{e}^{\top}x \leq n-k$.

For (ii), the additional assumption assures that $Me_i \leq \frac{n-k-1}{2} x_i$, and so $\bar{e}^{\top}M\bar{e} \leq \frac{n-k-1}{2} \bar{e}^{\top}x$, leading to the tighter bound $\bar{e}^{\top}x \leq \frac{n-k+1}{2}$.
\end{proof}

\def\x{360/8}
\def\z{360/16}
\def\k{2}
\begin{figure}[htbp]
\centering
\begin{tabular}{cc}
\begin{tikzpicture}
[scale=1.5, thick,main node/.style={circle, minimum size=3.8mm, inner sep=0.1mm,draw,font=\tiny\sffamily}]
\node[main node] at ({\k*cos((0)*\x+\z)},{\k*sin((0)*\x+\z)}) (1a) {$1_0$};
\node[main node] at ({\k*cos(1*\x+\z)},{\k*sin((1)*\x+\z)}) (2a) {$2_0$};
\node[main node] at ({\k*cos(2*\x+\z)},{\k*sin((2)*\x+\z)}) (3a) {$3_0$};
\node[main node] at ({\k*cos((3)*\x+\z)},{\k*sin((3)*\x+\z)}) (4a) {$4_0$};
\node[main node] at ({\k*cos(4*\x+\z)},{\k*sin((4)*\x+\z)}) (5a) {$5_0$};
\node[main node] at ({\k*cos(5*\x+\z)},{\k*sin((5)*\x+\z)}) (6a) {$6_0$};
\node[main node] at ({\k*cos((6)*\x+\z)},{\k*sin((6)*\x+\z)}) (7a) {$7_0$};
\node[main node] at ({\k*cos((7)*\x+\z)},{\k*sin((7)*\x+\z)}) (8a) {$8_0$};

\node[main node] at ({cos((0)*\x+\z)},{sin((0)*\x+\z)}) (5b) {$5_1$};
\node[main node] at ({cos(1*\x+\z)},{sin((1)*\x+\z)}) (6b) {$6_1$};
\node[main node] at ({cos(2*\x+\z)},{sin((2)*\x+\z)}) (7b) {$7_1$};
\node[main node] at ({cos((3)*\x+\z)},{sin((3)*\x+\z)}) (8b) {$8_1$};
\node[main node] at ({cos(4*\x+\z)},{sin((4*\x+\z)}) (1b) {$1_1$};
\node[main node] at ({cos(5*\x+\z)},{sin((5)*\x+\z)}) (2b) {$2_1$};
\node[main node] at ({cos((6)*\x+\z)},{sin((6)*\x+\z)}) (3b) {$3_1$};
\node[main node] at ({cos((7)*\x+\z)},{sin((7)*\x+\z)}) (4b) {$4_1$};

  \path[every node/.style={font=\sffamily}]
(1a) edge (2a)
(1a) edge (3a)
(2a) edge (3a)
(2a) edge (4a)
(3a) edge (4a)
(3a) edge (5a)
(4a) edge (5a)
(4a) edge (6a)
(5a) edge (6a)
(5a) edge (7a)
(6a) edge (7a)
(6a) edge (8a)
(7a) edge (8a)
(7a) edge (1a)
(8a) edge (1a)
(8a) edge (2a)
(1b) edge (2b)
(1b) edge (3b)
(2b) edge (3b)
(2b) edge (4b)
(3b) edge (4b)
(3b) edge (5b)
(4b) edge (5b)
(4b) edge (6b)
(5b) edge (6b)
(5b) edge (7b)
(6b) edge (7b)
(6b) edge (8b)
(7b) edge (8b)
(7b) edge (1b)
(8b) edge (1b)
(8b) edge (2b)
(1a) edge (2b)
(2a) edge (3b)
(3a) edge (4b)
(4a) edge (5b)
(5a) edge (6b)
(6a) edge (7b)
(7a) edge (8b)
(8a) edge (1b)
(1a) edge (4b)
(2a) edge (5b)
(3a) edge (6b)
(4a) edge (7b)
(5a) edge (8b)
(6a) edge (1b)
(7a) edge (2b)
(8a) edge (3b)
(1a) edge (6b)
(2a) edge (7b)
(3a) edge (8b)
(4a) edge (1b)
(5a) edge (2b)
(6a) edge (3b)
(7a) edge (4b)
(8a) edge (5b)
(1a) edge (8b)
(2a) edge (1b)
(3a) edge (2b)
(4a) edge (3b)
(5a) edge (4b)
(6a) edge (5b)
(7a) edge (6b)
(8a) edge (7b)
(1a) edge (5b)
(2a) edge (6b)
(3a) edge (7b)
(4a) edge (8b)
(5a) edge (1b)
(6a) edge (2b)
(7a) edge (3b)
(8a) edge (4b);
\end{tikzpicture}
&

\begin{tikzpicture}
[scale=1.5, thick,main node/.style={circle, minimum size=3.8mm, inner sep=0.1mm,draw,font=\tiny\sffamily}]
\node[main node] at ({\k*cos((3)*\x+\z)},{\k*sin((3)*\x+\z)}) (4a) {$4_0$};
\node[main node] at ({\k*cos(4*\x+\z)},{\k*sin((4)*\x+\z)}) (5a) {$5_0$};
\node[main node] at ({\k*cos(5*\x+\z)},{\k*sin((5)*\x+\z)}) (6a) {$6_0$};

\node[main node] at ({cos(2*\x+\z)},{sin((2)*\x+\z)}) (7b) {$7_1$};
\node[main node] at ({cos(4*\x+\z)},{sin((4)*\x+\z)}) (1b) {$1_1$};
\node[main node] at ({cos((6)*\x+\z)},{sin((6)*\x+\z)}) (3b) {$3_1$};

  \path[every node/.style={font=\sffamily}]
(4a) edge (5a)
(4a) edge (6a)
(5a) edge (6a)
(7b) edge (1b)
(1b) edge (3b)
(6a) edge (7b)
(4a) edge (7b)
(6a) edge (1b)
(4a) edge (1b)
(6a) edge (3b)
(4a) edge (3b)
(5a) edge (1b);
\end{tikzpicture}
\\
$G$ & $G \ominus 1_0$ 
\end{tabular}
\caption{A vertex-transitive graph $G$ where $\a_{\LS_+}(G) = \a(G)$ and $r_+(G) > 1$}\label{figDoubleCirculant}

\end{figure}
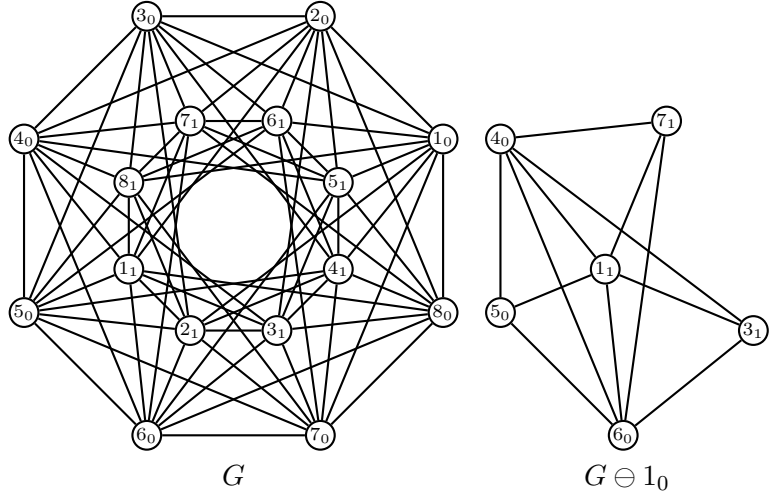

\begin{example}
Consider the graph $G$ where $V(G) \ce \set{ i_0, i_1 : i \in [8]}$ and
\[
E(G) \ce \set{ \set{i_0, j_0}, \set{i_1, j_1} : \set{i,j} \in E(C_8^{\set{1,2}})} \cup \set{ \set{i_0,j_1} : \set{i,j} \in E(C_{8}^{\set{1,3,4}}) }.
\]
(See Figure~\ref{figDoubleCirculant}, left.) Then $G$ is $k \ce 9$-regular and vertex-transitive, with $G \ominus 1_0$ being a 6-vertex graph (Figure~\ref{figDoubleCirculant}, right). Now notice that the edges $\set{4_0,7_1}$, $\set{5_0, 1_1}$, and $\set{6_0,3_1}$ form a perfect matching in $G \ominus 1_0$. Thus, $\max \set{ \bar{e}^{\top}x : x \in \FRAC(G \ominus 1_0 )} \leq 3 = \frac{1}{2} |V(G\ominus 1_0)|$. Since $G$ is vertex-transitive, the same must be true for all vertices $i \in V(G)$, and so $G$ satisfies the conditions of Proposition~\ref{prop11}(ii). Hence, $\alpha_{\LS_+}(G) \leq \frac{16-9+1}{2} = 4 = \alpha(G)$. However, since $G$ contains $C_{8}^{\set{1,2}}$ as an induced subgraph, we see that $r_+(G) \geq 2$.
\end{example}

We next determine $\overline{n}_+(3)$. 

\begin{proposition}\label{propVT3}
$\overline{n}_+(3) = 13$. Moreover, the only vertex-transitive graph $G$ where $|V(G)| \leq 13$ and $r_+(G) = 3$ is $G_{\ref{figVertexTransitive234},2} \ce C_{13}^{\set{1,5}}$ (Figure~\ref{figVertexTransitive234}, centre).
\end{proposition}

\begin{proof}
Suppose a vertex-transitive graph $G$ has $n$ vertices, is $k$-regular and has $\LS_+$-rank $3$. It follows from Lemma~\ref{lemnkl} that $3 \leq k \leq n-7$. Thus, we see that $n \geq 10$. As with in the proof of Proposition~\ref{propVT2}, we may assume that $G$ is connected. 

There are 38 connected vertex-transitive graphs on $n \in \set{10,11,12,13}$ vertices which satisfy the degree bounds from Lemma~\ref{lemnkl} with $\ell = 3$~\cite{RoyleH20}, which we list in Table~\ref{tabVT3}. (For compactness, we are listing these graphs in graph6 format.)

\begin{table}[htbp]
\centering
\scriptsize{
\begin{tabular}{c c c c}
\begin{tabular}{c|l}
$i$ & $G_i$\\
\hline
$1$ & \verb+Is?@WxcU?+ \\ 
$2$ & \verb+Is?HGtcU?+ \\ 
$3$ & \verb+IsP@OkWHG+ \\ 
$4$ & \verb+Js`@IStU`w?+ \\ 
$5$ & \verb+Juk?IKeTPT?+ \\ 
$6$ & \verb+Ks???wYP`KL?+ \\ 
$7$ & \verb+Ks?GOGUIQcKG+ \\ 
$8$ & \verb+Ks?GOObDRCI_+ \\ 
$9$ & \verb+Kt?GOHAOWsCg+ \\ 
$10$ & \verb+Ks_?BLMLasF_+ \\ 
\end{tabular}
&
\begin{tabular}{c|l}
$i$ & $G_i$\\
\hline
$11$ & \verb+Ks_?BLUR`wFO+ \\ 
$12$ & \verb+Ks_?BLeF_{N?+ \\ 
$13$ & \verb+Ks_GagjLASko+ \\ 
$14$ & \verb+Ksc?JLSQhSE`+ \\ 
$15$ & \verb+Ksc@ILKLAdDI+ \\ 
$16$ & \verb+Ksd@?STWY[Eo+ \\ 
$17$ & \verb+Ktk?IHBD_[kK+ \\ 
$18$ & \verb+Kuk?GLDGqhDQ+ \\ 
$19$ & \verb+K{cAGgeBQSeK+ \\ 
$20$ & \verb+Ksa?BtuZa{Fo+ \\ 
\end{tabular}
&
\begin{tabular}{c|l}
$i$ & $G_i$\\
\hline
$21$ & \verb+KsaIPKuUZkNG+ \\ 
$22$ & \verb+KseXa`JHrJLQ+ \\ 
$23$ & \verb+KseY`TIKZLKi+ \\ 
$24$ & \verb+KsiYISiDZdMI+ \\ 
$25$ & \verb+KsiZ?lEEZBnO+ \\ 
$26$ & \verb+KtaHGthYaiis+ \\ 
$27$ & \verb+KtiWBDRQqLfo+ \\ 
$28$ & \verb+KtiY@DFQYefo+ \\ 
$29$ & \verb+K{eY`dIPhRCj+ \\ 
$30$ & \verb+K{fw?DbWwuBX+ \\ 
\end{tabular}
&
\begin{tabular}{c|l}
$i$ & $G_i$\\
\hline
$31$ & \verb+K}iWBDRIol@r+ \\ 
$32$ & \verb+Ls_?GSTTPTLO\?+ \\ 
$33$ & \verb+Ls`?XGRQR@B`Kc+ \\ 
$34$ & \verb+Lts?GKEPPDHIKI+ \\ 
$35$ & \verb+L|qC@|]JakiiIj+ \\ 
$36$ & \verb+L}akqXXWomdULJ+ \\ 
$37$ & \verb+L}nDAwyBgmGfGv+ \\ 
$38$ & \verb+L~zTQgiDOT_n?~+ \\ 
& \\
& \\
\end{tabular}
\end{tabular}
}
\caption{The list of connected vertex-transitive graphs satisfying the degree bounds from Lemma~\ref{lemnkl} with $n \in \set{10,11,12,13}$ and $\ell = 3$, in graph6 format.}
\label{tabVT3}
\end{table}

Next, notice that
\begin{itemize}
\item 
For every $i \in \set{1,6,8,10,11,12,20}$, $G_i$ is bipartite, so these graphs all have $\LS_+$-rank 0.
\item 
For every $i \in \set{2,3,4, 7, 21, 26, 32,35}$, $G_i  -1$ is bipartite, so these graphs all have $\LS_+$-rank 1.
\item 
For every $i \in \set{5, 9,14,15,17,18,19,22,23,24,25,27,28,29,30,34,36,37,38}$, $G_i \ominus 1$ is perfect, so these graphs all have $\LS_+$-rank at most 2.
\item 
Notice that $(G_{13} \ominus 1) - 12$ and $(G_{16} \ominus 1) - 8$ are both bipartite, and so $r_+(G_{13}), r_+(G_{16}) \leq 2$.
\item 
Finally, notice that $G_{31} \ominus 1$ is the 5-wheel, which has $\LS_+$-rank $1$. Thus, $r_+(G_{31}) \leq 2$.
\end{itemize}

This leaves $G_{33}$, which is isomorphic to $G \ce C_{13}^{\set{1,5}}$. We provide a $\LS_+^2$ certificate package~\cite{AuT25data} showing that $\frac{17}{55} \bar{e} \in \LS_+^2(G)$, which implies that $\a_{\LS_+^2}(G) \geq \frac{221}{55} > \a(G) = 4$. This implies that $r_+(G) \geq 3$. Also, destroying a vertex in $G$ yields an 8-vertex graph, which has $\LS_+$-rank at most 2 by Theorem~\ref{thmLiptakT031}. Thus, it follows that $r_+(G) = 3$.
\end{proof}

We now determine $\overline{n}_+(4)$ using a similar analysis.

\begin{proposition}\label{propVT4}
$\overline{n}_+(4) =16$. Moreover, the only vertex-transitive graph $G$ where $|V(G)| \leq 16$ and $r_+(G) = 4$ is $G_{\ref{figVertexTransitive234},3}$ (Figure~\ref{figVertexTransitive234}, right).
\end{proposition}

\begin{proof}
Let $G$ be a $k$-regular vertex-transitive graph on $n$ vertices where $r_+(G) =4$. Again, we may assume that $G$ is connected. By Lemma~\ref{lemnkl}, we know that $3 \leq k \leq n-10$, so $n \geq 13$. When $n=13$, $k=3$, and no such graph exists. So we may assume $n \geq 14$.

Next, we look into the connected $k$-regular vertex-transitive graphs on $n \in \set{14,15,16}$ vertices which satisfy $3 \leq k \leq n-10$. There are 96 such graphs~\cite{RoyleH20}, which are listed in Table~\ref{tabVT4}.

\begin{table}[htbp]
\centering
\tiny{
\begin{tabular}{c c c c}
\begin{tabular}{c|l}
$i$ & $G_i$\\
\hline
$1$ & \verb+Ms???@KOpSBGHOD_?+ \\ 
$2$ & \verb+Ms??OGKB?ccKS_WO?+ \\ 
$3$ & \verb+Ms?GGSG@?bCQSGWC?+ \\ 
$4$ & \verb+Ms_???VBpeHg[_Z??+ \\ 
$5$ & \verb+Ms_??@eTPeHWJOF_?+ \\ 
$6$ & \verb+Ms_??KEXbKBKEW]??+ \\ 
$7$ & \verb+Ms_AHGHCihDaQOK__+ \\ 
$8$ & \verb+Mts?GKE@QDCIQIKD?+ \\ 
$9$ & \verb+Ns_?@DDDOTDASBL_Ho?+ \\ 
$10$ & \verb+Ns_?ACKChIIDISR_Eo?+ \\ 
$11$ & \verb+Ns_?GGAAohdWTGYOMC?+ \\ 
$12$ & \verb+Nsc?GCCGyJI_I_QHEGO+ \\ 
$13$ & \verb+Nsc?GCDP`Bi_Q`IOECG+ \\ 
$14$ & \verb+Nto?AHBHPC`PCgD?`_g+ \\ 
$15$ & \verb+Nuk?GGA@oLCIOIIPICO+ \\ 
$16$ & \verb+Os???GA?WBBAEAP_CoBO?+ \\ 
$17$ & \verb+Os?GOGEAOC?GCDGII?b?O+ \\ 
$18$ & \verb+Os?GOO??AHGQGSDCAK@D?+ \\ 
$19$ & \verb+Os?GOOD@_C?GCEGHI@B?G+ \\ 
$20$ & \verb+Os_????@zKIgIgLGHW?r?+ \\ 
$21$ & \verb+Os_????AwjKgJ_RGESAe?+ \\ 
$22$ & \verb+Os_????BWfKoTOLGHS@U?+ \\ 
$23$ & \verb+Os_??CDAGaiaIaX_HS@i?+ \\ 
$24$ & \verb+Os_??CDCGQePJCX_ISAY?+ \\ 
$25$ & \verb+Os_??DAGOZAqIGSCho@i?+ \\ 
$26$ & \verb+Os_??KE@_KKKWWEEBBBo?+ \\ 
$27$ & \verb+Os_??L@GOS`SEKT@E`BK?+ \\ 
$28$ & \verb+Os_?A?_DOhA[BWQaDQBK?+ \\ 
$29$ & \verb+Os_?GGC@?RaUSgISLGBc?+ \\ 
$30$ & \verb+Os_?IGbC`Ga_O`ISB@QPG+ \\ 
$31$ & \verb+Otk?GGB?WBgAOBD_@oJ?K+ \\ 
$32$ & \verb+Ouk?GKC?_J?YGIODDCQ_a+ \\ 
\end{tabular}
&
\begin{tabular}{c|l}
$i$ & $G_i$\\
\hline$33$ & \verb+Osa????Dw^KwXgUWFKBs?+ \\ 
$34$ & \verb+Osa????ExN@{UWXgNGBk?+ \\ 
$35$ & \verb+Osa????WyfH[PwN_Fg@x?+ \\ 
$36$ & \verb+Osa????WyuFKRWJgIs@m?+ \\ 
$37$ & \verb+Osa??CH@gqctUSXgNGBk?+ \\ 
$38$ & \verb+Osa??cQAWJIRKeZ_LgBX?+ \\ 
$39$ & \verb+Osa?AKUB`KKEWKERBHRs?+ \\ 
$40$ & \verb+Osa?OGgPGqasKsSwLGbKO+ \\ 
$41$ & \verb+OsaBA`GP@`dIHWEcas_]O+ \\ 
$42$ & \verb+OseW?CBSaPgjK_IGcgZF?+ \\ 
$43$ & \verb+OseW?CFOQ`eAHPI`h[BCo+ \\ 
$44$ & \verb+OseW?CFOQ`eAH`IPh[BCo+ \\ 
$45$ & \verb+OseW?CbOyHCaSaQQdO`aH+ \\ 
$46$ & \verb+OseW?CbOyPCaSaKPGhPSI+ \\ 
$47$ & \verb+OseW?SFGHAiOKeQBGajF?+ \\ 
$48$ & \verb+OseWOLAOYECXSLKADAHHB+ \\ 
$49$ & \verb+Oseg?D@AqcgfWwEOEOHHB+ \\ 
$50$ & \verb+Osf_?KFO`CBRSSKQdDAP`+ \\ 
$51$ & \verb+Osf_GK_O@Ba[SLQcBJ@TA+ \\ 
$52$ & \verb+OsiW??`CjCiKPPQad[?m_+ \\ 
$53$ & \verb+OsiW?CAOWNKQEQDWkUBoG+ \\ 
$54$ & \verb+OsiW?CBWieHEPRK_DCGiB+ \\ 
$55$ & \verb+OsiWA?RG_KiPWWQRAePU_+ \\ 
$56$ & \verb+OsiWA@@AgMkXSkK_DCGiB+ \\ 
$57$ & \verb+O{fw?CB?wE?XWKWKb@_oX+ \\ 
$58$ & \verb+OsaC???FzrMkYwXwJw@|?+ \\ 
$59$ & \verb+OsaC???RxnNKYw\WJs@}?+ \\ 
$60$ & \verb+OsaC???]ZrK{XwFwB{B{?+ \\ 
$61$ & \verb+OsaCB@_EWrKrXeFwB{B{?+ \\ 
$62$ & \verb+OsaKYCQGbRIjXKLWJEHqc+ \\ 
$63$ & \verb+OsaKYCcQXRBKSbLDmTBi_+ \\ 
$64$ & \verb+OsaKYDDGgdLBUELQeibr?+ \\ 
\end{tabular}
&
\begin{tabular}{c|l}
$i$ & $G_i$\\
\hline$65$ & \verb+OsaKYPDHOiDFEM[wMPRcg+ \\ 
$66$ & \verb+OsaKg?dQGid\[S[qLSQy_+ \\ 
$67$ & \verb+OsaKg?hSZEkkTIIdJWPyA+ \\ 
$68$ & \verb+OsaKiCaC`RbMYKTYLabiC+ \\ 
$69$ & \verb+OsaKiCdPPDaUBR]WNAboS+ \\ 
$70$ & \verb+OsaSWSTOhDKiXQUEfBbr?+ \\ 
$71$ & \verb+OsaSXCdOha`T[DRRNEAyC+ \\ 
$72$ & \verb+OsaSXDCSXRbKTBIdmSBY@+ \\ 
$73$ & \verb+OsaSYHBGpH`XDL]SNDBoK+ \\ 
$74$ & \verb+Osedw?DOYFHBSZW\Fg@w`+ \\ 
$75$ & \verb+Osedw?HOYFGbSZW\Fg@w`+ \\ 
$76$ & \verb+Osedw@??pJhMP\UWEjBpA+ \\ 
$77$ & \verb+OsfDw?@G`bgmW\RWFFAyA+ \\ 
$78$ & \verb+OsfDw@@GXPCZP]TSFDayA+ \\ 
$79$ & \verb+OsfLg?@WYHcZQYKXJcPuA+ \\ 
$80$ & \verb+Osqsw@@AXCclW]USEXaxA+ \\ 
$81$ & \verb+Osqsy@@GWRcdGtUEESqyA+ \\ 
$82$ & \verb+OtaCXOTHOphhRKSsKTJcK+ \\ 
$83$ & \verb+OtaLw?@OaRgn[S[Kdk?z@+ \\ 
$84$ & \verb+OtaLw?@ObBiNQ[P[fg@x@+ \\ 
$85$ & \verb+OtaLw?BOBBhNP[S[fa@y@+ \\ 
$86$ & \verb+OtaLy@@AWJg[WFSFfg@x@+ \\ 
$87$ & \verb+OtaLyD`SYQGhKBKB`eOXb+ \\ 
$88$ & \verb+Ota\W@@GYChJS]PZFc@u@+ \\ 
$89$ & \verb+Otakw?@QYbKLPLOufc@u@+ \\ 
$90$ & \verb+Otaky@@GWbhBPROlfc@u@+ \\ 
$91$ & \verb+OtrTOGBW@`hIP]G|Bg_tP+ \\ 
$92$ & \verb+Ouj\w?@?YBcMSUILIhPTI+ \\ 
$93$ & \verb+O{fL_@HKQJCZCsBW_pzp_+ \\ 
$94$ & \verb+O{fL_CCQP`GmCzB[C\Joo+ \\ 
$95$ & \verb+O}akqPPWOV@iHIDHcROcj+ \\ 
$96$ & \verb+O~aKYPDOxQBHHIGeacocj+ \\ 
\end{tabular}
\end{tabular}
}
\caption{The list of connected vertex-transtive graphs satisfying the degree bounds from Lemma~\ref{lemnkl} with $n \in \set{14,15,16}$ and $\ell = 4$, in graph6 format.}
\label{tabVT4}
\end{table}

Observe that
\begin{itemize}
\item
For every $i \in \set{1,2,4,5, 16,18,19,20,21,22,26,29,33,34,35,36,58,59,60}$, $G_i$ is bipartite, so these graphs all have $\LS_+$-rank 0.
\item
For every $i \in \set{3,6,11,17,37,38,61}$, $G_i - 1$ is bipartite, so these graphs all have $\LS_+$-rank 1.
\item
For every $i \in \{8,12,13,15, 31,32,42,45,46,53,57,65,66,67,69,73,74,75,76,77,78,79$,
$80,81,83,84,86,88,89,90,92,96 \}$, $G_i \ominus 1$ is perfect, so these graphs all have $\LS_+$-rank at most 2.
\item
The graphs
\[
\begin{array}{llll}
(G_{7} \ominus 1) - 6, &
(G_{9} \ominus 1) - 7, &
(G_{10} \ominus 1) - 7,&
(G_{23} \ominus 1) - 14, \\
(G_{24} \ominus 1) - 14, &
(G_{25} \ominus 1) - 6, & 
(G_{27} \ominus 1) - 6,&
(G_{28} \ominus 1) - 7, \\
(G_{39} \ominus 1) - 7
\end{array}
\]
are all bipartite. Thus, $r_+(G_i) \leq 2$ for every $i \in \set{7,9,10,23,24,25,27,28,39}$.
\item
The graphs
\[
(G_{14} \ominus 1) - \set{6,7},
(G_{40} \ominus 1) - \set{7,8},
(G_{52} \ominus 1) - \set{7,8}
\]
are all bipartite. Thus, $r_+(G_i) \leq 3$ for every $i \in \set{14,40,52}$.
\item
The graphs
\[
\begin{array}{llll}
(G_{43} \ominus 1) - 7, &
(G_{44} \ominus 1) - 7, &
(G_{47} \ominus 1) - 7, &
(G_{48} \ominus 1) - 7, \\
(G_{49} \ominus 1) - 7, &
(G_{54} \ominus 1) - 7, &
(G_{55} \ominus 1) - 8, &
(G_{56} \ominus 1) - 7, \\
(G_{63} \ominus 1) - 10, &
(G_{64} \ominus 1) - 8, &
(G_{68} \ominus 1) - 8, &
(G_{70} \ominus 1) - 8, \\
(G_{71} \ominus 1) - 8, &
(G_{72} \ominus 1) - 8, &
(G_{85} \ominus 1) - 8, &
(G_{87} \ominus 1) - 8, \\
(G_{93} \ominus 1) - 9, &
(G_{94} \ominus 1) - 10 &
\end{array}
\]
are all perfect. Thus, $r_+(G_i) \leq 3$ for every $i \in \{43,44,47,48,49,54,55,56,63,64,68,70$,
$71,72,85,87,93,94\}$.
\item
For every $i \in \set{62,91,95}$, observe that $G_i \ominus 1$ has exactly $9$ vertices, all of which have degree at least $3$, and thus $G_i \ominus 1$ is not 3-minimal (since every 3-minimal graph must have at least one vertex with degree 2 due to Theorem~\ref{thmAuT24b1}). Hence $r_+(G_i \ominus 1) \leq 2$, which implies that $r_+(G_i) \leq 3$.
\item
Next, for convenience, let $H \ce G_{50} \ominus 1$ (Figure~\ref{figPropVT4}, left). For every $i \in V(H)$, $H \ominus i$ is either a 5-vertex graph, or a 6-vertex graph which is not isomorphic to $G_{\ref{figKnownEG},1}$ or $G_{\ref{figKnownEG},2}$. Thus, $r_+(H \ominus i) \leq 1$ for every $i \in V(H)$, which implies that $r_+(H) \leq 2$, and thus $r_+(G_{50}) \leq 3$. The same argument applies for $H \ce G_{82} \ominus 1$ (Figure~\ref{figPropVT4}, centre). Thus, $r_+(G_{50}), r_+(G_{82}) \leq 3$.
\item
Next, consider the graph $H \ce (G_{51} \ominus 1) - 10$ (Figure~\ref{figPropVT4}, right). Notice that $\set{8,11}$, $\set{9,12}$, and $\set{7,16}$ are all cut cliques in $H$, and one can show that $r_+(H) \leq 1$ by decomposing $H$ using Proposition~\ref{propCliqueCut}. Thus implies that $r_+(G_{51}) \leq 3$.
\item
Finally, $G_{41}$ is the (5-regular) Clebsch graph, and so $G_{41} \ominus 1$ yields the Peterson graph, which has $\LS_+$-rank 1. (To see this, observe that the Peterson graph is vertex-transitive in its own right, and destroying any vertex yields a bipartite graph.) Thus, $r_+(G_{41}) \leq 2$.
\end{itemize}

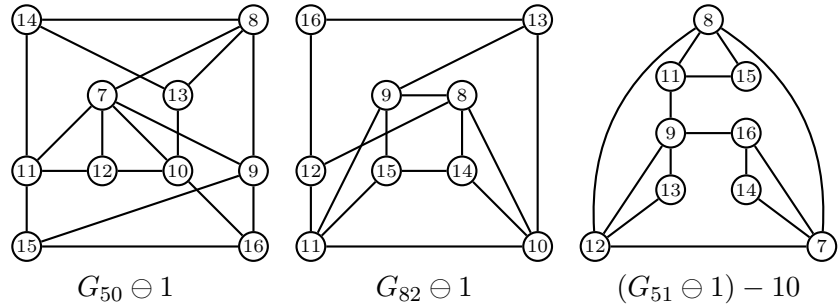
\begin{figure}[htbp]
\centering
\begin{tabular}{ccc}
\def\k{1}
\def\x{360/10}
\def\z{360/20}
\begin{tikzpicture}[scale=1, thick,main node/.style={circle, minimum size=3.8mm, inner sep=0.1mm,draw,font=\tiny\sffamily}]
\node[main node] at (1,2) (7) {$7$};
\node[main node] at (3,3) (8) {$8$};
\node[main node] at (3,1) (9) {$9$};
\node[main node] at (2,1) (10) {$10$};
\node[main node] at (0,1) (11) {$11$};
\node[main node] at (1,1) (12) {$12$};
\node[main node] at (2,2) (13) {$13$};
\node[main node] at (0,3) (14) {$14$};
\node[main node] at (0,0) (15) {$15$};
\node[main node] at (3,0) (16) {$16$};

  \path[every node/.style={font=\sffamily}]
(7) edge (8) (7) edge (9) (7) edge (10) (7) edge (11) (7) edge (12) (8) edge (9) (8) edge (13) (8) edge (14) (9) edge (15) (9) edge (16) (10) edge (12) (10) edge (13) (10) edge (16) (11) edge (12) (11) edge (14) (11) edge (15) (13) edge (14) (15) edge (16) ;
\end{tikzpicture}
&
\begin{tikzpicture}[scale=1, thick,main node/.style={circle, minimum size=3.8mm, inner sep=0.1mm,draw,font=\tiny\sffamily}]
\node[main node] at (2,2) (8) {$8$};
\node[main node] at (1,2) (9) {$9$};
\node[main node] at (3,0) (10) {$10$};
\node[main node] at (0,0) (11) {$11$};
\node[main node] at (0,1) (12) {$12$};
\node[main node] at (3,3) (13) {$13$};
\node[main node] at (2,1) (14) {$14$};
\node[main node] at (1,1) (15) {$15$};
\node[main node] at (0,3) (16) {$16$};

  \path[every node/.style={font=\sffamily}]
(8) edge (9) (8) edge (10) (8) edge (12) (8) edge (14) (9) edge (11) (9) edge (13) (9) edge (15) (10) edge (11) (10) edge (13) (10) edge (14) (11) edge (12) (11) edge (15) (12) edge (16) (13) edge (16) (14) edge (15) ;

\end{tikzpicture}
&

\begin{tikzpicture}[scale=1, thick,main node/.style={circle, minimum size=3.8mm, inner sep=0.1mm,draw,font=\tiny\sffamily}]
\node[main node] at (3,0) (7) {$7$};
\node[main node] at (1.5,3) (8) {$8$};
\node[main node] at (1,1.5) (9) {$9$};
\node[main node] at (1,2.25) (11) {$11$};
\node[main node] at (0,0) (12) {$12$};
\node[main node] at (1,0.75) (13) {$13$};
\node[main node] at (2,0.75) (14) {$14$};
\node[main node] at (2,2.25) (15) {$15$};
\node[main node] at (2,1.5) (16) {$16$};

  \path[every node/.style={font=\sffamily}]
 (7) edge[bend right] (8) (7) edge (12) (7) edge (14) (7) edge (16) (8) edge (11) (8) edge[bend right] (12) (8) edge (15) (9) edge (11) (9) edge (12) (9) edge (13) (9) edge (16) (11) edge (15) (12) edge (13) (14) edge (16) ;
\end{tikzpicture}
\\
$G_{50} \ominus 1$ & $G_{82} \ominus 1$ & $(G_{51} \ominus 1) - 10$
\end{tabular}
\caption{Illustrations for the proof of Proposition~\ref{propVT4}.}
\label{figPropVT4}
\end{figure}

This leaves $G_{30}$, which is isomorphic to $G \ce G_{\ref{figVertexTransitive234},3}$ in Figure~\ref{figVertexTransitive234}. The attached $\LS_+^3$ certificate package~\cite{AuT25data} shows that $\frac{26}{81}\bar{e} \in \LS_+^3(G)$, and thus $\a_{\LS_+^3}(G) \geq \frac{416}{81} > \a(G) =  5$, implying that $r_+(G) \geq 4$. Also, observe that $(G \ominus 1) - \set{11,15}$ is a 9-cycle ($\LS_+$-rank $1$), implying that $r_+(G) \leq 4$. 
\end{proof}

Interestingly, $G_{\ref{figVertexTransitive234},1}$, $G_{\ref{figVertexTransitive234},2}$, and $G_{\ref{figVertexTransitive234},3}$ are all 4-regular. Also, all three graphs have some stretched-clique structures embedded in them, which are highlighted in Figure~\ref{figVertexTransitive234SK}. First, as mentioned in the proof of Proposition~\ref{propVT2}, $G_{\ref{figVertexTransitive234},1} - \set{1,3}$ gives a copy of $G_{\ref{figKnownEG},2} \in \hat{\K}_{4,1}$. Also, notice that $G_{\ref{figVertexTransitive234},2} \ominus 1 \in \hat{\K}_{4,2}$ and $G_{\ref{figVertexTransitive234},3} \ominus 1 \in \hat{\K}_{5,3}$. Moreover, in both cases, one can join a new vertex to every vertex in the aforementioned stretched-clique subgraph and then 4-stretch the new vertex to obtain the full graph.  Thus, these graphs seem to share some structural similarities with $\ell$-minimal graphs, as they can also be obtained from  $K_3$ by a sequence of 1-joining and $k$-stretching operations akin to that described in Corollary~\ref{corJoinStretch}.

\def\sc{1.7}

\begin{figure}[htbp]
\centering
\begin{tabular}{ccc}
\def\x{270 - 360/4}
\def\z{360/4}
\def\y{0.7}

\begin{tikzpicture}
[scale=\sc, thick,main node/.style={circle, minimum size=3.8mm, inner sep=0.1mm,draw,font=\tiny\sffamily}]
\node[main node] at ({cos(\x+(0)*\z)},{sin(\x+(0)*\z)}) (5) {$5$};
\node[main node] at ({cos(\x+(1)*\z)},{sin(\x+(1)*\z)}) (6) {$6$};
\node[main node] at ({cos(\x+(2)*\z)},{sin(\x+(2)*\z)}) (7) {$7$};
\node[main node] at ({ \y* cos(\x+(3)*\z) + (1-\y)*cos(\x+(2)*\z)},{ \y* sin(\x+(3)*\z) + (1-\y)*sin(\x+(2)*\z)}) (8) {$8$};
\node[main node] at ({cos(\x+(3)*\z)},{sin(\x+(3)*\z)}) (2) {$2$};
\node[main node] at ({ \y* cos(\x+(3)*\z) + (1-\y)*cos(\x+(4)*\z)},{ \y* sin(\x+(3)*\z) + (1-\y)*sin(\x+(4)*\z)}) (4) {$4$};

  \path[every node/.style={font=\sffamily}]
(2) edge (4)
(4) edge (5)
(4) edge (6)
(5) edge (6)
(5) edge (7)
(6) edge (7)
(6) edge (8)
(7) edge (8)
(8) edge (2);
\end{tikzpicture}
&

\def\x{270 - 360/8}
\def\z{360/4}

\begin{tikzpicture}
[scale=\sc, thick,main node/.style={circle, minimum size=3.8mm, inner sep=0.1mm,draw,font=\tiny\sffamily}]
\node[main node] at ({cos(\x+(0)*\z)},{sin(\x+(0)*\z)}) (11) {$11$};
\node[main node] at ({cos(\x+(1)*\z)},{sin(\x+(1)*\z)}) (12) {$12$};

\node[main node] at ({ \y* cos(\x+(2)*\z) + (1-\y)*cos(\x+(1)*\z)},{ \y* sin(\x+(2)*\z) + (1-\y)*sin(\x+(1)*\z)}) (7) 
{$7$};
\node[main node] at ({cos(\x+(2)*\z)},{sin(\x+(2)*\z)}) (8) {$8$};
\node[main node] at ({ \y* cos(\x+(2)*\z) + (1-\y)*cos(\x+(3)*\z)},{ \y* sin(\x+(2)*\z) + (1-\y)*sin(\x+(3)*\z)}) (3) {$3$};

\node[main node] at ({ \y* cos(\x+(3)*\z) + (1-\y)*cos(\x+(2)*\z)},{ \y* sin(\x+(3)*\z) + (1-\y)*sin(\x+(2)*\z)}) (4) {$4$};
\node[main node] at ({cos(\x+(3)*\z)},{sin(\x+(3)*\z)}) (5) {$5$};
\node[main node] at ({ \y* cos(\x+(3)*\z) + (1-\y)*cos(\x+(4)*\z)},{ \y* sin(\x+(3)*\z) + (1-\y)*sin(\x+(4)*\z)}) (10) {$10$};

  \path[every node/.style={font=\sffamily}]
(3) edge (4)
(4) edge (5)
(7) edge (8)
(10) edge (11)
(11) edge (12)
(3) edge (8)
(5) edge (10)
(7) edge (12)
(11) edge (3)
(12) edge (4)
;
\end{tikzpicture}
&

\def\x{270 - 360/10}
\def\z{360/5}
\begin{tikzpicture}
[scale=\sc, thick,main node/.style={circle, minimum size=3.8mm, inner sep=0.1mm,draw,font=\tiny\sffamily}]
\node[main node] at ({cos(\x+(0)*\z)},{sin(\x+(0)*\z)}) (3b) {$11$};
\node[main node] at ({cos(\x+(1)*\z)},{sin(\x+(1)*\z)}) (7b) {$15$};

\node[main node] at ({ \y* cos(\x+(2)*\z) + (1-\y)*cos(\x+(1)*\z)},{ \y* sin(\x+(2)*\z) + (1-\y)*sin(\x+(1)*\z)}) (2b) 
{$10$};
\node[main node] at ({cos(\x+(2)*\z)},{sin(\x+(2)*\z)}) (5b) {$13$};
\node[main node] at ({ \y* cos(\x+(2)*\z) + (1-\y)*cos(\x+(3)*\z)},{ \y* sin(\x+(2)*\z) + (1-\y)*sin(\x+(3)*\z)}) (8b) {$16$};

\node[main node] at ({ \y* cos(\x+(3)*\z) + (1-\y)*cos(\x+(2)*\z)},{ \y* sin(\x+(3)*\z) + (1-\y)*sin(\x+(2)*\z)}) (4b) {$12$};
\node[main node] at ({cos(\x+(3)*\z)},{sin(\x+(3)*\z)}) (4a) {$4$};
\node[main node] at ({ \y* cos(\x+(3)*\z) + (1-\y)*cos(\x+(4)*\z)},{ \y* sin(\x+(3)*\z) + (1-\y)*sin(\x+(4)*\z)}) (3a) {$3$};

\node[main node] at ({ \y* cos(\x+(4)*\z) + (1-\y)*cos(\x+(3)*\z)},{ \y* sin(\x+(4)*\z) + (1-\y)*sin(\x+(3)*\z)}) (7a) {$7$};
\node[main node] at ({cos(\x+(4)*\z)},{sin(\x+(4)*\z)}) (6a) {$6$};
\node[main node] at ({ \y* cos(\x+(4)*\z) + (1-\y)*cos(\x+(5)*\z)},{ \y* sin(\x+(4)*\z) + (1-\y)*sin(\x+(5)*\z)}) (6b) {$14$};

  \path[every node/.style={font=\sffamily}]
(3a) edge (4a)
(6a) edge (7a)
(3a) edge (7a)
(4b) edge (7b)
(7b) edge (2b)
(2b) edge (5b)
(5b) edge (8b)
(8b) edge (3b)
(3b) edge (6b)
(2b) edge (6b)
(3b) edge (7b)
(4b) edge (8b)
(3a) edge (3b)
(4a) edge (4b)
(6a) edge (6b)
(7a) edge (7b)
;
\end{tikzpicture}\\
$G_{\ref{figVertexTransitive234},1} - \set{1,3} $ & $G_{\ref{figVertexTransitive234},2} \ominus 1$ & $G_{\ref{figVertexTransitive234},3}\ominus 1$ 
\end{tabular}
\caption{Notable stretch-clique induced subgraphs of $G_{\ref{figVertexTransitive234},1}$, $G_{\ref{figVertexTransitive234},2}$, and $G_{\ref{figVertexTransitive234},3}$.}
\label{figVertexTransitive234SK}
\end{figure}
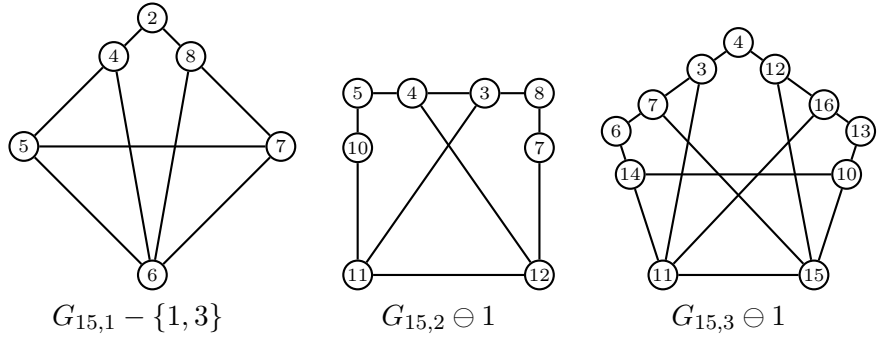

\subsection{Towards rank $5$}

We end this section by proving a lower bound on $\overline{n}_+(5)$.

\begin{proposition}\label{propVT5}
$\overline{n}_+(5) \geq 20$.
\end{proposition}

\begin{proof}
Again, we may focus on connected vertex-transitive graphs which satisfy the degree bounds from Lemma~\ref{lemnkl}. It follows from the proof of Proposition~\ref{propVT4} that no 16-vertex graph in our consideration has $\LS_+$-rank $5$, so we may assume $n \geq 17$ here. There is a total of 58 such graphs on $n \in \set{17,18,19}$ vertices~\cite{RoyleH20}, as listed in Table~\ref{tabVT5}. To prove our claim, it suffices to show that all of these graphs have $\LS_+$-rank at most 4.

\begin{table}[htbp]
\centering
\tiny{
\begin{tabular}{c c c}
\begin{tabular}{c|l}
$i$ & $G_i$\\
\hline
$1$ & \verb+Ps_?GCD@?C@UDQIKIHBO_[A?+ \\ 
$2$ & \verb+Ps_?GCDP?acOAHKcK_OqGQI?+ \\ 
$3$ & \verb+Ps_?GGAP@C@MCYRGEW@O_SAC+ \\ 
$4$ & \verb+Pts?GKE@OD?I?IGDG@QODK@O+ \\ 
$5$ & \verb+Qs??OGC@?O@?GgCcCG_aOOD@?g?+ \\ 
$6$ & \verb+Qs??WOC@?O??O`GHAGaA_CI?_W?+ \\ 
$7$ & \verb+Qs??WWG@?@?A?W?cA?h?SS@@_G?+ \\ 
$8$ & \verb+Qs?GGSG@?@?A?W?cA?h?SS@@_G?+ \\ 
$9$ & \verb+Qt?G?CG@?A_S?E?HG?I?DF??b??+ \\ 
$10$ & \verb+Qs_?????@E`[C[BKIQ@k?RO@c_?+ \\ 
$11$ & \verb+Qs_?????AHgUO[HoAw?]?MG?wO?+ \\ 
$12$ & \verb+Qs_?????AJCqA[CkIW@U?RC?s_?+ \\ 
$13$ & \verb+Qs_?????AK`MGsDKKgAd?Ig?\??+ \\ 
$14$ & \verb+Qs_??GB@`@gIDAA`EcAU?I?`O@G+ \\ 
$15$ & \verb+Qs_??KE@?C?oALCRDEAaOYC@oG?+ \\ 
$16$ & \verb+Qs_??KE@_K?K?WWEKB?oKE@`w??+ \\ 
$17$ & \verb+Qs_?G?C?aFCiAGCCdKAa_Q_OsA?+ \\ 
$18$ & \verb+Qsc???A?QCcKPgHS@a?h@I`@PG?+ \\ 
$19$ & \verb+Qsc?GGB@?C_oOXGFIGAP@IA?oOG+ \\ 
$20$ & \verb+Qsc?GKC@?D?IGYOdHA@`?IADOC_+ \\ 
\end{tabular}
&
\begin{tabular}{c|l}
$i$ & $G_i$\\
\hline
$21$ & \verb+Qts?GKE@OD?I?I?DC@Q?IQ?SoA_+ \\ 
$22$ & \verb+Qsa??????V_}E[HkJIBL?]O@u??+ \\ 
$23$ & \verb+Qsa??????mb[IkP[Hs@Y_[c@p_?+ \\ 
$24$ & \verb+Qsa?????@ed[D[BkKwBH_US@Y_?+ \\ 
$25$ & \verb+Qsa?????AR`yHkD[MQ@r?Ug@f??+ \\ 
$26$ & \verb+Qsa?????BEkUW[EwA{?^?VC?{O?+ \\ 
$27$ & \verb+Qsa??CA?OJ@RDXIeIkBS_\G@wO?+ \\ 
$28$ & \verb+Qsa??CCAATCjDSEgGiHTC[W@qO?+ \\ 
$29$ & \verb+Qsa??KCOObCrBAEcBUBDCWGpx??+ \\ 
$30$ & \verb+QsaG??A?QUGfHKDKDWOuA[g@pO?+ \\ 
$31$ & \verb+QsaG??A?XfIYHGQCak?s`XOPiA?+ \\ 
$32$ & \verb+QsaG??aCQdKIDHIEGkAS`HoOZA?+ \\ 
$33$ & \verb+QsaH?_`CaCgKQoPKch?kGIDGWgO+ \\ 
$34$ & \verb+Qseg??B?oE_uOoSAceAaSIHHw@?+ \\ 
$35$ & \verb+QsiW??A?PDaMSgWSGoQQBDd?hg?+ \\ 
$36$ & \verb+QsiW??B?x?aDWSSgHGQKBIR?UK?+ \\ 
$37$ & \verb+QsiW?D?OhCAIAW@qkAabCDBGig?+ \\ 
$38$ & \verb+QsnO??ACWSGHOjGOCMGl_IaPDD?+ \\ 
$39$ & \verb+QtaG?@@OgU_{W`SQdA?h?HOKY?o+ \\ 
$40$ & \verb+QtaG?C@?jEIFQcPWdO@I@ECGS_W+ \\ 
\end{tabular}
&
\begin{tabular}{c|l}
$i$ & $G_i$\\
\hline
$41$ & \verb+QtaIA@@O`E@T@UCLdA@H@EOGU?W+ \\ 
$42$ & \verb+QtiW?CB??B`T@rOKGIJ?`SCS}??+ \\ 
$43$ & \verb+QtiW?CB?GB?tArOSGBI_`WCS}??+ \\ 
$44$ & \verb+Q{fw?CB?wE?X?K?Kk@b?XE?oWBG+ \\ 
$45$ & \verb+Rs_??CD@_KGAGAAhAgQP_KQ?gCA_CG+ \\ 
$46$ & \verb+Rs_?GCD@?C?S@OAKaH`OcSB@gOB_O?+ \\ 
$47$ & \verb+Rs_?GGA?O@CRPEAOACHU?TG@GG`_a?+ \\ 
$48$ & \verb+Rts?GKE@OD?A?A?T?DP?DO@PGCH_C_+ \\ 
$49$ & \verb+RsaC??@?gI`TDTTTDTRTO\S@y_By??+ \\ 
$50$ & \verb+RsaC?GAHAC`dCtLIHYHrOU[@kaBeG?+ \\ 
$51$ & \verb+RsaKg??AGagiGiQZEMQZ_Ky@iCBgCG+ \\ 
$52$ & \verb+RsaKg??GYBg[HWSLDHYggIa`Y_Py?O+ \\ 
$53$ & \verb+RsaKg?@P@CDMSYCs@dJEcXJ@Wp@qP?+ \\ 
$54$ & \verb+RsaSW_GGqHgXHKK@h?wj@AwPogJQ__+ \\ 
$55$ & \verb+RsfLg?@?GAGNG]PTCTQPaKQGyGIw__+ \\ 
$56$ & \verb+Rtq}w?@?WB_M?YOTCDQOiKAhIDPgDO+ \\ 
$57$ & \verb+RujL_?@COP_[@WSTDDOdwBU`aDBGDG+ \\ 
$58$ & \verb+R}akq?D?wOC@CBSQkI_ZQAfGce?dcG+ \\ 
& \\
&
\end{tabular}
\end{tabular}
}
\caption{The list of connected vertex-transitive graphs on $n \in \set{17,18,19}$ vertices satisfying the degree bounds from Lemma~\ref{lemnkl} with $\ell = 5$.}
\label{tabVT5}
\end{table}

Observe that
\begin{itemize}
\item 
For every $i \in \set{5,6,7,10,11,12,13,15,22,23,24,25,26}$, $G_i$ is bipartite, so these graphs all have $\LS_+$-rank 0.
\item 
For every $i \in \set{1,8,16,27,46,49}$, $G_i - 1$ is bipartite, so these graphs all have $\LS_+$-rank 1.
\item 
For every $i \in \set{4,9,19,20,21,30,34,42,43,44,48,52,55,56}$, $G_i \ominus 1$ is perfect, so these graphs all have $\LS_+$-rank at most 2.
\item 
The graph $(G_{17} \ominus 1) - 8$ is bipartite. Thus, $r_+(G_{17}) \leq 2$.
\item 
The graphs
\[
\begin{array}{lll}
(G_{2} \ominus 1) - \set{6,8}, &
(G_{3} \ominus 1) - \set{8,9}, &
(G_{28} \ominus 1) - \set{7,9}, \\
(G_{29} \ominus 1) - \set{11,14}, &
(G_{31} \ominus 1) - \set{7,12}, &
(G_{32} \ominus 1) - \set{7,8}, \\
(G_{33} \ominus 1) - \set{7,9}, &
(G_{45} \ominus 1) - \set{10,11}, &
(G_{47} \ominus 1) - \set{11,12}, \\
(G_{50} \ominus 1) - \set{9,16}, &
(G_{51} \ominus 1) - \set{8,9} &
\end{array}
\]
are all bipartite. Thus, $r_+(G_i) \leq 3$ for every $i \in \set{2,3,28,29,31,32,33,45,47,50,51}$.
\item 
The graphs
\[
\begin{array}{lll}
(G_{14} \ominus 1) - \set{10,13,14}, &
(G_{18} \ominus 1) - \set{6,9,10}, &
(G_{36} \ominus 1) - \set{7,8,11}, \\
(G_{37} \ominus 1) - \set{7,12,13}, &
(G_{38} \ominus 1) - \set{8,9,11}, &
(G_{39} \ominus 1) - \set{8,10,12}, \\
(G_{40} \ominus 1) - \set{7,8,10} &
\end{array}
\]
are all bipartite. Thus, $r_+(G_i) \leq 4$ for every $i \in \set{14,18,36,37,38,39,40}$.
\item 
For every $i \in \set{35,41,53,54,57,58}$, observe that $G_i \ominus 1$ has exactly $12$ vertices, all of which have degree at least $3$. This implies that $G_i \ominus 1$ is not 4-minimal, as it follows from Theorem~\ref{thmAuT24b1} that every 4-minimal graph has at least one vertex with degree 2. Hence $r_+(G_i \ominus 1) \leq 3$, which implies that $r_+(G_i) \leq 4$.
\end{itemize}
Thus, there does not exist a vertex-transitive graph on at most 19 vertices with $\LS_+$-rank 5, which implies that $\overline{n}_+(5) \geq 20$.
\end{proof}

We remark that there are 267 connected vertex-transitive graphs which satisfy the degree bounds from Lemma~\ref{lemnkl} with $n=20$ and $\ell=5$. While the observations we used in the proof of Proposition~\ref{propVT5} can be used to show that the majority of them have $\LS_+$-rank at most 4, checking every graph in this set for a possible graph with $\LS_+$-rank 5 is still a non-trivial task.

\section{Concluding remarks and future work}\label{sec06}

We focused on $\ell$-minimal graphs in manuscript, discovering many new instances in the cases of $\ell = 3$ and $\ell = 4$. In particular, our computations provide strong evidence for Conjecture~\ref{conj3minimal}, which asserts that the 49 graphs identified in Section~\ref{sec03} exhaust the 3-minimal case. We also identified aspects of our findings which align with our pre-existing understanding of $\ell$-minimal graphs, such as the prevalence of the stretched cliques and the relevance of the clique number of a graph. The continuation of patterns found in the newly-discovered $\ell$-minimal graphs led to Conjectures~\ref{conj02}, ~\ref{conj02b}, and ~\ref{conj03}. On the other hand, we also saw many surprises, such as the discovery of many $\ell$-minimal graphs which are not stretched cliques, the most striking of which perhaps being a 4-minimal graph which does not contain $K_6$ as a graph minor ($G_{\ref{fig4minimal},12}$ from Figure~\ref{fig4minimal}). Taken together, these results sharpen the current structural picture of lift-and-project relaxations for the stable set polytope and provide a broader collection of extremal examples for future work on closely related problems. In addition, we also believe the framework of numerical certificates developed in this manuscript can be adapted for very reliably and rigorously verifying solutions in many other convex optimization problems. 

We conclude with several open problems motivated by this work.

\begin{problem}
Given $\ell \in \mathbb{N}$, what are the maximum and minimum possible edge densities of an $\ell$-minimal graph?
\end{problem}

Given $\ell \in \mN$, define 
\begin{align*}
d^-(\ell) &\ce \min \set{  \frac{ |E(G)|}{ \binom{ |V(G)|}{2} } : \tn{$G$ is an $\ell$-minimal graph}}, \\
d^+(\ell) &\ce \max \set{  \frac{ |E(G)|}{ \binom{ |V(G)|}{2} } : \tn{$G$ is an $\ell$-minimal graph}}.
\end{align*}
We first raised the problem of computing $d^-(\ell)$ and $d^+(\ell)$ in~\cite{AuT24b}, and have made some progress on this front since with Theorem~\ref{thmAuT251} and the $\ell$-minimal graphs discovered herein. Here is what we currently know about these two quantities for $\ell \leq 4$:
\[
{\def\arraystretch{1.2}
\begin{array}{r|rrrr}
\ell & 1 & 2 & 3 & 4 \\
\hline
d^-(\ell) & \frac{3}{3} & \frac{8}{15} & \frac{14}{36} & \leq \frac{21}{66} \\
d^+(\ell) & \frac{3}{3} & \frac{9}{15} & \geq \frac{19}{36} & \geq \frac{29}{66} 
\end{array}
}
\]
In particular, the new 3- and 4-minimal graphs discovered in this manuscript seem to suggest that $d^-(\ell)$ is attained by a sparse stretched clique for every $\ell \geq 1$, while $d^+(\ell)$ is likely not attained by a stretched clique for $\ell \geq 3$.

\begin{problem}
Can we characterize exactly when the vertex-stretching operation is $\LS_+$-rank increasing?
\end{problem}

Recall that the vertex-stretching operation is generally $\LS_+$-rank non-decreasing (Theorem~\ref{thmAuT250}). With Theorem~\ref{thmAuT251}, 
Proposition~\ref{propStretchedKk}, and the new $\ell$-minimal graphs discovered herein, we now know of a range of  situations where the vertex-stretching operation increases the $\LS_+$-rank of the underlying graph by one. It would be interesting to characterize exactly when the operation is $\LS_+$-rank increasing, as well as find out whether it is possible for a vertex-stretching operation to increase the rank of a graph by two or more.

\begin{problem}
Can we develop better tools for proving rank upper bounds?
\end{problem}

In this manuscript, we proposed the framework of $\LS_+^{\ell}$ certificate packages to help establish $\LS_+$-rank lower bounds. However, we currently do not know of an analogous tool for proving rank upper bounds. For instance, it would be helpful to develop theoretical and/or computational tools to show that the graphs we saw with $r_+(G) \lesssim \ell$ indeed satisfy $r_+(G) \leq \ell$. Progress in this direction will solidify our understanding of $\ell$-minimal graphs, and represents a step towards a combinatorial characterization of these graphs.

On a related topic, we can adopt the notions of $\LS_+^{\ell}$ certificate packages to the duals of the SDPs we have considered in this work. Doing so has the potential of generating more reliable certificates to conclude that $\gamma_{\ell}(G,a) =1$.

\bibliographystyle{alpha}
\bibliography{ref} 

\section*{Statements and Declarations}
\section*{Statements and Declarations}

\textbf{Data Availability.}
The data supporting the results of this study are publicly available in the Zenodo repository at \url{https://doi.org/10.5281/zenodo.15483991}. These files include the $\LS_+$, $\LS_+^2$, and $\LS_+^3$ certificate packages, the full list of 4-minimal graphs identified in the computational search, and MATLAB code used to verify the certificates.

\end{document}